\documentclass[letterpaper,11pt,oneside,reqno]{article}

\usepackage[pdftex,backref=page,colorlinks=true,linkcolor=blue,citecolor=red]{hyperref}
\usepackage[alphabetic,nobysame]{amsrefs}

\usepackage{amsmath,amssymb,amsthm,amsfonts,mathtools}
\usepackage{graphicx,color}
\usepackage{upgreek}
\usepackage[mathscr]{euscript}

\allowdisplaybreaks
\numberwithin{equation}{section}

\usepackage{tikz}
\usetikzlibrary{shapes,arrows,positioning,decorations.markings}

\usepackage{array}
\usepackage{adjustbox}
\usepackage{cleveref}
\usepackage{enumerate}
\usepackage{esint}

\usepackage[DIV=11]{typearea}

\newcommand{\ssp}{\hspace{1pt}}

\newtheorem{proposition}{Proposition}[section]
\newtheorem{lemma}[proposition]{Lemma}

\newtheorem{theorem}[proposition]{Theorem}
\newtheorem{conjecture}[proposition]{Conjecture}
\theoremstyle{definition}
\newtheorem{definition}[proposition]{Definition}
\newtheorem{remark}[proposition]{Remark}

\title{Tilted biorthogonal ensembles, Grothendieck
random partitions, and determinantal tests}
\author{Svetlana Gavrilova and Leonid Petrov}
\date{}

\begin{document}

\maketitle

\begin{abstract}
	We study probability measures on partitions based on symmetric Grothendieck polynomials. These deformations of Schur polynomials introduced in the K-theory of Grassmannians share many common properties. Our Grothendieck measures are analogs of the Schur measures on partitions introduced by Okounkov \cite{okounkov2001infinite}. Despite the similarity of determinantal formulas for the probability weights of Schur and Grothendieck measures, we demonstrate that Grothendieck measures are \emph{not} determinantal point processes. This question is related to the principal minor assignment problem in algebraic geometry, and we employ a determinantal test first obtained by Nanson in 1897 for the $4\times4$ problem. We also propose a procedure for getting Nanson-like determinantal tests for matrices of any size $n\ge4$, which appear new for $n\ge 5$.

	By placing the Grothendieck measures into a new framework of tilted biorthogonal ensembles generalizing a rich class of determinantal processes introduced by Borodin \cite{Borodin1998b}, we identify Grothendieck random partitions as a cross-section of a Schur process, a determinantal process in two dimensions. This identification expresses the correlation functions of Grothendieck measures through sums of Fredholm determinants, which are not immediately suitable for asymptotic analysis. A more direct approach allows us to obtain a limit shape result for the Grothendieck random partitions. The limit shape curve is not particularly explicit as it arises as a cross-section of the limit shape surface for the Schur process. The gradient of this surface is expressed through the argument of a complex root of a cubic equation.
\end{abstract}

\section{Introduction}
\label{sec:intro}

\subsection{Random partitions from symmetric functions}

The study of random integer partitions involving probability weights expressed through symmetric polynomials has been a long-standing topic in integrable probability and related fields \cite{BorodinGorinSPB12}, \cite{BorodinPetrov2013Lect}. Asymptotic analysis of various measures on partitions produced law of large numbers and asymptotic fluctuation results in many stochastic models describing complex real-world phenomena, including longest increasing subsequences \cite{logan_shepp1977variational}, \cite{VershikKerov_LimShape1077}, \cite{baik1999distribution}, interacting particle systems \cite{Johansson2000}, random growth models \cite{BorFerr2008DF}, random polymer models \cite{Oconnell2009_Toda}, \cite{COSZ2011}, \cite{OSZ2012}, random matrices \cite{BorodinGorin2013beta}, and geometry \cite{okounkov2000random}, \cite{Okounkov2006uses}.

One of the earliest studied ensembles of random partitions based on symmetric functions is the \emph{Schur measure} introduced in \cite{okounkov2001infinite}. The Schur measure probability weights have the form
\begin{equation}
	\label{eq:Schur_measure_def}
	\mathop{\mathrm{Prob}}(\lambda)\coloneqq
	\frac{1}{Z}
	\underbrace{\frac{
	\det[x_j^{\lambda_i+N-i}]_{i,j=1}^{N}}
	{\prod_{1\le i<j\le N}(x_i-x_j)}}_{s_\lambda(x_1,\ldots,x_N )}
	\underbrace{\frac{
	\det[y_j^{\lambda_i+N-i}]_{i,j=1}^{N}}
	{\prod_{1\le i<j\le N}(y_i-y_j)}}_{s_\lambda(y_1,\ldots,y_N )}.
\end{equation}
Here $\lambda=(\lambda_1\ge \ldots\ge \lambda_N\ge0 )$ are integer partitions which we think of as our random objects, $x_i,y_j\ge 0$ with $x_iy_j<1$ are parameters of the measure. The quantities $s_\lambda(x_1,\ldots,x_N )$ and $s_\lambda(y_1,\ldots,y_N )$ in \eqref{eq:Schur_measure_def} are the well-known \emph{Schur symmetric polynomials} in the variables $x_1,\ldots,x_N $ and $y_1,\ldots,y_N $, respectively, indexed by the same partition $\lambda$. The probability normalizing constant $Z=\prod_{i,j=1}^{N}(1-x_iy_j)^{-1}$ has a product form thanks to the Cauchy summation identity for Schur polynomials.

The Schur measures are particularly tractable thanks to their determinantal structure, which allows expressing correlation functions 
\begin{equation}
	\label{eq:Schur_corr_f}
	\rho(a_1,\ldots,a_m )\coloneqq
	\mathop{\mathrm{Prob}}\left( \textnormal{the random set $\left\{ \lambda_i+N-i \right\}\subset \mathbb{Z}_{\ge0}$ contains each $a_1,\ldots,a_m $} \right) 
\end{equation}
of an arbitrary order $m$ as $m\times m$ determinants $\det[K(a_i,a_j)]_{i,j=1}^m$ of a fixed correlation kernel $K(a,b)$, where $a,b\in \mathbb{Z}_{\ge0}$. The kernel has a double contour integral form, readily amenable to asymptotic analysis by steepest descent.

Over the past two decades, Schur measures have been generalized to other families of symmetric polynomials, including Macdonald polynomials \cite{BorodinCorwin2011Macdonald} and their degenerations such as Jack \cite{BorodinGorin2013beta}, Hall-Littlewood \cite{BufetovPetrov2014}, \cite{BufetovMatveev2017}, and $q$-Whittaker polynomials \cite{BorodinPetrov2013Lect}, \cite{MatveevPetrov2014}. More recently, these efforts have extended to symmetric rational functions (like spin $q$-Whittaker and spin Hall-Littlewood functions) arising as partition functions of integrable (in the sense of the Yang-Baxter equation) vertex models \cite{Borodin2014vertex}, \cite{BorodinPetrov2016_Hom_Lectures}, \cite{BorodinWheelerSpinq}, \cite{ABPW2021free}. The vertex model approach also naturally included distinguished nonsymmetric polynomials powering the structure of multispecies stochastic systems \cite{borodin_wheeler2018coloured}, \cite{agg-bor-wh2020-sl1n}. 

While these more general symmetric polynomials and rational functions share many properties with the Schur polynomials, the technique of determinantal point processes does not straightforwardly extend. This has led to several interesting alternative approaches, including eigenoperators \cite{BorodinCorwin2011Macdonald} and duality \cite{BorodinCorwinSasamoto2012}, which brought multiple contour integral formulas for expectations of observables. Recently \cite{Imamura2021skewRSK} presented a direct mapping between $q$-Whittaker and cylindric Schur measures \cite{borodin2007periodic} preserving specific observables. Since the latter measures are determinantal, this allows for employing determinantal process methods for the asymptotic analysis of these observables.

\subsection{Grothendieck measures on partitions}

Our primary focus is on \emph{Grothendieck measures on partitions} whose probability weights are expressed through the Grothendieck symmetric polynomials:
\begin{equation}
	\label{eq:Grothendieck_measure_definition_complete_intro}
	\mathop{\mathrm{Prob}}(\lambda)
	\coloneqq
	\frac1{Z'}	
	\underbrace{
	\frac{\det\bigl[ 
		x_i^{\lambda_j+N-j}(1-\beta x_i)^{j-1}\bigr]_{i,j=1}^{N}
	}
	{\prod_{1\le i<j\le n}(x_i-x_j)}
	}_{G_\lambda(x_1,\ldots,x_N )}
	\underbrace
	{\frac{
	\det\bigl[ 
	y_i^{\lambda_j+N-j}(1-\beta y_i^{-1})^{N-j}\bigr]_{i,j=1}^{N}
	}{\prod_{1\le i<j\le n}(y_i-y_j)}}_
	{\overline{G}_\lambda(y_1,\ldots,y_N )}.
\end{equation}
Here $x_i,y_j$, and $\beta$ are parameters
such that $x_i,y_j\ge0$, $x_iy_j<1$, and $\beta\le \min_{1\le i\le N}(x_i^{-1},y_i)$.
The latter condition implies that the probability
weights are nonnegative. The Grothendieck symmetric polynomials $G_\lambda(x_1,\ldots,x_N )$ and $\overline{G}_\lambda(y_1,\ldots,y_N )$ are one-parameter deformations of the Schur polynomials appearing in the K-theory of Grassmannians.
The normalizing constant is 
$
Z'=
\prod_{i=1}^{N}(1-x_i\beta)^{N-1}
\prod_{i,j=1}^N(1-x_iy_j)^{-1}
$.
When $\beta=0$, the Grothendieck measure 
\eqref{eq:Grothendieck_measure_definition_complete_intro} reduces to the
Schur measure \eqref{eq:Schur_measure_def}.
We refer to
\cite{lascoux1982structure},
\cite{buch2002littlewood},
\cite{fomin1994grothendieck},
\cite{yeliussizov2015duality}, 
\cite{chan2021combinatorial},
and 
\cite{hwang2021refined}
for details, properties, and various 
multiparameter generalizations of Grothendieck polynomials.
All methods of the present paper apply in a setting when there are multiple $\beta_j$'s (see the polynomials $G_\lambda$ and $\overline{G}_\lambda$ in \eqref{eq:Grothendieck_polynomials} in the text). However, in the Introduction and asymptotic analysis, we restrict to the case of the homogeneous $\beta_j$'s.

\medskip

In this paper, we obtain two main results for the Grothendieck measures:
\begin{enumerate}[$\bullet$]
	\item We show that despite the determinant representation of their probability weights, Grothendieck measures do not possess a determinantal structure of correlations. This observation may appear unexpected given the similarity of Grothendieck probability weights compared to the Schur measures.
	\item We establish a link between Grothendieck random partitions and Schur processes, the latter being determinantal point processes on a two-dimensional lattice. We perform this connection within an extended framework of tilted biorthogonal ensembles, which we introduce. This connection provides an essential structure for the Grothendieck measures. It enables us to derive formulas expressing their correlations through sums of Fredholm determinants and prove limit shape results.	
\end{enumerate}
We formulate these results in the remainder of the Introduction.

\begin{remark}
	\label{rmk:degenerations_of_sHL}
	It was observed in \cite[Sections 8.3 and 8.4]{Borodin2014vertex} that
	the $q=0$ specialization of spin Hall-Littlewood polynomials 
	produces determinantal partition functions of vertex models
	which resemble the Grothendieck polynomials $G_\lambda$ 
	in \eqref{eq:Grothendieck_measure_definition_complete_intro}. 
	Most of the machinery for computing expectations of observables
	of the form $q^{\mathrm{height}}$ breaks down for $q=0$, so it is not 
	immediately clear whether vertex models are applicable in the analysis
	of Grothendieck measures.
	Moreover,
	limit shape results are not yet established for random
	partitions with spin Hall-Littlewood weights or their $q=0$ degenerations
	(see, however, \cite{Ahn2020} for limit shapes of Macdonald random partitions in another regime, as $q,t\to1$).
\end{remark}

\subsection{Absence of determinantal structure}

\begin{theorem}
	\label{thm:non_determinantal_intro}
	For certain fixed $N$ and values of parameters $x_i,y_j$, and $\beta$, the correlations \eqref{eq:Schur_corr_f} of the Grothendieck measures do not possess a determinantal form. That is, there does not exist a function $K\colon \mathbb{Z}_{\ge0}^{2}\to \mathbb{C}$ for which $\rho(a_1,\ldots,a_m )=\det[K(a_i,a_j)]_{i,j=1}^m$ for all $m$ and all pairwise distinct $a_1,\ldots,a_m\in\mathbb{Z}_{\ge0}$.
\end{theorem}

We show the nonexistence of a correlation kernel $K$ by constructing an explicit polynomial in the correlation functions $\rho(a_1,\ldots,a_m )$, which vanishes identically if the correlation functions have a determinantal form (we call such polynomials \emph{determinantal tests}).  We then show that for a specific choice of parameters, $N=2$, $x_i=y_j=1/2$, $\beta=-1$, the determinantal test does not vanish. While for \Cref{thm:non_determinantal_intro}, we only need a specific choice of parameters, we expect the absence of determinantal structure to hold for generic parameters in the Grothendieck measures.

\medskip

The problem of finding a kernel representing all
correlations $\rho(a_1,\ldots,a_m )$ in a determinantal form
is the same as the well-known \emph{principal minor
assignment problem} in algebraic geometry. This problem
seeks an $n\times n$ matrix whose all principal (diagonal)
minors are given, but such an underlying matrix does not
exist for all choices of (prospective) principal minors.
Therefore, one has to find relations between principal
minors. These relations are polynomial, and each may be used
as a determinantal test. The variety of $n\times n$
principal minors becomes complicated already for $n=4$ (it
is minimally generated by 65 polynomials of degree 12), but
for \Cref{thm:non_determinantal_intro}, it suffices to show
that one generating polynomial does not vanish. In fact, the
determinantal test we employ in our proof was written down
by Nanson in 1897 for $4\times 4$ matrices
\cite{nanson1897xlvi}. In \Cref{sub:history}, we discuss the
rich history of the principal minor assignment problem and
several instances of its rediscovery within the study of
determinantal point processes. In
\Cref{sub:Nanson_det_test_derivation,sub:higher_order_Nanson},
we present a self-contained derivation of the Nanson's
determinantal test and suggest a generalization of the
Nanson's test to matrices of arbitrary size. This
generalization appears new.

\subsection{Tilted biorthogonal ensembles}
\label{sub:tilted_intro}

To connect Grothendieck measures to Schur processes, which are determinantal processes on the two-dimensional integer lattice, we consider a more general framework of \emph{tilted biorthogonal ensembles}, which is inspired by a talk of Kenyon \cite{Kenyon_IPAM_talk}.  The ordinary biorthogonal ensembles introduced in \cite{Borodin1998b} are measures on partitions with probability weights of the form
\begin{equation}
	\label{eq:biorthogonal_ensembles_intro}
	\mathop{\mathrm{Prob}}(\lambda)=\frac{1}{Z}
	\det\bigl[ \Phi_i(\ell_j) \bigr]_{i,j=1}^{N}
	\det\bigl[ \Psi_i(\ell_j) \bigr]_{i,j=1}^{N},
	\qquad \ell_j\coloneqq \lambda_j+N-j,
\end{equation}
where $\Phi_i,\Psi_j$ are given functions, and $Z$ is the normalizing constant.  Biorthogonal ensembles are determinantal processes on $\mathbb{Z}_{\ge0}$ in the same sense as the Schur measures. Moreover, when $\Phi_i(k)=x_i^k$, $\Psi_j(k)=y_j^k$, the weights \eqref{eq:biorthogonal_ensembles_intro} coincide with \eqref{eq:Schur_measure_def}.

We ``tilt'' the biorthogonal ensemble
\eqref{eq:biorthogonal_ensembles_intro} by inserting
$j$-dependent difference operators into the determinants.\footnote{Recall that in the Introduction, we only deal with the homogeneous beta parameters
$\beta_j\equiv \beta$, see \Cref{sub:def_tilted_biorthogonal}
below for the general case.}
When $\Phi_i(k)=x_i^k$, $\Psi_j(k)=y_j^k$, the action of
these operators results in the factors $(1-\beta x_i)^{j-1}$
and $(1-\beta y_i^{-1})^{N-j}$ in
\eqref{eq:Grothendieck_measure_definition_complete_intro}.
In general, we apply the operator $(D)^{j-1}$ to
$\Phi_i(\ell_j)$, where $Df(k)=f(k)-\beta f(k+1)$, and
$(D^{\dagger})^{N-j}$ to $\Psi_i(\ell_j)$, where
$D^{\dagger}f(k)=f(k)-\beta f(k-1)\mathbf{1}_{k\ge1}$.  Here
and throughout the paper, $\mathbf{1}_{A}$ stands for the
indicator of an event or a condition $A$. We arrive at the
following measure on partitions:
\begin{equation}
	\label{eq:biorthogonal_ensembles_tilted_intro}
	\mathop{\mathrm{Prob}}(\lambda)=\frac{1}{Z'}
	\det\bigl[ (D)^{j-1}\Phi_i(\ell_j) \bigr]_{i,j=1}^{N}
	\det\bigl[ (D^{\dagger})^{N-j}\Psi_i(\ell_j) \bigr]_{i,j=1}^{N},
	\qquad 
	\ell_j= \lambda_j+N-j.
\end{equation}
For details, we refer to \Cref{sub:def_tilted_biorthogonal} in the text.

The action of $D$ is the same as the multiplication by the
matrix $T_\beta(k,l)\coloneqq \mathbf{1}_{l=k}-\beta \ssp
\mathbf{1}_{l=k-1}$, and $D^{\dagger}$ is the multiplication
by $T_\beta$ on the opposite side.  Using this, we identify
(\Cref{thm:properties_of_W_2d}) the joint distribution of
$(\ell_1>\ldots>\ell_N )$ under the tilted biorthogonal
ensemble with that of the points $(x_1^1>\ldots>x^N_N )$ in
the two-dimensional ensemble $\{x^m_j\colon 1\le m,j\le N \}$
which has probability weights proportional to 
\begin{equation}
	\label{eq:2d_process_intro}
	\det\left[ \Phi_i(x^1_j) \right]_{i,j=1}^N
	\biggl(\ssp
		\prod_{m=1}^{N-1}
		\det\left[ T_{\beta}(x^m_i,x^{m+1}_j) \right]_{i,j=1}^{N}
	\biggr)
	\det\left[ \Psi_i(x^{N}_j) \right]_{i,j=1}^N.
\end{equation}
The two-dimensional process has probability weights given by products of determinants. Thus, it is determinantal thanks to the well-known Eynard--Mehta theorem \cite{eynard1998matrices}, \cite{borodin2005eynard}, see also \cite[Theorem 4.2]{borodin2009}. 

The above identification allows us to write down certain Fredholm determinantal formulas for marginal distributions and correlation functions of tilted biorthogonal ensembles; see \Cref{sub:correlations_ensemble} and \Cref{prop:correlations_tilted} in particular.

When $\Phi_i(k)=x_i^k$ and $\Psi_j(k)=y_j^k$ for all $i,j$,
the two-dimensional determinantal process
\eqref{eq:2d_process_intro} becomes the Schur process whose
correlation kernel has a double contour integral form
\cite{okounkov2003correlation}. The particular
specializations of the Schur process parameters are given in
\Cref{sub:2d_Schur_process} in the text. Our Schur process
has nonnegative probability weights only for $\beta<0$, and
this is the case we restrict to in our asymptotic analysis
(see \Cref{sub:asymp_intro} below). The case $\beta=0$ is
covered by standard results on Schur measures. 
It is plausible that our results on the Grothendieck limit shape still apply to values of $\beta>0$, even if probabilities in the two-dimensional process are negative, as long as the Grothendieck probability weights \eqref{eq:Grothendieck_measure_definition_complete_intro} remain nonnegative. See \Cref{conj:positive_beta} for details.

\begin{remark}[Application to the five-vertex model]
	\label{rmk:5v}
	In \cite{Kenyon_IPAM_talk}, Kenyon expressed certain
	distributions arising in the five-vertex model (see also
	\cite{deGierKenyon2021limit}) as tilted biorthogonal
	ensembles.  It would be very interesting to apply our
	results to the asymptotic analysis of the five-vertex
	model, but there are three clear obstacles.
	First, the two-dimensional process for the five-vertex
	model is not the Schur process but rather a multiparameter
	analog of the more complicated model of lozenge tilings of
	the hexagon (see, e.g., \cite{Gorin2007Hexagon},
	\cite{Petrov2012} for the determinantal structure of the
	original tilings of the hexagon). One does not have as
	elegant expressions for the correlation kernel in the case
	of multiple parameters. 
	Second, the probability weights in
	the two-dimensional process are complex-valued.  This makes
	probabilistic identification of limit shapes problematic;
	see also the discussion in \Cref{subsub:positive_beta}
	below.  Third, for the five-vertex model, the multiple
	parameters $x_i,y_j$ are solutions to the Bethe equations.
	This makes a potential asymptotic analysis even more
	intricate (see, however, \cite{Priezzhev2003} and
	\cite{BaikLiuTASEP-CPAM} for a related analysis of TASEP on
	the ring).
\end{remark}

\subsection{Limit shape}
\label{sub:asymp_intro}

Consider Grothendieck random partitions \eqref{eq:Grothendieck_measure_definition_complete_intro}
with homogeneous parameters $x_i\equiv x>0$, $y_j\equiv y>0$, such that $xy<1$ and $\beta<0$.
Let us draw Young diagrams of our Grothendieck random
partitions in the $(u,v)$ coordinate system rotated by~$45^\circ$, see \Cref{fig:YD_45_and_limsh}, left.
Each partition is encoded by a 
piecewise linear function $v=\mathfrak{W}_N(u)$
with derivatives $\pm1$
and integer maxima and minima. 
Since our partitions have at most $N$ parts, 
we almost surely have 
$\mathfrak{W}_N(u)\ge |u|$ for all $u$, 
$\mathfrak{W}_N(u)= |u|$ if $|u|$ is large enough,
and $\mathfrak{W}_N(u)\le u+2N$ if $u\ge -N$.

\begin{theorem}
	\label{thm:G_limsh}
	Fix the parameters $x,y,\beta$ as above. There exists a 
	continuous, piecewise differentiable, 1-Lipschitz function $\mathfrak{W}(u)=
	\mathfrak{W}(u\mid x,y,\beta)$ with
	$\mathfrak{W}(u)\ge |u|$ and $\mathfrak{W}(u)=|u|$ if $|u|$ is large enough, such that
	\begin{equation*}
		\lim_{N\to+\infty}\frac{\mathfrak{W}_N(u N)}{N}=\mathfrak{W}(u),\qquad u\in \mathbb{R},
	\end{equation*}
	where the convergence is pointwise in probability.
	See \Cref{fig:YD_45_and_limsh}, right, for an illustration.
\end{theorem}

\begin{figure}[htpb]
	\centering
	\includegraphics[width=.45\textwidth]{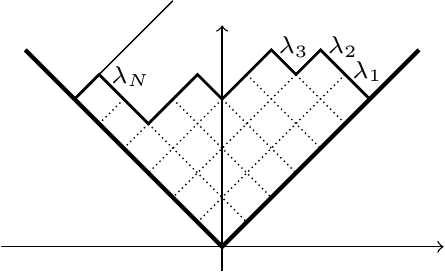}
	\qquad 
	\includegraphics[width=.45\textwidth]{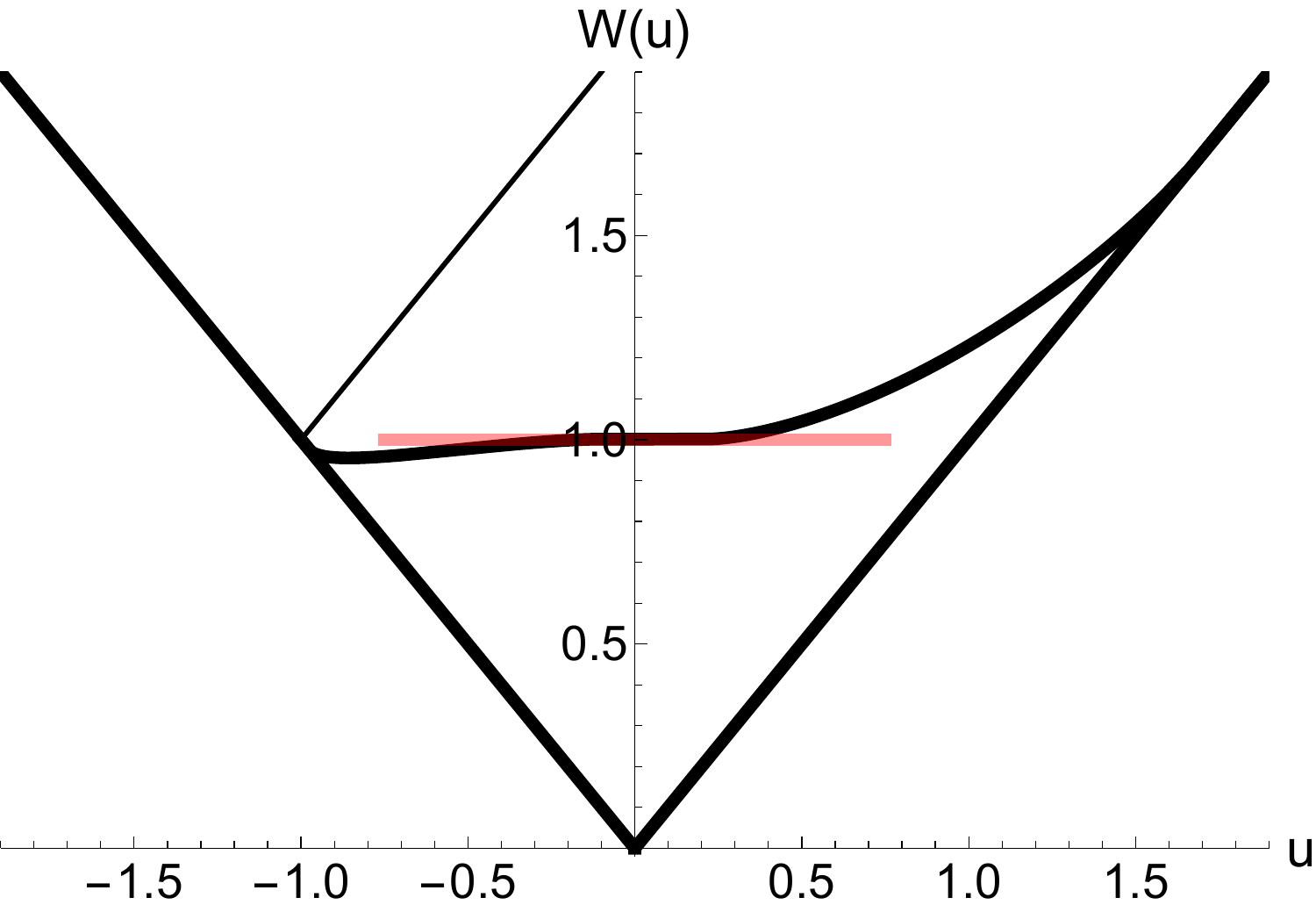}
	\caption{Left: The Young diagram for $\lambda=(6,6,5,3,1,1)$
	in the coordinate system rotated by~$45^\circ$. 
	The diagonal line $v=u+2N$ represents the upper boundary of the shape of $\lambda$.
	Right: 
	An example of a limit shape $\mathfrak{W}(u)$ of the Grothendieck random partition 
	for $x=1/3$, $y=1/5$, and $\beta=-25$. 
	We added a horizontal line to highlight the staircase frozen facet where the
	limit shape $\mathfrak{W}(u)$ is horizontal. An exact sample of a random partition corresponding to the 
	limit shape on the right is given in \Cref{fig:samples}, right (see also \Cref{sub:sampling} for a discussion of how to 
	sample Grothendieck random partitions).}
	\label{fig:YD_45_and_limsh}
\end{figure}

The first limit shape result for random partitions
(with Plancherel measure, which is a particular case of Schur measures)
was obtained by 
Logan--Shepp \cite{logan_shepp1977variational} and
Vershik--Kerov
\cite{VershikKerov_LimShape1077}.
We do not have an analytic formula for our shapes $\mathfrak{W}(u)$
in contrast to this classical VKLS shape.
Let us briefly describe how $\mathfrak{W}(u)$ is related to the 
limit surface of the Schur process. We used this connection to 
numerically plot all our examples; see \Cref{sub:simulations_and_particular_cases} for details and more discussion.

Let 
$\{x^m_j\colon 1\le m,j\le N \}$ be distributed according to the Schur process
as in \eqref{eq:2d_process_intro}. Define the 
height function
$H_N(a,t)\coloneqq \#\{j\colon x^t_j\ge a \}$, where $(a,t)\in \mathbb{Z}_{\ge0}\times\left\{ 1,\ldots,N  \right\}$.
Using the standard steepest descent analysis of the correlation kernel 
of the Schur process (dating back to \cite{okounkov2003correlation},
see also \cite[Section 3]{Okounkov2002}), one can show that $H_N$ 
has a limit shape
$\mathfrak{H}(\xi,\tau)
=
\lim\limits_{N\to\infty}N^{-1}\ssp H_N(\lfloor \xi N \rfloor ,\lfloor \tau N \rfloor )$,
where $(\xi,\tau)\in \mathbb{R}_{\ge0} \times[0,1]$.
The gradient of $\mathfrak{H}$ is expressed through 
arguments of the complex root $z_c=z_c(\xi,\tau)$
of a certain cubic equation depending on $(\xi,\tau)$ and our parameters $(x,y,\beta)$, see
\eqref{eq:cubic_equation_z}
and 
\eqref{eq:gradient_via_angles}
for the formulas.

The identification between Grothendieck random partitions and the slice 
$(x_1^1>\ldots>x^N_N )$ of the Schur process (see \Cref{sub:tilted_intro} above)
helps to express the Grothendieck limit shape $\mathfrak{W}(u)$ through $\mathfrak{H}(\xi,\tau)$.
Namely, let 
$\mathfrak{L}(\tau)$ be an auxiliary function defined from the implicit equation
\begin{equation}
	\label{eq:intro_implicit_eq}
	\mathfrak{H}
	\left( 
	\mathfrak{L}(\tau),\tau
	\right)
	=\tau \quad \textnormal{for all }\tau\in[0,1].
\end{equation}
In other words, the three-dimensional parametric curve $(\mathfrak{L}(\tau),\tau,\tau)$ is the cross-section of the
Schur process limit shape surface $\eta=\mathfrak{H}(\xi,\tau)$
in the $(\xi,\tau,\eta)$ coordinates by the plane $\eta=\tau$.
From the Schur process limit shape result,
we have $\mathfrak{L}(\tau)=1-\tau+\lim\limits_{N\to\infty}N^{-1}\lambda_{\lfloor N\tau \rfloor }$
(the shift by $1-\tau$ comes from $\ell_j=\lambda_j+N-j$, see \eqref{eq:biorthogonal_ensembles_tilted_intro}).
Then the Grothendieck limit shape curve $(u,\mathfrak{W}(u))$ 
as in \Cref{fig:YD_45_and_limsh}, right, has the following parametrization
through $\mathfrak{L}(\tau)$:
\begin{equation*}
	u=\mathfrak{L}(\tau)-1
	,\qquad 
	\mathfrak{W}=\mathfrak{L}(\tau)-1+2\tau.
\end{equation*}

The functions $\mathfrak{L}(\tau)$ and $\mathfrak{W}(u)$
satisfy differential equations 
involving the root $z_c(\xi,\tau)$,
see \eqref{eq:diff_eq_for_L}--\eqref{eq:diff_eq_for_W}.
However, the implicit equation
\eqref{eq:intro_implicit_eq} turns out to be more convenient for plotting the shapes.

The flat, ``frozen'', facets of the Schur limit shape surface 
(where the gradient is at a vertex of its allowed triangle,
see \eqref{eq:gradient_triangle_condition})
lead to the three possible flat facets of $\mathfrak{W}(u)$, where
$\mathfrak{W}'(u)$
is equal to $-1,0$, or $1$, respectively. The derivatives $\pm1$
occur when 
$\mathfrak{W}(u)=|u|$ outside of the curved part of the limit shape.
The facet $\mathfrak{W}'(u)=0$ always
arises for sufficiently negative~$\beta$ (\Cref{lemma:large_negative_beta_always_zone_II}).
In this facet, the random partition develops the deterministic
\emph{frozen staircase} behavior, that is, $\lambda_i=\lambda_{i+1}+1$
for all $i$ in some interval of order~$N$.
See
the horizontal part of the limit shape in \Cref{fig:YD_45_and_limsh}, right.

\medskip

Besides limit shapes, the study of random partitions often involves fluctuations in various regimes (at the edge, in the bulk, and global Gaussian fluctuations). It would be interesting to obtain fluctuation results for Grothendieck random partitions in these regimes and compare them to the classical case of the Plancherel random partitions \cite{baik1999distribution}, \cite{Johansson1999}, \cite{Borodin2000b}, \cite{ivanov2002kerov}. The tilted nature of the cross-section leading to Grothendieck measures seems to be affecting all Grothendieck fluctuations except the edge ones. Indeed, for any fixed $k$, $(\lambda_1,\ldots,\lambda_k)$ have the same joint distribution as $(\mu^1_1,\ldots,\mu^k_k )$, where the partitions $\mu^1,\ldots,\mu^k $ for a Schur process. Moreover, we have $|\mu^j_j-\mu^1_j|\le j$ for all $j$ (see \Cref{subsub:Plancherel_reduction} for details). Therefore, we expect that the joint distribution of $(\lambda_j-cN)/(\sigma N^{1/3})$, $j=1,\ldots,k$, should converge to the Airy$_2$ point process, just like for the Plancherel measure. We also expect that the bulk fluctuations are not given by the same discrete sine process as in the Plancherel case. It would be interesting to compute the correlations of the Gaussian limit, and compare them to the Plancherel case.

\subsection{Outline}

In \Cref{sec:tilted_biorthogonal}, we introduce the framework of tilted biorthogonal ensembles and show that they are cross-sections of two-dimensional determinantal processes. The correlation kernel of the latter is given by the Eynard--Mehta theorem. In \Cref{sec:Grothendieck_measures_definition_discussion}, we specialize tilted biorthogonal ensembles to Grothendieck measures on partitions and write down the correlation kernel of the corresponding two-dimensional Schur process in a double contour integral form (specializing the results of \cite{okounkov2003correlation}). In \Cref{sec:non_determinantal}, we prove \Cref{thm:non_determinantal_intro} that Grothendieck measures are not determinantal point processes. \Cref{sub:history} provides a brief historical account of the relation between the determinantal structure of probability measures and the principal minor assignment problem. Finally, in \Cref{sec:limit_shape}, we establish limit shape results for Schur processes and Grothendieck random partitions and illustrate these results by several plots and exact sampling simulations.

\subsection{Acknowledgments}

We are grateful to Richard Kenyon, Grigori Olshanski, Bernd Sturmfels, Alexander Povolotsky, and Damir Yeliussizov for helpful discussions at various stages of the project. LP is grateful to Promit Ghosal for bringing attention to Grothendieck measures in 2018. We appreciate the comments from the anonymous reviewer, which have led to the discussion of sampling of Grothendieck random partitions in \Cref{sub:sampling}.

The first author is partially supported by International Laboratory of Cluster Geometry, NRU HSE, 
ag. number 075-15-2021-608.
The work of the second author was partially supported by the
NSF grants DMS-1664617 and DMS-2153869 as well as by
DMS-1928930 (in connection with the program ``Universality
and Integrability in Random Matrix Theory and Interacting
Particle Systems'' hosted by the Mathematical Sciences
Research Institute in Berkeley, California, during the Fall
2021 semester) and by the Simons Collaboration Grant for
Mathematicians 709055.

\section{Tilted biorthogonal ensembles}
\label{sec:tilted_biorthogonal}

In this section, we present the main framework 
for measures 
on particle configurations in $\mathbb{Z}_{\ge0}$
given by a certain product of determinants,
and discuss their characteristics.

\subsection{Definition of the ensemble}
\label{sub:def_tilted_biorthogonal}

Fix $N$, and let $\Phi_k$, $\Psi_k$, $k=1,\ldots,N $,
be arbitrary complex-valued functions on $\mathbb{Z}_{\ge0}$.
Fix additional complex parameters $\beta_1,\beta_2,\ldots,\beta_{N-1}$. 
Let us define the following operators acting on finitely supported functions on $\mathbb{Z}_{\ge0}$:
\begin{equation} 
	\label{eq:def_operators_D}
	D_k^{(r)}f(k)\coloneqq f(k)-\beta_r f(k+1),\qquad 
	D_k^{(r)\dagger}f(k)\coloneqq f(k)-\beta_r f(k-1)\ssp \mathbf{1}_{k\ge 1},
\end{equation}
where $r=1,\ldots,N-1 $.
These operators are conjugate
to each other with respect to the bilinear form 
$\sum_{k=0}^{\infty}f(k)g(k)$ on finitely supported functions on $\mathbb{Z}_{\ge0}$.
Denote 
\begin{equation}
	\label{eq:D_chains}
	D_\ell^{[a,b)}\coloneqq D_{\ell}^{(a )}D_{\ell}^{(a+1 )}\ldots 
	D_{\ell}^{( b-1 )},
\end{equation}
and similarly for 
other types of segments and the conjugate operators
$D_\ell^{[a,b)\dagger}$. 
Clearly, $D_{\ell}^{[1,1)}$ is the identity
operator.

Assign the following weights to $N$-point configurations
on $\mathbb{Z}_{\ge0}$:
\begin{equation}
	\label{eq:tilted_biorthogonal}
	\mathscr{W}_{\vec \beta}(X)\coloneqq
	\det\bigl[ D_{\ell_j}^{[1,j)}\Phi_i(\ell_j) \bigr]_{i,j=1}^{N}
	\det\bigl[ D_{\ell_j}^{[j,N)\dagger}\Psi_i(\ell_j) \bigr]_{i,j=1}^{N},
\end{equation}
where 
$X=(\ell_1> \ell_2> \ldots> \ell_N\ge0 )$.
If the number of points in $X$ is not $N$, then set $\mathscr{W}_{\vec\beta}(X)=0$.
Assume that the series for the partition function for the weights
\eqref{eq:tilted_biorthogonal}, 
\begin{equation*}
	\mathscr{Z}_{\vec \beta}\coloneqq \sum_{X=(\ell_1> \ell_2> \ldots> \ell_N\ge0 )} \mathscr{W}_{\vec \beta}(X)
\end{equation*}
converges and is nonzero.\footnote{Throughout this section (which discusses abstract ensembles) 
we assume that all similar infinite series converge.}

\begin{definition}
	\label{def:tilted_biorthogonal}
	The normalized weights 
	\begin{equation}
		\label{eq:M_normalized_1d_weights}
		\mathscr{M}_{\vec\beta}(X)
		=
		\mathscr{W}_{\vec \beta}(X)/\mathscr{Z}_{\vec \beta}
	\end{equation}
	define a probability
	measure on $N$-particle configurations on $\mathbb{Z}_{\ge0}$.
	We call this measure the \emph{$\vec \beta$-tilted $N$-point biorthogonal ensemble}.
\end{definition}

The term ``probability measure'' here refers to the fact that the sum of the normalized weights is equal to $1$. The weights are generally complex-valued but become nonnegative real numbers in the specializations we discuss later.

When $\beta_j\equiv 0$, the operators \eqref{eq:def_operators_D} become identity operators, and the tilted biorthogonal ensemble turns into the usual biorthogonal ensemble with probability weights proportional to
\begin{equation}
	\label{eq:biorthogonal_ensembles_classical}
	\det\bigl[ \Phi_i(\ell_j) \bigr]_{i,j=1}^{N}
	\det\bigl[ \Psi_i(\ell_j) \bigr]_{i,j=1}^{N}.
\end{equation}
Biorthogonal ensembles were introduced and studied
in \cite{Borodin1998b}, see also \cite[Section 4]{borodin2009} for a summary of 
formulas.

\subsection{Normalization}
\label{sub:normalization}

Let us compute the normalizing constant $\mathscr{Z}_{\vec\beta}$:

\begin{proposition}
	\label{prop:partition_function_tilted}
	We have 
	\begin{equation}
	\label{eq:tilted_Gram_matrix_partition_function}
		\mathscr{Z}_{\vec \beta}=\det\bigl[ G_{ij}(\vec\beta  ) \bigr]_{i,j=1}^{N},
	\end{equation}
	where
	\begin{equation}
		\label{eq:tilted_Gram_matrix_definition}
		G_{ij}(\vec \beta )\coloneqq
		\sum_{k=0}^\infty
		\Psi_j(k)\ssp
		D_k^{[1,N)}\Phi_i(k).
	\end{equation}
\end{proposition}
\begin{proof}
	Observe that for any $0\le a\le b$, we have the following summation by parts:
	\begin{equation}
		\label{eq:boundary_terms_summation_by_parts}
		\sum_{k=a}^b f(k)\ssp D^{(r)\dagger}_k g(k)-
		\sum_{k=a}^b g(k)\ssp D^{(r)}_kf(k) =
		\beta_r g(b)\ssp f(b+1) - \beta_r f(a)\ssp g(a-1)\ssp \mathbf{1}_{a\ge 1}.
	\end{equation}
	We have
	\begin{equation*}
		\begin{split}
			\mathscr{Z}_{\vec \beta}
			&=
			\sum_{\ell_1>\ell_2>\ldots>\ell_N\ge0 }
			\Bigl(D_{\ell_1}^{[1,1)}
			\ldots 
			D_{\ell_N}^{[1,N)}
			\det\bigl[ \Phi_i(\ell_j) \bigr]_{i,j=1}^{N}
			\Bigr)\Bigl(
			D_{\ell_N}^{[N,N)\dagger}\ldots 
			D_{\ell_1}^{[1,N)\dagger}
			\det\bigl[ \Psi_i(\ell_j) \bigr]_{i,j=1}^{N}
			\Bigr)
			\\&=
			\sum_{\ell_1>\ell_2>\ldots>\ell_N\ge0 }
			\det\bigl[ \Psi_i(\ell_j) \bigr]_{i,j=1}^{N}
			\,
			D_{\ell_1}^{[1,N)}
			D_{\ell_2}^{[1,N)}
			\ldots 
			D_{\ell_N}^{[1,N)}
			\det\bigl[ \Phi_i(\ell_j) \bigr]_{i,j=1}^{N},
		\end{split}
	\end{equation*}
	where we moved each of the operators $D_{\ell_j}^{[j,N)\dagger}$ to the other function and observed that the presence of the determinants eliminates the boundary terms arising from \eqref{eq:boundary_terms_summation_by_parts}.
	Writing 
	\begin{equation*}
		D_{\ell_1}^{[1,N)}
		D_{\ell_2}^{[1,N)}
		\ldots 
		D_{\ell_N}^{[1,N)}
		\det\bigl[ \Phi_i(\ell_j) \bigr]_{i,j=1}^{N}
		=
		\det\Bigl[
			D_{k}^{[1,N)}\Phi_i(k)\ssp\Big\vert_{k=\ell_j} \ssp
		\Bigr]_{i,j=1}^{N},
	\end{equation*}
	we can use the Cauchy--Binet summation to replace the sum of products of two determinants over $\ell_1>\ell_2>\ldots>\ell_N\ge 0 $ by the determinant of single sums.
\end{proof}

\subsection{Two-dimensional process}
\label{sub:tilted_biorthogonal_process}

We will show in \Cref{sec:non_determinantal}
below that a $\vec\beta$-tilted
$N$-point biorthogonal ensemble on $\mathbb{Z}_{\ge0}$ is not necessarily a determinantal point process, even though its probability weights are products of determinants.

On the other hand, 
each $\vec \beta$-tilted biorthogonal
ensemble can be embedded into a \emph{two-dimensional}
determinantal point process 
on $\mathfrak{X}\coloneqq \mathbb{Z}_{\ge0}\times\left\{ 1,\ldots,N  \right\}$.
A similar construction for TASEP first appeared in
\cite{BorodinFPS2007} (and was later exploited 
to construct the KPZ fixed point \cite{matetski2017kpz}).
The embedding which we describe below in this subsection
is suggested in the talk by Kenyon
\cite{Kenyon_IPAM_talk}.

This process lives on particle configurations
$X^{2d}=\{x^m_j \colon 1\le m,j\le N \}$ satisfying
\begin{equation}
	\label{eq:2d_particle_configuration}
	x^m_N<x^m_{N-1}<\ldots<x^m_2<x^m_1,\qquad 1\le m\le N.
\end{equation}
Denote $|x^m|\coloneqq x^m_1+\ldots+x^m_N $.
Let
\begin{equation*}
	T_\beta(x,y)\coloneqq \mathbf{1}_{y=x}-\beta \ssp \mathbf{1}_{y=x-1},\qquad x,y\in \mathbb{Z}_{\ge0}.
\end{equation*}
One readily sees that
\begin{equation}
	\label{eq:T_beta_determinant_evaluation}
	\det[T_\beta(x^m_i,x^{m+1}_j)]_{i,j=1}^N=
	(-\beta)^{|x^m|-|x^{m+1}|}
	\prod_{j=1}^{N}\mathbf{1}_{x^{m}_j-x^{m+1}_j=\ssp0\text{ or $1$}}.
\end{equation}

Using the given notation, assign (possibly complex) weights to configurations $X^{2d}$:
\begin{equation}
	\label{eq:2d_unnormalized_weights}
	\mathscr{W}_{\vec \beta}^{2d}(X^{2d})\coloneqq
	\det\left[ \Phi_i(x^1_j) \right]_{i,j=1}^N
	\biggl(\ssp
		\prod_{m=1}^{N-1}
		\det\left[ T_{\beta_m}(x^m_i,x^{m+1}_j) \right]_{i,j=1}^{N}
	\biggr)
	\det\left[ \Psi_i(x^{N}_j) \right]_{i,j=1}^N.
\end{equation}

In the proof of the next statement and 
throughout the rest of the section, we use the notation
``$*$'' for discrete convolution of functions on $\mathbb{Z}_{\ge0}$,
and assume that all series thus arising converge absolutely.
For example, we write
$(f*h)(x)=\sum_{y=0}^{\infty}f(x,y)h(y)$
for functions $f(x,y)$ and $h(x)$.
See also \cite[Section 4]{borodin2009} for further examples of this notation.

\begin{theorem}
	\label{thm:properties_of_W_2d}
	The normalizing constant of the two-dimensional distribution
	\begin{equation*}
		\mathscr{Z}_{\vec \beta}^{2d}\coloneqq \sum_{X^{2d}}
		\mathscr{W}_{\vec \beta}^{2d}(X^{2d})
	\end{equation*}
	is equal to the one-dimensional normalizing constant
	$\mathscr{Z}_{\vec \beta}$ given by 
	\eqref{eq:tilted_Gram_matrix_partition_function}--\eqref{eq:tilted_Gram_matrix_definition}.
	Moreover, under the 
	normalized two-dimensional probability distribution 
	$\mathscr{M}^{2d}_{\vec \beta}(X^{2d})\coloneqq
	\mathscr{W}_{\vec \beta}^{2d}(X^{2d})/\mathscr{Z}_{\vec \beta}$,
	the marginal distribution of $x^1_1>x^2_2>\ldots>x^N_N $ coincides with that of
	$\ell_1>\ell_2>\ldots\ell_N $ under $\mathscr{M}_{\vec \beta}$.
\end{theorem}

For complex-valued probabilities, the coincidence of marginal distributions means that for any finitely supported function $f$ in $N$ variables, we have
\begin{equation}
	\label{eq:X2d_marginal}
	\sum_{X^{2d}=\{x^m_j\}}
	f(x^1_1,\ldots,x^N_N )\ssp
	\mathscr{M}_{\vec \beta}^{2d}(X^{2d})
	=
	\sum_{X=(\ell_1>\ldots>\ell_N\ge0 )}
	f(\ell_1,\ldots,\ell_N )\ssp
	\mathscr{M}_{\vec \beta}(X).
\end{equation}

\begin{proof}[Proof of \Cref{thm:properties_of_W_2d}]
	By the Cauchy--Binet summation, we have
	\begin{equation*}
		\mathscr{Z}^{2d}_{\vec \beta}=
		\det\bigl[ \Phi_i * T_{\beta_1} * 
		\ldots * T_{\beta_{N-1}} * \Psi_j  \bigr]_{i,j=1}^{N}.
	\end{equation*}
	Next, for any function $h(y)$ on $\mathbb{Z}_{\ge0}$ we have
	$(h*T_{\beta_r})(y)=h(y)-\beta_r h(y+1)=D^{(r)}_y h(y)$.
	By \eqref{eq:tilted_Gram_matrix_definition},
	this implies the first claim about the normalizing constant.

	The second claim essentially follows from the LGV (Lindstrom--Gessel--Viennot) lemma, which expresses the partition function of nonintersecting path collections in a determinantal form
	\cite{lindstrom1973vector}, \cite{gessel1985binomial}.
	By the first claim, it suffices to prove 
	\eqref{eq:X2d_marginal}
	for unnormalized weights
	$\mathscr{W}^{2d}_{\vec \beta}$ and 
	$\mathscr{W}_{\vec \beta}$.
	Next, the weights
	$\mathscr{W}^{2d}_{\vec \beta}(X^{2d})$
	\eqref{eq:2d_unnormalized_weights}
	and 
	$\mathscr{W}_{\vec \beta}(X)$
	\eqref{eq:tilted_biorthogonal}
	are multilinear in 
	$(\Phi_1,\ldots,\Phi_N; \Psi_1,\ldots,\Psi_N )$,
	so it suffices to prove the summation identity
	in the case of delta functions
	\begin{equation*}
		\Phi_i(x)=\mathbf{1}_{x=k_i}
		,\qquad 
		\Psi_i(x)=\mathbf{1}_{x=k_i'}
		,
		\qquad 
		i=1,\ldots,N,
	\end{equation*}
	where $k_1>\ldots>k_N\ge0$ and $k_1'>\ldots>k_N'\ge0$ are arbitrary
	but fixed.
	With this choice of $\Phi_i,\Psi_i$, the distribution
	of $X^{2d}$ is the same as the distribution of the
	nonintersecting path ensemble on the graph shown in 
	\Cref{fig:nonintersecting_path_ensemble},
	where the paths connect $k_1,\ldots,k_N $ to 
	$k_1',\ldots,k_N' $.

	\begin{figure}[htpb]
		\centering
		\includegraphics[width=.4\textwidth]{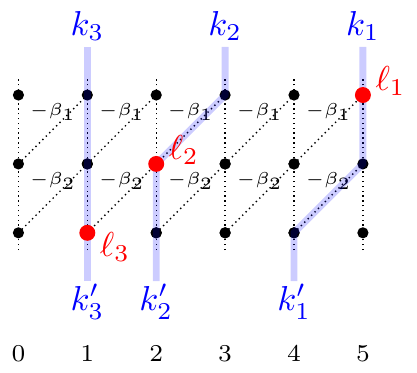}
		\caption{The directed
		graph with vertices $\mathbb{Z}_{\ge0}\times\left\{ 1,\ldots,N  \right\}$
		and edges which can be vertical (with weight $1$) or diagonal (with 
		weight $-\beta_m$, $m=1,\ldots,N $).
		We consider an ensemble of $N$ nonintersecting paths
		connecting $k_1,\ldots,k_N $ to $k_1',\ldots,k_N' $.
		The particles $x^m_j$ encode the intersections of the paths with 
		the  $m$-th horizontal line,
		$m=1,\ldots,N $.}
		\label{fig:nonintersecting_path_ensemble}
	\end{figure}

	Then the marginal distribution of $\ell_1,\ldots,\ell_N $
	can be expressed through the product of two determinants:
	One for the nonintersecting paths connecting $k_1,\ldots,k_N  $ to
	$\ell_1,\ldots,\ell_N $, and the other one for the nonintersecting
	paths from 
	$\ell_1,\ldots,\ell_N $
	to 
	$k_1',\ldots,k_N' $.
	These determinants are immediately
	identified with the two determinants in \eqref{eq:tilted_biorthogonal},
	and so we are done.
\end{proof}

\subsection{Determinantal kernel}
\label{eq:det_kernel}

The two-dimensional ensemble $X^{2d}$ defined 
in \Cref{sub:tilted_biorthogonal_process} is a \emph{determinantal
point process}. This means that
for any $p\ge1$ and pairwise
distinct points
$(y_i,t_i)\in \mathbb{Z}_{\ge0}\times\left\{ 1,\ldots,N  \right\}$, 
$i=1,\ldots,p $, we have
\begin{equation}
	\label{eq:Kernel_for_2d_definition}
	\sum_{\substack{X^{2d}\colon\text{$X^{2d}$ cointains each}\\
	(y_i,t_i),\ i=1,\ldots,p}}
	\mathscr{M}^{2d}_{\vec \beta}(X^{2d})
	=\det
	\bigl[ K_{\vec\beta}^{2d}(y_i,t_i;y_j,t_j) \bigr]_{i,j=1}^{p}.
\end{equation}
Here $K_{\vec\beta}^{2d}(x,t;y,s)$ is a function called the \emph{correlation kernel}.
Both the determinantal structure and an expression for the correlation
kernel
follow from the well-known Eynard--Mehta
theorem 
\cite{eynard1998matrices}, 
\cite{borodin2005eynard}, 
see also \cite[Theorem 4.2]{borodin2009}. 

\begin{proposition}
	\label{prop:Eynard_Mehta_kernel}
	The correlation kernel \eqref{eq:Kernel_for_2d_definition} 
	for the point process $X^{2d}$ has the form
	\begin{equation}
		\label{eq:Kernel_Eynard_Mehta}
		K_{\vec\beta}^{2d}(x,t;y,s)=
		-\mathbf{1}_{t>s}
		\left( T_{\beta_s}*\ldots*T_{\beta_{t-1}}  \right)(y,x)
		+
		\sum_{i,j=1}^{N}
		G_{ji}^{-1}(\vec \beta)\cdot
		D_x^{[1,t)}
		\Phi_i(x)
		\cdot
		D_y^{[s,N)\dagger}\Psi_j(y).
	\end{equation}
\end{proposition}
\begin{proof}
	By \cite[Theorem 4.2]{borodin2009}, the correlation kernel has the form
	\begin{multline*}
		K_{\vec\beta}^{2d}(x,t;y,s)=
		-\mathbf{1}_{t>s}
		\left( T_{\beta_s}*T_{\beta_{s+1}}*\ldots*T_{\beta_{t-1}}  \right)(y,x)
		\\+
		\sum_{i,j=1}^{N}
		G_{ji}^{-1}(\vec \beta)\cdot
		\bigl( \Phi_i*T_{\beta_1}*\ldots* T_{\beta_{t-1}}  \bigr)(x)
		\cdot
		\bigl( T_{\beta_s}*\ldots* T_{\beta_{N-1}}*\Psi_j  \bigr)(y).
	\end{multline*}
	As in the proof of \Cref{thm:properties_of_W_2d},
	we can rewrite the convolutions with the $T_\beta$'s as 
	applications of the difference operators \eqref{eq:def_operators_D}:
	\begin{equation*}
		\bigl(\Phi_i*T_{\beta_1}*\ldots* T_{\beta_{t-1}}\bigr)(x) 
		= 
		D_x^{[1,t)}
		\Phi_i(x);\qquad 
		\bigl( T_{\beta_s}*\ldots* T_{\beta_{N-1}}*\Psi_j  \bigr)(y)
		=D_y^{[s,N)\dagger}\Psi_j(y).
	\end{equation*}
	This completes the proof.
\end{proof}

Note that the variables $t,s\in \left\{ 1,\ldots,N  \right\}$
in the correlation kernel 
\eqref{eq:Kernel_Eynard_Mehta}
correspond to the vertical coordinates
in \Cref{fig:nonintersecting_path_ensemble}
which increases from top to bottom. We use this 
convention throughout the rest of the paper.

\begin{remark}
	\label{rmk:Expression_of_measure_via_K2d}
	From the fact that $x^j_j=\ell_j$, $j=1,\ldots,N $, as joint distributions
	(where the $x^j_j$'s come from $\mathscr{M}_{\vec \beta}^{2d}$ and the $\ell_j$'s come from 
	$\mathscr{M}_{\vec \beta}$),
	one may think that the probabilities 
	$\mathscr{M}_{\vec\beta}$ are expressed through the correlation kernel
	$K_{\vec\beta}^{2d}$ as 
	\begin{equation}
		\label{eq:measure_through_K_2d}
		\mathscr{M}_{\vec\beta}(\ell_1,\ldots,\ell_N )=
		\det
		\bigl[
			K_{\vec\beta}^{2d}(\ell_i,i;\ell_j,j)
		\bigr]_{i,j=1}^{N},\qquad \ell_1>\ldots>\ell_N.
	\end{equation}
	However, identity \eqref{eq:measure_through_K_2d}
	is generally \emph{false} when $\ell_{i+1}=\ell_i+1$ for some $i$.
	Indeed, this is because the correlation event 
	in the right-hand side of \eqref{eq:measure_through_K_2d}
	includes more configurations of nonintersecting paths (as in \Cref{fig:nonintersecting_path_ensemble})
	than just the ones with $x^j_j=\ell_j$ for all $j=1,\ldots,N $.
	One can check that if
	$\ell_j-\ell_{j+1}\ge 2$ for all $j=1,\ldots,N $,
	then identity \eqref{eq:measure_through_K_2d} holds.
\end{remark}

\subsection{Marginals and correlations of the tilted biorthogonal ensemble}
\label{sub:correlations_ensemble}

Fix $k\ge1$ and 
$\mathcal{I}=\{i_1<\ldots<i_k \}\subset\left\{ 1,\ldots,N  \right\}$.
Let $a_{\mathcal{I}}=(a_{i_1}>\ldots>a_{i_k}\ge0 )$ be a fixed integer vector,
and also let $X_{\mathcal{I}}=(\ell_{i_1}>\ldots>\ell_{i_k} \ge0)$ be a
random vector, which is a marginal of the 
$\vec \beta$-tilted biorthogonal ensemble
$\mathscr{M}_{\vec\beta}(X)$
defined by
\eqref{eq:tilted_biorthogonal}--\eqref{eq:M_normalized_1d_weights}.
Using \Cref{thm:properties_of_W_2d} and \Cref{prop:Eynard_Mehta_kernel},
we can express the probability 
$\mathscr{M}_{\vec \beta}(X_{\mathcal{I}}=a_{\mathcal{I}})$ 
through the 
correlation kernel $K^{2d}_{\vec\beta}$ in a polynomial way.

We use the following statement adapted to our space
$\mathfrak{X}=\mathbb{Z}_{\ge0}\times\left\{ 1,\ldots,N  \right\}$:
\begin{lemma}[{\cite[Theorem 2]{Soshnikov2000}}]
	\label{lemma:det_pp_generating_function}
	Fix
	a finite number of disjoint subsets of
	$\mathfrak{X}$
	and denote them by
	$B_1,\ldots,B_p$. 
	Let $B=B_1\cup \ldots\cup B_p $.
	For a determinantal
	point process on $\mathfrak{X}$ with kernel $K$,
	let $\#_{B_i}$ be the random number of points of the process
	which belong to $B_i$.
	Then
	we have the following identity of generating functions in $z_1,\ldots,z_p$:
	\begin{equation}
		\label{eq:det_pp_generating_function}
		\mathbb{E}
		\bigl( z_1^{\#_{B_1}}\ldots z_p^{\#_{B_p}} \bigr)=
		\det
		\Bigl( \mathbf{1}-\chi_{_B} \sum_{i=1}^{p}(1-z_i)\cdot K \cdot \chi_{_{B_i}} \Bigr),
	\end{equation}
	where $\mathbf{1}$ is the identity operator, in the right-hand side there is a Fredholm determinant, and $\chi_{_B},\chi_{_{B_i}}$ are the indicator functions of these subsets.
\end{lemma}
In our applications, the sets $B_i$ will be finite, and thus the 
Fredholm determinants in 
\eqref{eq:det_pp_generating_function} are simply finite-dimensional determinants
of the corresponding block matrices. In general, the right-hand
side of \eqref{eq:det_pp_generating_function} is an infinite series,
see, for example, \cite[Remark 3]{Soshnikov2000}.

To illustrate the general formula of
\Cref{prop:correlations_tilted} below,
let us first look at the case $k=1$.
For fixed $a$ and $i$, the event $\ell_i=a$
is equivalent to $\#_{B_i(a)}=N-i$, $\#_{C_i(a)}=1$, 
where
\begin{equation}
	\label{eq:events_notation}
	B_i(a)\coloneqq \left\{ 0,1,\ldots,a-1 \right\}\times\left\{ i \right\},\qquad 
	C_i(a)\coloneqq \left\{ a \right\}\times\left\{ i \right\},
	\qquad 
	F_i(a)\coloneqq B_i(a)\cup C_i(a).
\end{equation}
Indeed, for $\ell_i=x^i_i=a$, we need to have exactly $N-i$ points of the configuration
$X^{2d}$ to the left of $a$, and exactly one point at $a$.
Thus, we 
can write by \Cref{lemma:det_pp_generating_function}:
\begin{equation}
	\label{eq:one_point_marginal}
	\mathscr{M}_{\vec \beta}(\ell_i=a)=
	[z^{N-i}w]\det
	\Bigl( 
		\mathbf{1}
		-
		(1-z)\chi_{_{F_i(a)}}K^{2d}_{\vec\beta} \chi_{_{B_{i}(a)}}
		-
		(1-w)\chi_{_{F_i(a)}}K^{2d}_{\vec\beta} \chi_{_{C_{i}(a)}}
	\Bigr),
\end{equation}
where $[z^{N-i}w]$ is the operator of taking the coefficient of a polynomial by $z^{N-i}w$.
The matrix in the right-hand side \eqref{eq:one_point_marginal} has dimensions $(a+1)\times(a+1)$ and looks as
\begin{equation*}
	\begin{bmatrix}
		1+(z-1)K(0;0)&
		(z-1)K(0;1)&\ldots& 
		(w-1)K(0;a)\\
		(z-1)K(1;0)&
		1+(z-1)K(1;1)&\ldots& 
		(w-1)K(1;a)\\
		\dotfill&\dotfill&\dotfill&
		\dotfill
		\\
		(z-1)K(a-1;0)&
		(z-1)K(a-1;1)&\ldots& 
		(w-1)K(a-1;a)
		\\
		(z-1)K(a;0)&
		(z-1)K(a;1)&\ldots& 
		1+(w-1)K(a;a)
		\end{bmatrix},
\end{equation*}
where we abbreviated $K(x;y)=K^{2d}_{\vec\beta}(x,i;y,i)$.

Finally, to get the correlation function, we simply have to sum \eqref{eq:one_point_marginal}
over all $i=1,\ldots,N $:
\begin{equation}
	\label{eq:one_point_correlation}
	\mathscr{M}_{\vec \beta}(\textnormal{$X$ contains $a$})=
	\sum_{i=1}^{N}
	[z^{N-i}w]\det
	\Bigl( 
		\mathbf{1}
		-
		(1-z)\chi_{_{F_i(a)}}K^{2d}_{\vec\beta} \chi_{_{B_{i}(a)}}
		+
		(1-w)\chi_{_{F_i(a)}}K^{2d}_{\vec\beta} \chi_{_{C_{i}(a)}}
	\Bigr).
\end{equation}
Notice that this is a polynomial in the entries
$K^{2d}_{\vec\beta}(x,t;y,s)$ of the correlation kernel \eqref{eq:Kernel_Eynard_Mehta}.

The next statement for general $k$
follows from an argument for several points which is analogous to the above 
computations:
\begin{proposition}
	\label{prop:correlations_tilted}
	For arbitrary $k\ge1$ and $\mathcal{I}=\{i_1<\ldots<i_k \}$,
	the marginal distribution of $X_{\mathcal{I}}$ 
	under $\mathscr{M}_{\vec \beta}$ has the form
	\begin{equation}
		\label{eq:marginal_multipoint}
		\begin{split}
			&\mathscr{M}_{\vec\beta}(X_{\mathcal{I}}=a_{\mathcal{I}})=
			[z_1^{N-i_1}\ldots z_{k}^{N-i_k} w_1\ldots w_k ]
			\det\biggl(
				\mathbf{1}
				-
				\chi_{_{F_{\mathcal{I}}(a_{\mathcal{I}})}}\sum_{p=1}^k 
				(1-z_p)
				K^{2d}_{\vec\beta}\chi_{_{B_{i_p}(a_{i_p})}}
				\\&\hspace{260pt}+
				\chi_{_{F_{\mathcal{I}}(a_{\mathcal{I}})}}\sum_{p=1}^k (1-w_p)K^{2d}_{\vec\beta}\chi_{_{C_{i_p}(a_{i_p})}}
			\biggr).
		\end{split}
	\end{equation}
	Here the
	square matrix has dimensions 
	$\sum_{p=1}^{k}(a_{i_p}+1)$, 
	the union of all the sets is denoted by
	$F_{\mathcal{I}}(a_{\mathcal{I}})\coloneqq\bigcup_{p=1}^{k}\left( B_{i_p}(a_{i_p})\cup C_{i_p}(a_{i_p}) \right)$,
	and the 
	determinant 
	is a
	polynomial in the entries
	of the correlation kernel $K^{2d}_{\vec \beta}$.

	The correlation functions of $\mathscr{M}_{\vec \beta}$ are finite sums of 
	determinants of the form \eqref{eq:marginal_multipoint}. Namely, 
	for any $k$ and any pairwise distinct $a_1,\ldots,a_k \in \mathbb{Z}_{\ge0}$, we have
	\begin{equation*}
		\mathscr{M}_{\vec\beta}(\textnormal{$X$ contains $a_1,\ldots,a_k $})=
		\sum_{\substack{\mathcal{I}=\{1\le i_1<\ldots<i_k\le N \}}}
		\mathscr{M}_{\vec\beta}\left( X_{\mathcal{I}}=\{a_1,\ldots,a_k \} \right).
	\end{equation*}
\end{proposition}

\section{Grothendieck random partitions}
\label{sec:Grothendieck_measures_definition_discussion}

Here, we specialize the setup of $\vec\beta$-tilted biorthogonal ensembles
developed in \Cref{sec:tilted_biorthogonal}
to \emph{Grothendieck random partitions}.
A crucial feature in this special case is that the
corresponding two-dimensional ensemble $X^{2d}$
becomes the well-known \emph{Schur process} introduced in
\cite{okounkov2003correlation}.

\subsection{Specialization of tilted biorthogonal ensemble}
\label{sub:Grothendieck}

Fix $N\ge 1$ and
parameters $x_1,\ldots,x_N,y_1,\ldots,y_N$ such that
$|x_iy_j|<1$ for all $i,j$, and specialize
\begin{equation}
	\label{eq:Phi_Psi_specializations_to_Grothendieck}
	\Phi_i(k)=x_i^k,\qquad \Psi_j(k)=y_j^k,\qquad k\in \mathbb{Z}_{\ge0}.
\end{equation}
Then the operators \eqref{eq:def_operators_D}
act as 
\begin{equation}
	\label{eq:D_D_dagger_specializations_to_Grothendieck}
	D_k^{(r)}\Phi_i(k)=x_i^k(1-\beta_r x_i)
	,\qquad 
	D_k^{(r)\dagger}\Psi_j(k)=y_j^k(1-\beta_r y_j^{-1}\mathbf{1}_{k\ge1})
	.
\end{equation}
For a configuration $X=(\ell_1>\ldots>\ell_N\ge0 )$,
denote $\lambda_j\coloneqq \ell_j+j-N$, $j=1,\ldots,N $,
so $\ell_j=\lambda_j+N-j$.
Clearly, we have 
$\lambda=(\lambda_1\ge \ldots\ge \lambda_N\ge0 )$, and $\lambda$ is an integer
partition with at most $N$ parts.
The $\vec\beta$-tilted biorthogonal weight 
\eqref{eq:tilted_biorthogonal}
specializes to 
\begin{equation*}
	\begin{split}
		\mathscr{W}_{\vec \beta;\mathrm{Gr}}(\lambda)
		&=
		\det\bigl[ 
		x_i^{\lambda_j+N-j}(1-\beta_1x_i)\ldots(1-\beta_{j-1}x_i) \bigr]_{i,j=1}^{N}
		\\&\hspace{90pt}\times\det\bigl[ 
			y_i^{\lambda_j+N-j}(1-\beta_j y_i^{-1})\ldots(1-\beta_{N-1}y_i^{-1}) 
		\bigr]_{i,j=1}^{N}.
	\end{split}
\end{equation*}
Observe that in the second determinant, the operators $D_{\ell_j}^{[j,N)\dagger}$ are applied
in $\ell_1,\ldots,\ell_{N-1} $, which are strictly positive.
Therefore, the special case $k=0$ in $D_k^{(r)\dagger}$ in
\eqref{eq:D_D_dagger_specializations_to_Grothendieck}
does not occur.

The normalizing constant 
in \Cref{prop:partition_function_tilted}
becomes
\begin{equation*}
	\begin{split}
		\mathscr{Z}_{\vec \beta;\mathrm{Gr}}
		&=
		\det\Bigl[ 
			\frac{(1-\beta_1x_i)\ldots (1-\beta_{N-1}x_i) }{1-x_iy_j}
		\Bigr]_{i,j=1}^{N}
		\\&=
		\frac{
		\prod_{i=1}^{N}\prod_{r=1}^{N-1}
		(1-x_i\beta_r)}{\prod_{i,j=1}^N(1-x_iy_j)}
		\ssp
		\prod_{1\le i<j\le N}(x_i-x_j)(y_i-y_j),
	\end{split}
\end{equation*}
where the matrix elements are geometric
sums, see \eqref{eq:tilted_Gram_matrix_definition}, 
and the well-known Cauchy determinant is evaluated in a product form.

Let us denote 
\begin{equation}
	\label{eq:Grothendieck_polynomials}
	\begin{split}
		G_\lambda(x_1,\ldots,x_N )
		&\coloneqq
		\frac{\det\bigl[ 
		x_i^{\lambda_j+N-j}(1-\beta_1x_i)\ldots(1-\beta_{j-1}x_i) \bigr]_{i,j=1}^{N}}{\prod_{1\le i<j\le N}(x_i-x_j)};
		\\
		\overline{G}_\lambda(y_1,\ldots,y_N )
		&\coloneqq
		\frac{\det\bigl[ 
			y_i^{\lambda_j+N-j}(1-\beta_j y_i^{-1})\ldots(1-\beta_{N-1}y_i^{-1})
		\bigr]_{i,j=1}^{N}}{\prod_{1\le i<j\le N}(y_i-y_j)}.
	\end{split}
\end{equation}
Since $\lambda_j+N-j\ge N-j$ and in the matrix elements in $\overline{G}_\lambda$
there are $N-j$ factors of the form $(1-\beta_r y_i^{-1})$,
we see that both $G_\lambda$ and $\overline G_\lambda$
are symmetric polynomials in $N$ variables. 
We thus see that the $\vec\beta$-tilted
biorthogonal ensemble with the specialization 
\eqref{eq:Phi_Psi_specializations_to_Grothendieck}
has the form
\begin{equation}
	\label{eq:Grothendieck_measure_definition_complete}
	\mathscr{M}_{\vec \beta;\mathrm{Gr}}(\lambda)
	=
	\frac
	{\prod_{i,j=1}^N(1-x_iy_j)}
	{\prod_{i=1}^{N}\prod_{r=1}^{N-1}(1-x_i\beta_r)}
	\ssp
	G_\lambda(x_1,\ldots,x_N )\ssp
	\overline G_{\lambda}(y_1,\ldots,y_N ).
\end{equation}
We call \eqref{eq:Grothendieck_measure_definition_complete}
the (multiparameter)
\emph{Grothendieck measure on partitions}.
This distribution is analogous to the Schur
measure introduced in \cite{okounkov2001infinite}
which is a particular case of 
$\mathscr{M}_{\vec \beta;\mathrm{Gr}}$
for $\beta_r\equiv 0$.

Note that the probability weights 
\eqref{eq:Grothendieck_measure_definition_complete} may be 
complex-valued. In \Cref{sub:references_positivity} below
we discuss conditions on the parameters $x_i,y_j,\beta_r$ which make
the weights nonnegative real.

\subsection{Grothendieck polynomials and positivity}
\label{sub:references_positivity}

Here, we comment on the relations between the polynomials
$G_\lambda,\overline G_\lambda$ 
and Grothendieck polynomials appearing in the literature. We also
discuss the nonnegativity of the measure 
$\mathscr{M}_{\vec \beta;\mathrm{Gr}}$
\eqref{eq:Grothendieck_measure_definition_complete}
on partitions.

Grothendieck polynomials are
well-known in algebraic combinatorics and geometry, 
going back to at least
\cite{lascoux1982structure}, see also
\cite{buch2002littlewood}.
Their one-parameter $\beta$-deformations
appeared in
\cite{fomin1994grothendieck}.
The recent paper 
\cite{hwang2021refined}
introduced and studied the most general (to date)
deformations
called \emph{refined canonical stable Grothendieck polynomials}
$\mathsf{G}_\lambda(x_1,\ldots,x_N;\vec \alpha,\vec \beta )$.
These objects generalize most known Grothendieck-like polynomials
in the literature,
in particular, the ones in \cite{buch2002littlewood},
\cite{fomin1994grothendieck}, as well
as more recent extensions in, e.g.,
\cite{yeliussizov2015duality}, \cite{chan2021combinatorial}.
The refined canonical stable Grothendieck polynomials
$\mathsf{G}_\lambda(x_1,\ldots,x_N;\vec \alpha,\vec
\beta )$ depend on two sequences of parameters
$\vec\alpha=(\alpha_1,\alpha_2,\ldots )$
and
$\vec\beta=(\beta_1,\beta_2,\ldots )$, and are defined as 
\begin{equation}
	\label{eq:kim_polynomials}
	\mathsf{G}_\lambda(x_1,\ldots,x_N ;\vec \alpha,\vec \beta)
	\coloneqq
	\frac{\det\biggl[ 
	x_i^{\lambda_j+N-j}
	\dfrac{(1-\beta_1x_i)\ldots(1-\beta_{j-1}x_i)}
	{(1-\alpha_1 x_i)\ldots(1-\alpha_{\lambda_j}x_i) } \biggr]_{i,j=1}^{N}}{\prod_{1\le i<j\le N}(x_i-x_j)}.
\end{equation}
Note that for nonzero $\alpha_j$'s, 
$\mathsf{G}_\lambda(x_1,\ldots,x_N;\vec \alpha,\vec \beta )$
are not polynomials but rather are generating series in the $x_j$'s.
When $\alpha_j=0$ for all $j$ (which drops the word ``canonical'' from the terminology), expressions
\eqref{eq:kim_polynomials} become polynomials and reduce to our $G_\lambda$'s from \eqref{eq:Grothendieck_polynomials}.
The polynomials $\overline{G}_\lambda$ in \eqref{eq:Grothendieck_polynomials}
are expressed through the $G_{\lambda}$'s as follows.
Denote $\beta_r^{rev}\coloneqq\beta_{N-r}$, $r=1,\ldots,N-1$, 
and $\lambda^{rev}\coloneqq (0\ge -\lambda_N\ge -\lambda_{N-1}\ge \ldots\ge -\lambda_1 )$.
Then, one readily sees that
\begin{equation}\label{eq:reversal_property_G_bar}
	\overline{G}_{\lambda}(x_1,\ldots,x_N \mid \vec \beta )=
	G_{\lambda^{rev}}(x_1^{-1},\ldots,x_{N}^{-1}\mid \vec \beta^{rev} ),
\end{equation}
where we explicitly indicated the dependence on the 
parameters $\beta_r$.
Moreover, since $G_\lambda$ satisfies the index shift property
$G_{\lambda+(1,\ldots,1 )}(x_1,\ldots,x_N )=x_1\ldots x_N \cdot G_{\lambda}(x_1,\ldots,x_N )$,
one can shift the negative coordinates $\lambda^{rev}$ to obtain a nonnegative partition.

\medskip

The sum to one property of the Grothendieck measure
$\mathscr{M}_{\vec\beta;\mathrm{Gr}}$ \eqref{eq:Grothendieck_measure_definition_complete}
is equivalent to the following Cauchy-type summation identity for the 
polynomials \eqref{eq:Grothendieck_polynomials}:
\begin{equation}
	\label{eq:Cauchy_for_G}
	\sum_{\lambda=(\lambda_1\ge \ldots\ge \lambda_N\ge 0 )}
	G_{\lambda}(x_1,\ldots,x_N )
	\ssp
	\overline{G}_{\lambda}(y_1,\ldots,y_N )
	=
	\frac{
	\prod_{i=1}^{N}\prod_{r=1}^{N-1}
	(1-x_i\beta_r)}{\prod_{i,j=1}^N(1-x_iy_j)},
	\qquad |x_iy_j|<1.
\end{equation}
It is instructive to compare this identity to Cauchy
identities for Grothendieck symmetric functions, for example, see
\cite[(36)]{yeliussizov2015duality} or
\cite[Corollary 3.6]{hwang2021refined}. The latter identities involve sums of products in the form $\mathsf{G}_\lambda\ssp \mathsf{g}_\lambda$, where $\mathsf{g}_\lambda$ are the dual Grothendieck symmetric functions.
The products in the right-hand side of these summation identities
have the form
$\prod_{i,j=1}^{\infty}(1-x_iy_j)^{-1}$, and a possible
analogue in our case would be 
$\prod_{i,j=1}^{\infty}\frac{1-x_i\beta_j}{1-x_iy_j}$.
However, in this paper, we will not explore a symmetric function
extension of the identity \eqref{eq:Cauchy_for_G}.

\medskip

Let us now discuss the nonnegativity of the probability
weights $\mathscr{M}_{\vec\beta;\mathrm{Gr}}$
\eqref{eq:Grothendieck_measure_definition_complete}.
Using the tableau formula for $G_\lambda$
(for example, \cite[Corollary 4.5]{hwang2021refined})
and the relation \eqref{eq:reversal_property_G_bar}
between $G_\lambda$ and $\overline{G}_\lambda$, 
we see that the probability weights 
$\mathscr{M}_{\vec\beta;\mathrm{Gr}}(\lambda)$ are nonnegative for all $\lambda$
when the parameters satisfy
\begin{equation}
	\label{eq:tableau_formula_for_G_and_easy_positivity}
	x_i\ge 0,\quad y_j\ge0,\quad \beta_r\le0;\qquad |x_iy_j|<1;
	\qquad 
	1\le i,j\le N,\quad 1\le r\le N-1.
\end{equation}
Indeed, under \eqref{eq:tableau_formula_for_G_and_easy_positivity}
we have 
nonnegativity (and even Schur--nonnegativity, 
cf.~\cite[Theorem 4.3]{hwang2021refined})
of $G_\lambda$ and $\overline{G}_\lambda$, 
as well as the convergence of 
the series \eqref{eq:Cauchy_for_G}.

Furthermore, we can extend the nonnegativity range 
of the Grothendieck measures
to 
certain positive values of $\beta_r$:
\begin{proposition}
	\label{prop:Grothendieck_extended_positivity}
	Let $x_i,y_j\ge0$ and $\beta_r\le x_i^{-1}$, $\beta_r\le y_j$ for all $i,j,r$.
	Then the Grothendieck polynomials $G_\lambda(x_1,\ldots,x_N )$ and
	$\overline{G}_{\lambda}(y_1,\ldots,y_N )$ defined by \eqref{eq:Grothendieck_polynomials} are nonnegative.
\end{proposition}
\begin{proof}
	We consider only the case $G_\lambda$, 
	as $\overline{G}_\lambda$ is completely analogous.
	For a nonnegative
	function $f$ on $\mathbb{Z}_{\ge0}$ we have under our conditions:
	\begin{equation*}
		D_k^{(r)}f(k) = f(k)-\beta_r f(k+1) \geq f(k)-x_r^{-1} f(k+1).
	\end{equation*}
	Therefore, 
	replacing $\beta_r$ 
	by $x_r^{-1}$
	in the 
	application of $D_k^{(r)}$
	can only decrease the result.

	The Grothendieck polynomial
	$G_\lambda(x_1,\ldots,x_N )$ is obtained by applying
	the operators $D_k^{(r)}$ to the Schur polynomial
	$s_\lambda(x_1,\ldots,x_N )$:
	\begin{equation*}
		G_\lambda(x_1,\ldots,x_N ) = \dfrac{ \det\bigl[
		D_{\ell_j}^{[1,j)}
		x_i^{\ell_j}\Bigr]_{i,j=1}^{N}}{\prod_{1\le i<j\le
		N}(x_i-x_j)} = D_{\ell_1}^{[1,1)}\dots D_{\ell_N}^{[1,N)}
		s_{\lambda}(x_1,\ldots,x_N ) .
	\end{equation*}
	It follows that $G_\lambda (x_1,\ldots,x_N ) \geq
	G_\lambda(x_1,\ldots,x_N )\bigl|_{\beta_r = x_r^{-1}\text{ for all $r$}}$.
	On the other hand, when $\beta_r=x_r^{-1}$ for all $r$, 
	the matrix in the numerator in $G_\lambda$
	\eqref{eq:Grothendieck_polynomials}
	becomes triangular, and we have
	\begin{equation*}
		\det\bigl[ x_i^{\ell_j}(1-x_i/x_1)\ldots (1-x_i/x_{j-1})  \bigr]_{i,j=1}^N
		=x_1^{\ell_1}\ldots x_N^{\ell_N} \prod_{1\le i<j\le N}\left( 1-\tfrac{x_j}{x_i} \right).
	\end{equation*}
	Cancelling the product over $i,j$ with the denominator in $G_\lambda$, we see that 
	after the substitution, the resulting expression
	$G_\lambda(x_1,\ldots,x_N )\bigl|_{\beta_r = x_r^{-1}\text{ for all $r$}}$
	is clearly nonnegative. This completes the proof.
\end{proof}

\Cref{prop:Grothendieck_extended_positivity} implies that the
Grothendieck 
probability weights 
$\mathscr{M}_{\vec\beta;\mathrm{Gr}}(\lambda)$ are nonnegative for all $\lambda$
when the parameters satisfy the extended conditions
\begin{equation}
	\label{eq:extended_G_positivity}
	x_i\ge 0,\quad y_j\ge0,\quad \beta_r\le x_i^{-1}
	,\quad \beta_r\le y_j;
	\quad |x_iy_j|<1;
	\quad 
	1\le i,j\le N,\quad 1\le r\le N-1.
\end{equation}

\subsection{Two-dimensional Schur process and its correlation kernel}
\label{sub:2d_Schur_process}

By \Cref{thm:properties_of_W_2d}, 
the Grothendieck measure 
is embedded into the two-dimensional ensemble
$X^{2d}$ \eqref{eq:2d_particle_configuration}.
Our specialization \eqref{eq:Phi_Psi_specializations_to_Grothendieck}
implies that $X^{2d}$ is distributed as the \emph{Schur process}.
Schur processes are a vast family of determinantal point
processes on the two-dimensional lattice introduced
and studied in
\cite{okounkov2003correlation}.

Assume that the parameters
satisfy \eqref{eq:tableau_formula_for_G_and_easy_positivity},
and define
$\mu_i^m\coloneqq x^m_i+i-N$, $i,m=1,\ldots,N $,
where the particles $x^m_i$ come from the two-dimensional ensemble
$X^{2d}$.
Clearly,
each $\mu^m=(\mu^m_1\ge \ldots\ge \mu^m_N\ge0)$ is a partition 
with at most $N$ parts.
From 
\eqref{eq:T_beta_determinant_evaluation}--\eqref{eq:2d_unnormalized_weights} we conclude that the 
probability weight of the tuple of partitions
$(\mu^1,\ldots,\mu^N)$ is
\begin{equation}
	\label{eq:Schur_process_from_Grothendieck}
	\begin{split}
		\mathscr{M}_{\vec \beta;\mathrm{Gr}}^{2d}(
		\mu^1,\ldots,\mu^N )
		&\propto
		s_{\mu^1}(x_1,\ldots,x_N )
		\\
		&\hspace{20pt}\times
		s_{(\mu^1)'/(\mu^2)'}(-\beta_1)\ldots 
		s_{(\mu^{N-1})'/(\mu^{N})'}(-\beta_{N-1})\ssp
		s_{\mu^{N}}(y_1,\ldots,y_N ).
	\end{split}
\end{equation}
Here
$(\mu^{m})'/(\mu^{m+1})'$ denote skew transposed partitions, and
we used a basic property of skew Schur functions
evaluated at a single variable (for example,
see \cite[Chapter I.5]{Macdonald1995}):
\begin{equation*}
	\det\left[ T_{\beta_m}(x^m_i,x^{m+1}_j) \right]_{i,j=1}^{N}=
	s_{(\mu^{m})'/(\mu^{m+1})'}(-\beta_{m}).
\end{equation*}

From
\Cref{thm:properties_of_W_2d} we immediately get:
\begin{proposition}
	\label{prop:Grothendieck_measures_embedding}
	The Grothendieck measure $\mathscr{M}_{\vec\beta;\mathrm{Gr}}(\lambda)$
	\eqref{eq:Grothendieck_measure_definition_complete}
	is embedded into the Schur process 
	\eqref{eq:Schur_process_from_Grothendieck}
	in the sense that 
	the joint distributions of the integer $N$-tuples
	$\{\lambda_i\}_{i=1,\ldots,N }$
	and
	$\{\mu_i^i \}_{i=1,\ldots,N }$
	coincide.
\end{proposition}

\begin{remark}
	\label{rmk:Schur_nonpositive_Groth_positive}
	While the Schur process
	\eqref{eq:Schur_process_from_Grothendieck}
	is not a nonnegative measure for $\beta_r>0$,
	the Grothendieck
	measures \eqref{eq:Grothendieck_measure_definition_complete}
	are still nonnegative probability measures
	under the more relaxed conditions
	\eqref{eq:extended_G_positivity}.
	Consequently, we
	will primarily focus on the case 
	when the parameters satisfy the more restrictive
	conditions \eqref{eq:tableau_formula_for_G_and_easy_positivity}.
	However, below in 
	\Cref{sub:simulations_and_particular_cases}
	we will also consider the question of limit shapes
	for Grothendieck measures with positive $\beta_r$'s
	(which do not correspond to nonnegative Schur processes).
\end{remark}

As shown in 
\cite{okounkov2003correlation},
the correlation kernel of the Schur process \eqref{eq:Schur_process_from_Grothendieck}
has a double contour integral form.
The alternative proof of this result given in \cite[Theorem 2.2]{borodin2005eynard} proceeds from the general kernel $K^{2d}_{\vec\beta}$ \eqref{eq:Kernel_Eynard_Mehta} and involves an explicit inverse matrix $G^{-1}(\vec\beta)$ which is available thanks to the Cauchy determinant. Let us record this double contour integral kernel:

\begin{proposition}
	\label{prop:Schur_correlation_kernel_via_integrals}
	The correlation kernel for the Schur process 
	$X^{2d}=\{x^m_i\colon 1\le m,i\le N\}$
	containing
	the Grothendieck measure $\mathscr{M}_{\vec\beta;\mathrm{Gr}}$
	\eqref{eq:Grothendieck_measure_definition_complete}
	has the form 
	\begin{equation}
		\label{eq:Kernel_Schur}
		K_{\vec\beta;\mathrm{Gr}}^{2d}(a,t;b,s)=
		\frac{1}{(2\pi \mathbf{i})^2}
		\oiint
		\frac{dz\ssp dw}{z-w}\frac{w^{b-N}}{z^{a-N+1}}\frac{F_t(z)}{F_s(w)},
	\end{equation}
	where 
	$a,b\in \mathbb{Z}_{\ge0}$,
	$t,s \in\left\{ 1,\ldots,N \right\}$,
	\begin{equation}
		\label{eq:F_t_function_for_kernel}
		F_t(z)\coloneqq
		\prod_{i=1}^{N}\frac{1-z^{-1}y_i}{1-zx_i}\ssp\prod_{r=t}^{N-1}
		\frac1{1-\beta_r z^{-1}}
		,
	\end{equation}
	and the integration contours in
	\eqref{eq:Kernel_Schur} are positively oriented simple closed 
	curves around $0$
	satisfying the following conditions:
	\begin{enumerate}[\rm{}(1)\/]
		\item $|z|>|w|$ for $t\le s$ and $|z|<|w|$ for $t>s$;
		\item On the contours it must be that $|\beta_r|<|z|<x_i^{-1}$
			and $|w|>y_i$ for all $i$ and $r$.
	\end{enumerate}
\end{proposition}

The integration contours in \Cref{prop:Schur_correlation_kernel_via_integrals}
exist only under certain conditions on the
parameters $x_i,y_j$, and $\beta_r$, for example, it must be that $|\beta_r|<x_i^{-1}$
for all $i,r$. When these conditions on the parameters are violated, we should 
deform the integration contours to take the same residues.
In other words, we can analytically continue the kernel by declaring that the double contour integral in \eqref{eq:Kernel_Schur} is always equal to the sum of the same residues:
 first
take the sum of the residues in
$w$ at $0$, all $y_i$'s, and at $z$ if $t>s$;
then take the sum of the residues of the resulting expression
in 
$z$ at 0 and all $\beta_r$'s.

	\begin{figure}[htpb]
		\centering
		\includegraphics[width=.55\textwidth]{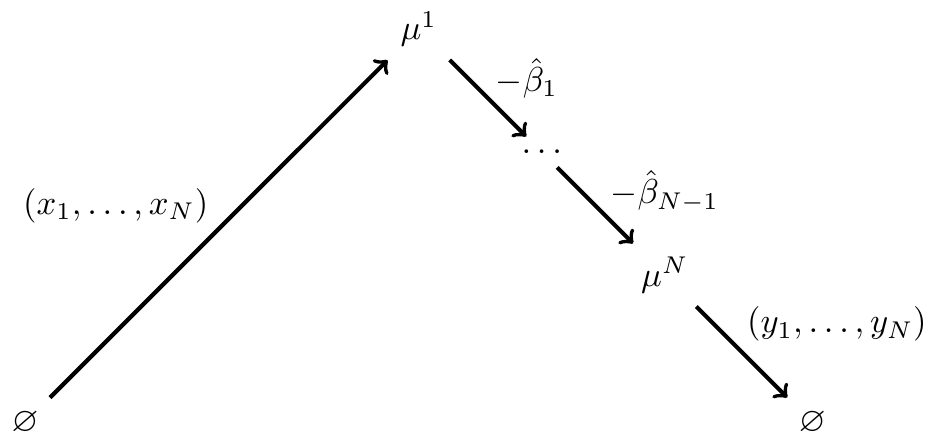}
		\caption{Graphical representation of the Schur process 
		\eqref{eq:Schur_process_from_Grothendieck}.
        Arrows indicate 
	    the diagram inclusion relations.}
		\label{fig:Schur_process}
	\end{figure}

\begin{proof}[Proof of \Cref{prop:Schur_correlation_kernel_via_integrals}]
	It is well-known 
	that 
	determinantal correlation kernels for the 
	Schur measures and processes
	have double contour integral form,
	see 
	\cite{okounkov2001infinite},
	\cite{okounkov2003correlation}.
	Namely,
    the generating function for the Schur process kernel has the form
	\begin{equation}
        \label{eq:Schur_kernel_generating_function_OR}
		\sum_{a,b\in \mathbb{Z}}
		K^{\textnormal{Schur process}}(a,t;b,s)\ssp
		z^{a}w^{-b-1}=
		\frac{1}{z-w}\frac{\Phi(t,z)}{\Phi(s,w)},
	\end{equation}
	where $\Phi(t,z)$ are is a function which is read off from 
	the specializations in the Schur process,
	and the generating series in 
	\eqref{eq:Schur_kernel_generating_function_OR} is expanded differently depending on whether $t\le s$ or $t>s$.
	This difference in expansion is that we assume either $|z|>|w|$ or $|z|<|w|$, which can be ultimately traced back to the normal ordering of the fermionic operators $\psi(z),\psi^*(w)$ in the notation of \cite[Section~2.3.4]{okounkov2003correlation}.
    Formula \eqref{eq:Schur_kernel_generating_function_OR} is the same as
    \cite[Theorem~1]{okounkov2003correlation}, up to switching from half-integers to integers in the indices $a,b$.

    Let us remark that it is not immediate how to adapt the generating function \eqref{eq:Schur_kernel_generating_function_OR} to a particular specialization of the Schur process (that is, how to select the integration contours to pick out the correct coefficients). In general, one could use the contour integrals from \cite{Johansson2000} (see also \cite[Remark 2 after Theorem 5.3]{BorodinGorinSPB12}) or \cite[Theorem 2.2]{borodin2005eynard}, but here for convenience let us briefly record a ``user's manual'' for such an adaptation. There are three principles:
    \begin{enumerate}[$\bullet$]
			\item First, start with the Schur measures ($t=s$). 
				By \cite[Theorem~2]{okounkov2001infinite}, the contours must satisfy $|z|>|w|$ for $t=s$.
			\item Second, on the integration contours for all $t,s$, all denominators in the integrand should expand as geometric series in a natural way as $\frac{1}{1-\xi}=\sum_{n=0}^{\infty}\xi^n$.
			\item The first two principles allow to select the integration contours for $t=s$, and it only remains to determine their ordering ($|z|>|w|$ or
				$|z|<|w|$) for $t \ne s$.
				This is done by inspecting how the specializations of the Schur measures at $t=s$ change with $t$.
    \end{enumerate}

    Let us implement these principles for our Schur process \eqref{eq:Schur_process_from_Grothendieck}.
	Its
	graphical representation 
	is given in \Cref{fig:Schur_process}.
	Each $\mu^t$ distributed
	as the Schur measure with probability weights
	\begin{equation}
		\label{eq:Schur_measure}
		\propto
		s_{\mu^{t}}(x_1,\ldots,x_N )
		s_{\mu^t}(-\hat \beta_t,\ldots,-\hat \beta_{N-1};y_1,\ldots,y_N ).
	\end{equation}
	Here the notation $-\hat\beta_i$ 
	means that these are
	``dual'' parameters, that is, they correspond to 
	transpositions of the Young diagrams. Moreover,
	we unified a number of these ``dual'' parameters
	with the usual parameters $y_i$ in the 
	second Schur function
	(see, e.g., \cite[Section~2]{BorodinGorinSPB12} for details).
	The weights \eqref{eq:Schur_measure}
	follows from the weights
	\eqref{eq:Schur_process_from_Grothendieck}
	and the skew Cauchy identity for Schur functions.

	Now, from \eqref{eq:Schur_measure} and \cite[Theorem~2]{okounkov2001infinite}, we see that the functions in the integrand for $t=s$ are given by 
	\begin{equation}
		\label{eq:Kernel_Schur_proof_new}
		F_t(z)=\frac{H_{(x_1,\ldots,x_N )}(z)}
		{
		H_{-\hat \beta_t,\ldots,-\hat \beta_{N-1};
		y_1,\ldots,y_N }(z^{-1})}=
		\prod_{i=1}^{N}\frac{1-z^{-1}y_i}{1-zx_i}\ssp\prod_{r=t}^{N-1}
		\frac1{1-\beta_r z^{-1}},
	\end{equation}
	where $H_\rho(z)=\sum_{n\ge0}z^n\ssp s_{(n)}(\rho)$ 
	is the single-variable Cauchy kernel for a specialization
	$\rho$ of Schur functions. 
	The second equality in \eqref{eq:Kernel_Schur_proof_new}
	follows from the Cauchy identity, and is precisely the expression \eqref{eq:F_t_function_for_kernel} for $F_t(z)$.
	Thus, using the first two principles above, we get all the conditions on the contours in our \Cref{prop:Schur_correlation_kernel_via_integrals} for $t=s$. In particular, the second condition $|\beta_r|<|z|<x_i^{-1}$,
	$|w|>y_i$ follows from requiring the expansion of
	\begin{equation*}
		\frac{F_t(z)}{F_s(w)}=
		\prod_{i=1}^{N}\frac{1-z^{-1}y_i}{1-zx_i}
		\prod_{i=1}^{N}
		\frac
		{1-wx_i}
		{1-w^{-1}y_i}
		\ssp
		\frac{
		\prod_{r=s}^{N-1}
		(1-\beta_r w^{-1})
		}{
		\prod_{r=t}^{N-1}
		(1-\beta_r z^{-1})
		}
	\end{equation*}
	as geometric series.
	Observe that to get the integral \eqref{eq:Kernel_Schur},
	we also needed to shift the indices $(a,b)$
	by $N-\frac12$
	compared to formulas in \cite{okounkov2001infinite}, \cite{okounkov2003correlation}. Indeed, 
	in these references the point configuration associated to a partition 
	$\mu$ is $\left\{ \mu_i-i+\frac12 \right\}_{i\in \mathbb{Z}_{\ge1}}$,
	while we work with
	$\left\{ \mu_i+N-i \right\}_{i=1}^{N}$.

	Extending our formula \eqref{eq:Kernel_Schur} to $t \ne s$ in a natural way leaves only the 
	question of the ordering of the integration contours
	($|z|>|w|$ or $|z|<|w|$)
	for $t\ne s$. This
	can be resolved by comparing \eqref{eq:Kernel_Schur_proof_new} with \cite[(20)]{okounkov2003correlation}. We see that 
	$H_{-\hat \beta_t,\ldots,-\hat \beta_{N-1}; y_1,\ldots,y_N }(z^{-1})$ should be matched to the product $\prod_{m<t}\phi^+[m](z^{-1})$. 
	In the latter product, increasing $t$ will \emph{increase} the number of factors, which is opposite to how the number of factors depends on $t$ in \eqref{eq:Kernel_Schur_proof_new}. Thus, we must choose $|z|>|w|$ for $t\le s$, which is opposite to \cite[Theorem~1]{okounkov2003correlation}. This completes the proof.
\end{proof}

\section{Absence of determinantal structure}
\label{sec:non_determinantal}

The Grothendieck measure
$\mathscr{M}_{\vec\beta;\mathrm{Gr}}$
\eqref{eq:Grothendieck_measure_definition_complete} has
probability weights expressed as products of two
determinants. This structure is very similar to that of
biorthogonal ensembles
\eqref{eq:biorthogonal_ensembles_classical}, which are
well-known determinantal point processes. However, this
section shows that the Grothendieck measures are \emph{not}
determinantal point processes. This question is deeply
linked to the \emph{principal minor assignment problem} from
linear algebra and algebraic geometry.
We describe this problem 
in \Cref{sub:principal_minor},
and discuss its
long history in \Cref{sub:history}.
Then in
\Cref{sub:Nanson_det_test_derivation}, we present a
self-contained derivation of a determinantal test 
for minors of a $4\times 4$ matrix
originally
obtained by Nanson in 1897 \cite{nanson1897xlvi}, and in
\Cref{sub:higher_order_Nanson} we extend this test
to matrices of arbitrary size. Finally, in
\Cref{sub:applying_Nanson_test}, we apply the original
Nanson's test to show that the Grothendieck measures are not
determinantal.

\subsection{Principal minor assignment problem}
\label{sub:principal_minor}

Let $A$ be an $n\times n$ complex matrix. 
To it, we associate $2^{n}$
principal minors
$A_I=\det[A_{i_a,i_b}]_{a,b=1}^{|I|}$,
where $I$ runs over all subsets of $\left\{ 1,\ldots,n \right\}$, 
and $|I|$ is the number of elements in $I$. This includes the 
empty minor $A_{\varnothing}=1$. 
The map $\mathbb{C}^{n^2}\to \mathbb{C}^{2^n}$, 
$A\to ( A_I )_{I\subseteq \left\{ 1,\ldots,n \right\}}$,
is called the \emph{affine principal minor map}. 
The (affine) \emph{principal minor assignment problem}
\cite{holtz2002open}, \cite{lin2009polynomial}
aims to characterize the image 
under this map in $\mathbb{C}^{2^n}$.
Denote this image by $\mathcal{A}_n\subset \mathbb{C}^{2^n}$. 
This complex algebraic
variety is closed and has dimension $n^2-n+1$
\cite{stouffer1924independence},
\cite{lin2009polynomial}.

For $n\le 3$, the dimension of $\mathcal{A}_n$ is equal to $2^n-1$,
(full available dimension because $A_{\varnothing}=1$), but 
starting with $n=4$, $\mathcal{A}_n$ becomes very complicated.
Indeed, by \cite[Theorem 2]{lin2009polynomial},
the prime ideal of the (13-dimensional) variety $\mathcal{A}_4$ 
is minimally generated by 65 polynomials of degree 12 
in the $A_I$'s.

\medskip

Let us translate the principal minor assignment problem
into the language of point processes.
Let $\mathscr{M}$ be a point process on $\mathbb{Z}_{\ge0}$, that is, 
a
probability measure 
on point configurations in $\mathbb{Z}_{\ge0}$.
This measure
may have complex weights, 
but has to be normalized to have total probability mass $1$,
and has to be bounded in absolute value by a nonnegative probability
measure on point configurations in $\mathbb{Z}_{\ge0}$.
The base space for the point process may be arbitrary and 
is not necessarily finite, and here we take $\mathbb{Z}_{\ge0}$ 
for an easier future application. 
For each finite subset $I\subset\mathbb{Z}_{\ge0}$ consider the correlation function
\begin{equation*}
	\rho_I = \mathscr{M}\left( \textnormal{the random point configuration contains
	all points from $I$} \right).
\end{equation*}
It is natural to ask whether the 
point process $\mathscr{M}$ is determinantal, that is, whether
there exists a kernel $K(x,y)$, $x,y\in \mathbb{Z}_{\ge0}$,
such that for any finite $I\subset \mathbb{Z}_{\ge0}$ we have
$\rho_I=\det[K(a,b)]_{a,b\in I}$.
A clear necessary condition for the process to be determinantal
is as follows:

\begin{proposition}
	\label{prop:det_pp_and_minor_assignment}
	If the process $\mathscr{M}$ is determinantal, then 
	for any $n\ge1$ and any $n$-point subset $\mathsf{J}\subset \mathbb{Z}_{\ge0}$,
	the vector $(\rho_I\colon I\subseteq \mathsf{J})
	\in \mathbb{C}^{2^n}$ belongs to the image $\mathcal{A}_n$ under the 
	principal minor map.
\end{proposition}

Thus, 
if for some $n$ and some $n$-point $\mathsf{J}\subset \mathbb{Z}_{\ge0}$ the vector
$(\rho_I\colon I\subseteq \mathsf{J})
\in \mathbb{C}^{2^n}$ \emph{does not} belong to $\mathcal{A}_n$, then 
the process $\mathscr{M}$ is \emph{not determinantal}.
Due to the complicated nature of $\mathcal{A}_n$ for $n\ge4$, 
checking that a vector belongs to $\mathcal{A}_n$
is hard. However, to show that some vector $(\rho_I)$
\emph{does not} belong to $\mathcal{A}_n$,
it suffices to find a polynomial in the ideal of $\mathcal{A}_n$ that does not 
vanish on $(\rho_I)$. This leads us to the following definition:

\begin{definition}
	\label{def:determinantal_test}
	Fix $n\ge4$. A \emph{determinantal test} of
	order $n$ is any element in the ideal of $\mathcal{A}_n$, that is,
	a polynomial 
	in the indeterminates $(A_I\colon I\subset \left\{ 1,\ldots,n \right\})$
	which vanishes if the $A_I$'s are principal minors of some matrix $A$.
\end{definition}

Thus, to show that the process $\mathscr{M}$ is not
determinantal, it suffices to show that there exists
$\mathsf{J}\subset \mathbb{Z}_{\ge0}$ and a determinantal
test which does not vanish on the vector $(\rho_I\colon I
\subset \mathsf{J})$. Let us describe an example of such a
test of order $4$ which we call the \emph{Nanson's test}
as it first appeared in 1897 
in \cite{nanson1897xlvi}.
First, we need another definition:
\begin{definition}
	\label{def:cluster_cycle_sum}
	Let 
	$A=(a_{ij})$ be a complex $n\times n$ matrix, 
	and fix
	$I\subseteq \left\{ 1,\ldots,n \right\}$ with $|I|=k\ge2$.
	For a $k$-cycle $\pi\in S_n$ with support $I$ (there are $(k-1)!$ such cycles), 
	define
	$t_\pi(A) \coloneqq \prod_{i\colon i\ne \pi(i)}a_{i,\pi(i)}$.
	Let the \emph{cycle-sum}
	\cite{lin2009polynomial}
	be 
	\begin{equation}
		\label{eq:cycle_sum}
		T_I \coloneqq \sum_{\textnormal{all $k$-cycles $\pi$
		with support $I$}}t_\pi(A).
	\end{equation}
\end{definition}

The cycle-sums 
are the same as 
\emph{cluster functions} in the terminology of 
\cite{tracy1998correlation}, and they can be expressed through the 
principal minors $A_I$ as follows:
\begin{equation}
	\label{eq:T_via_principal_minors}
	T_I=\sum_{I=I_1\sqcup \ldots\sqcup I_m }(-1)^{k+m}(m-1)!\ssp
	A_{I_1}\cdots 
	A_{I_m},
\end{equation}
where the sum is taken over all set partitions of $I$ into exactly $m$
nonempty parts.
For example, 
\begin{equation*}
	\begin{split}
		T_{ \left\{ 1,2,3 \right\}}
		&=a_{12}a_{23}a_{31}+a_{13}a_{21}a_{32}
		\\&=
		2
		A_{ \left\{ 1 \right\}}
		A_{ \left\{ 2 \right\}}
		A_{ \left\{ 3 \right\}}
		-\left( 
			A_{ \left\{ 1 \right\}}A_{ \left\{ 2,3 \right\}}
			+
			A_{ \left\{ 2 \right\}}A_{ \left\{ 1,3 \right\}}
			+
			A_{ \left\{ 3 \right\}}A_{ \left\{ 1,2 \right\}}
		\right)
		+
		A_{ \left\{ 1,2,3 \right\}}.
	\end{split}
\end{equation*}

\begin{definition}
	\label{def:Nanson}
	The \emph{Nanson's determinantal test}
	is a polynomial
	$\mathfrak{N}_4$ of order $4$
	in the indeterminates $T_I$
	which has the form
	\begin{equation}
		\label{eq:Nanson_test}
		\mathfrak{N}_{4}=
		\frac{1}{2}\det
		\begin{bmatrix}
			T_{123} T_{14} & T_{124} T_{13} & 
			T_{134} 
			T_{12} 
			& 
			2 T_{12} T_{13} T_{14} T_{234}+T_{123} T_{124} T_{134} 
			\\
			T_{124} T_{23} & T_{123} T_{24} &
			T_{234} 
			T_{12} 
			& 
			2 T_{12} T_{23} T_{24}T_{134} +T_{123} T_{124} T_{234} 
			\\
			T_{134} T_{23} &  T_{234}T_{13} & T_{123} T_{34} &
			2  T_{13} T_{23} T_{34}T_{124} 
			+
			T_{123} T_{134} T_{234}
			\\
			T_{234}T_{14} & T_{134} T_{24} & T_{124} T_{34} & 
			2 T_{14} T_{24} T_{34} T_{123}+T_{124} T_{134} T_{234} 
		\end{bmatrix},
	\end{equation}
	where we abbreviated $T_{ij}=T_{ \left\{ i,j \right\} }$, and so on.
	By \eqref{eq:T_via_principal_minors},
	$\mathfrak{N}_4$ is also a 
	polynomial in the indeterminates $A_I$.
\end{definition}

One readily verifies
(for example, using computer algebra)
that $\mathfrak{N}_4$ is indeed a determinantal test:

\begin{proposition}
	\label{prop:Nanson_test_works}
	If for all $I\subseteq \left\{ 1,2,3,4 \right\}$
	we replace the indeterminates
	$T_I$'s in \eqref{eq:Nanson_test} by the 
	cycle-sums 
	\eqref{eq:cycle_sum}
	coming from a 
	$4\times 4$ matrix $A$, then the polynomial
	$\mathfrak{N}_4$ \eqref{eq:Nanson_test} vanishes identically.
\end{proposition}

We apply Nanson's test $\mathfrak{N}_4$ to show that the
Grothendieck measures are not determinantal in
\Cref{sub:applying_Nanson_test} below. 

\subsection{On the history of the principal minor assignment problem}
\label{sub:history}

Let us briefly
discuss the rich history and variants of the principal minor
assignment problem. Within this history, we can observe at
least two instances where similar questions were
independently formulated and addressed in the context of
algebra (on the original principal minor assignment) and
probability (concerning determinantal processes). 
We hope that 
these two research avenues will become increasingly
aware of one another.

\medskip

The problem itself dates to the late 19th century work of 
MacMahon \cite{MacMahon1894}, with initial results 
due to Muir \cite{Muir1894} \cite{Muir1898} and 
Nanson \cite{nanson1897xlvi}.
In particular, Nanson 
has partially solved the $4\times 4$ principal
minor assignment problem, and
obtained the determinantal test 
$\mathfrak{N}_4$ 
\eqref{eq:Nanson_test}.
He also obtained four other tests
algebraically independent from
$\mathfrak{N}_4$ (which
enter the list of 65 polynomials in Lin--Sturmfels
\cite{lin2009polynomial}).
The question of relations on principal minors 
is investigated by Stouffer \cite{stouffer1924independence},
and in particular
he showed that the dimension of $\mathcal{A}_n$ is $n^2-n+1$.

\medskip

Another question related to the principal minor assignment
problem, when it has a solution, concerns the relationship
between two $n\times n$ complex matrices $A,B$ with the same
principal minors. Under various natural conditions, it has
been shown that the matrix $A$ should be diagonally
conjugate either to $B$, or to $B^{\mathrm{transpose}}$. Here
``diagonally conjugate'' means $A=DBD^{-1}$, where
$D$ is a nondegenerate diagonal matrix. This question was
first addressed in the context of the principal minors
assignment problem by Loewy \cite{loewy1986principal}. More
recently, Stevens \cite{stevens2021equivalent} and Mantelos
\cite{mantelos2023classification} investigated essentially
the same question within the context of determinantal
processes, seemingly unaware of Loewy's work.

\medskip

Griffin--Tsatsomeros \cite{griffin2006principal} proposed
algorithms for finding the solution
of the principal minor assignment problem (that is, the matrix $A$),
which are computationally fast for
particular subclasses of matrices. While this does not
yield explicit polynomial determinantal tests, an algorithm 
can be used to (numerically) demonstrate that a
point process is not determinantal.
In our application to Grothendieck measures in \Cref{sub:applying_Nanson_test} below
we do not use an algorithm like in \cite{griffin2006principal}, but rather perform a symbolic computation 
based on the Nanson's test $\mathfrak{N}_4$.

\medskip

A particularly well-understood case of the
principal minor assignment problem assumes that the initial
complex $n\times n$ matrix $A$ is Hermitian symmetric.
Holtz--Sturmfels \cite{holtz2007hyperdeterminantal} and
Oeding \cite{oeding2011set} use the additional
hyperdeterminantal structure of the variety formed by
principal minors to solve the assignment problem
set-theoretically. More recently, Al Ahmadieh and Vinzant
\cite{alahmadieh2021characterizing},
\cite{alahmadieh2022determinantal} considered the principal
minor assignment problem over other rings and explored
connections to stable polynomials. These latter works
represent the current state of the art of the principal
minor assignment problem from an algebraic perspective. In
particular, \cite[Theorem~8.1]{alahmadieh2022determinantal} is a strong and unexpected
negative result.

\medskip

Finally, let us mention that there are several natural
generalizations of the principal minor assignment problem,
as considered by Borodin--Rains \cite[Section 4]{borodin2005eynard} 
and independently by Lin--Sturmfels
\cite{lin2009polynomial} (unaware at the time
of the work by Borodin--Rains). 
These variants allow more general
\emph{conditional} and/or \emph{Pfaffian} structure of the
correlation functions $\rho_I$. A
conditional determinantal process, by definition, has
correlation functions $\rho_I=\det\left[ K(a,b)
\right]_{a,b\in I\cup S}$, where $S=\left\{ n+1,\ldots,n+m
\right\}$, and $I \subseteq\left\{ 1,\ldots, n \right\}$.
In other words, it is a usual determinantal process on
$\left\{ 1,\ldots,n,n+1,\ldots,n+m  \right\}$ conditioned to
have particles at each of the points $n+1,\ldots,n+m $.  In
the terminology of \cite{lin2009polynomial}, the conditional
determinantal structure is the same as the \emph{projective
principal minor assignment problem}, a more natural setting
for algebraic geometry.  The projective variety analog of
$\mathcal{A}_n$ for $n=4$ is more complicated, with 718
generating polynomials.  The Pfaffian and conditional
Pfaffian structures (considered in \cite{borodin2005eynard};
they are motivated, in particular, by real and quaternionic
random matrix ensembles) are defined similarly to the
determinantal ones, but with determinants replaced by
Pfaffians. 
The $n=4$ conditional Pfaffian (projective)
variety analog of $\mathcal{A}_n$
is even more complicated than the determinantal one,
and experimentation suggests \cite{borodin2005eynard} that a
corresponding test could have degree 1146.

It would be interesting to develop determinantal and
Pfaffian tests for conditional processes (as well as for
further generalizations involving, for example,
$\alpha$-determinants and permanents), but we leave these
directions for future work.

\subsection{A self-contained derivation of Nanson's determinantal test}
\label{sub:Nanson_det_test_derivation}

Here we present a self-contained derivation of Nanson's
determinantal test polynomial $\mathfrak{N}_4$
\eqref{eq:Nanson_test}. This argument differs slightly from
Nanson's original work \cite{nanson1897xlvi} and was
obtained independently by the second author (unaware of the
principal minor assignment problem) over a decade ago
\cite{Petrov2011det_test_unpublished}. Here we see another
instance of the disconnect between the principal minor
assignment problem and determinantal processes
(complementing the two cases discussed in \Cref{sub:history}
above). In \Cref{sub:higher_order_Nanson} below, we discuss
how our derivation of $\mathfrak{N}_4$ can be adapted to
obtain Nanson-like higher-order determinantal tests.

We aim to explain where the polynomial $\mathfrak{N}_4$
\eqref{eq:Nanson_test} comes from. Checking that it is
indeed a determinantal test is a direct verification
(\Cref{prop:Nanson_test_works}), and we do not focus on this
here.

Assume that we are given the cluster functions $T_I$
\eqref{eq:T_via_principal_minors}, where $I$ runs over
subsets of $\left\{ 1,2,3,4 \right\}$ with $\ge2$ elements.
The $T_I$'s are polynomials in the minors $A_I$, but
working with the $T_I$'s is much more convenient. Let us use
the $T_I$'s to try finding the matrix elements $a_{ij}$ of
the original matrix $A$.

Throughout the rest of this section, we will abbreviate expressions
like $T_{ \left\{ 1,2 \right\}}$ as $T_{12}$. Note that all the 
$T_I$'s are symmetric in the indices.
Assume that $a_{1i}\ne 0$
for all $i=2,3,4$, and conjugate the matrix by the diagonal
matrix with the entries
$d_i=\mathbf{1}_{i=1}+a_{1i}\mathbf{1}_{i\ne 1}$. Then we
have for the conjugated matrix (denoted by $\tilde A=(\tilde a_{ij})$):
\begin{equation}
	\label{eq:conjugated_correlation_matrix}
	\tilde a_{1i}=\frac{d_1}{d_i}\ssp a_{1i}=1,\qquad i=2,3,4.
\end{equation}
With this notation, we have $T_{1i}=\tilde a_{i1}$ and
\begin{equation}
 \label{eq:equations_for_kernel}
 T_{ij}=\tilde a_{ij}\tilde a_{ji},\qquad 
 T_{1ij}=\tilde a_{ij}T_{1j}+\tilde a_{ji}T_{1i}.
\end{equation}
The second identity in \eqref{eq:equations_for_kernel}
is by the definition of the cluster
function \eqref{eq:cycle_sum}, simplified thanks to
\eqref{eq:conjugated_correlation_matrix}.
Equations \eqref{eq:equations_for_kernel} allow to find
$\tilde a_{ij}T_{1j}$ and $\tilde a_{ji}T_{1i}$ as two
distinct roots of a quadratic equation. We thus have
\begin{equation}
 \label{eq:equations_for_kernel_solution_quadratic}
 \tilde a_{ij}T_{1j}=\frac{T_{1ij}+R_{ij}}{2},
 \qquad 
 \tilde a_{ji}T_{1i}=\frac{T_{1ij}- R_{ij}}{2},
\end{equation}
where we denoted $R_{ij}\coloneqq \pm
\sqrt{T_{1ij}^2-4T_{1i}T_{1j}T_{ij}}$. Observe that $R_{ij}$
contains an unknown sign that we cannot determine a priori
(it may also depend on $i$ and $j$), but up to sign $R_{ij}$
is symmetric in $i,j$. 

Let us substitute
\eqref{eq:equations_for_kernel_solution_quadratic} into the following
identity (which is again an instance of
\eqref{eq:cycle_sum}):
\begin{equation*}
 T_{234}=\tilde a_{23}\tilde a_{34}\tilde a_{42}+
 \tilde a_{24}\tilde a_{43}\tilde a_{32}.
\end{equation*}
As a result, we obtain the following identity involving
three
square roots $R_{23},R_{34},R_{24}$ with unknown signs:
\begin{equation}
 \label{eq:nanson4_derivation_1}
 \begin{split}
 8T_{12}T_{13}T_{14}T_{234}&=
 \left( T_{123}+R_{23} \right)
 \left( T_{124}+R_{24} \right)
 \left( T_{134}+R_{34} \right)
 \\&\hspace{80pt}+
 \left( T_{123}-R_{23} \right)
 \left( T_{124}-R_{24} \right)
 \left( T_{134}-R_{34} \right).
 \end{split}
\end{equation}
Note that \eqref{eq:nanson4_derivation_1} does not contain
the matrix elements $\tilde a$. Thus, it is an algebraic (but not
yet polynomial) identity on the cluster functions $T_I$.
Simplifying \eqref{eq:nanson4_derivation_1}, we see that
\begin{equation}\label{eq:nanson4_derivation_2}
 4T_{12}T_{13}T_{14}T_{234}-T_{123}T_{124}T_{143}
 -T_{123}R_{24}R_{34}
 -
 T_{124}R_{23}R_{34}
 -
 T_{134}R_{23}R_{24}
 =0.
\end{equation}
The left-hand side contains three summands with
irrationalities $R_{24}R_{34}, R_{23}R_{34}$, and
$R_{23}R_{24}$ with uncertain signs. By choosing all
possible eight combinations of the signs for
$R_{23},R_{34},R_{24}$, we see that there are only four
possible combinations of signs in
\eqref{eq:nanson4_derivation_2}. Thus, by multiplying
together all these four expressions with different signs, we can
get rid of irrationality and obtain a polynomial in the
$T_I$'s:
\begin{equation}
 \label{eq:nanson4_derivation_3_final}
	\begin{split}
		&
		\bigl(
		4T_{12}T_{13}T_{14}T_{234}-T_{123}T_{124}T_{143}
		-T_{123}R_{24}R_{34}
		-
		T_{124}R_{23}R_{34}
		-
		T_{134}R_{23}R_{24}
		\big)
		\\&\hspace{10pt}
		\times
		\bigl(
		4T_{12}T_{13}T_{14}T_{234}-T_{123}T_{124}T_{143}
		+T_{123}R_{24}R_{34}
		+
		T_{124}R_{23}R_{34}
		-
		T_{134}R_{23}R_{24}
		\big)
		\\&\hspace{10pt}
		\times
		\bigl(
		4T_{12}T_{13}T_{14}T_{234}-T_{123}T_{124}T_{143}
		-T_{123}R_{24}R_{34}
		+
		T_{124}R_{23}R_{34}
		+
		T_{134}R_{23}R_{24}
		\big)
		\\&\hspace{10pt}
		\times
		\bigl(
		4T_{12}T_{13}T_{14}T_{234}-T_{123}T_{124}T_{143}
		+
		T_{123}R_{24}R_{34}
		-
		T_{124}R_{23}R_{34}
		+
		T_{134}R_{23}R_{24}
		\big)
		=0.
	\end{split}
\end{equation}
Clearly, expanding the left-hand side of
\eqref{eq:nanson4_derivation_3_final} 
squares all the quantities $R_{ij}$.
Thus, the resulting identity is \emph{polynomial}
in the $T_I$'s, and, moreover, 
all unknown signs present in the $R$'s 
disappear.

One can check (for example, using
computer algebra) that
the resulting polynomial 
\eqref{eq:nanson4_derivation_3_final}
in the $T_I$'s has 19 summands,
and it is symmetric in the indices $1,2,3,4$.
One can also verify
that this polynomial
divided by the common factor $256T_{12}^2T_{13}^2T_{14}^2$
is exactly the same as
the Nanson's test $\mathfrak{N}_4$ \eqref{eq:Nanson_test}. 
This concludes our derivation of the Nanson's determinantal test of order four.

\subsection{Procedure for higher-order Nanson tests}
\label{sub:higher_order_Nanson}

Adapting the derivation of the test $\mathfrak{N}_4$ given
in \Cref{sub:Nanson_det_test_derivation} above,
one can
produce concrete determinantal tests $\mathfrak{N}_{n}$
for minors of general $n\times n$ matrices, 
where $n\ge 4$. Let us explain the necessary steps for general $n$
without going into full detail.
We have from \eqref{eq:cycle_sum}:
\begin{equation}
	\label{eq:T_via_principal_minors_0}
	T_{2,3,\ldots,n }=
	\sum_{\textnormal{$(n-1)$-cycles $\sigma$ on $\{2,\ldots,n\}$}}
	\tilde a_{\sigma(2)\sigma(3)}\ldots \tilde a_{\sigma(n-1)\sigma(n)}\tilde
	a_{\sigma(n)\sigma(2)}.
\end{equation}
For every $i<j$, let us substitute 
the solutions \eqref{eq:equations_for_kernel_solution_quadratic},
so \eqref{eq:T_via_principal_minors_0} becomes
\begin{equation}
	\label{eq:T_general_Nansen_as_sum_of_products}
	2^{n-1} T_{12}\ldots T_{1n} T_{2,\ldots,n}
	=
	\sum_{\sigma}
	\prod_{i=2}^{\circlearrowright n}\left( T_{1\sigma(i)\sigma(i+1)}+(-1)^{\mathbf{1}_{\sigma(i)>\sigma(i+1)}}
	R_{\sigma(i)\sigma(i+1)} \right).
\end{equation}
Here the sum is also over $(n-1)$-cycles $\sigma$ on $\left\{ 2,\ldots,n  \right\}$, 
and ``$\circlearrowright$''
means that the 
product is cyclic in the sense that $n+1$ is identified with $2$.
We see that \eqref{eq:T_general_Nansen_as_sum_of_products}
is an algebraic identity 
on the $T_I$'s which does not contain the matrix
elements $\tilde a_{ij}$.

Opening up the parentheses in
\eqref{eq:T_general_Nansen_as_sum_of_products}, one readily
sees that all terms
with an odd number of the factors $R_{ij}$ cancel out, while
the terms with an even number of the factors $R_{ij}$ appear twice.
Therefore, we can continue
\eqref{eq:T_general_Nansen_as_sum_of_products} as
\begin{equation}
	\label{eq:T_general_Nansen_as_sum_of_products_2}
	2^{n-2} T_{12}\ldots T_{1n} T_{2,\ldots,n}
	-
	\sum_{\substack{\textnormal{ 
	non-oriented $(n-1)$-cycles}\\\textnormal{$\tau$ on $\left\{
	2,\ldots,n  \right\}$}}}
	\ \prod_{i=2}^{\circlearrowright n}
	T_{1\tau(i)\tau(i+1)}=\mathsf{RHS}.
\end{equation}
Here $\mathsf{RHS}$ is a sum over non-oriented $(n-1)$ cycles
$\tau$
on $\left\{ 2,\ldots,n  \right\}$, where the summands are
$(n-1)$-fold cyclic products of the quantities $T_{1ij}$ and
$R_{ij}$ with a nonzero even number of the $R$'s, and each
such monomial has coefficient $\pm1$.
More precisely, the sign is determined by the number
of descents $\tau(i)>\tau(i+1)$ in $\tau$ for which the 
monomial contains $R_{\tau(i)\tau(i+1)}$ (and not $T_{1\tau(i)\tau(i+1)}$).

Next, in $\mathsf{RHS}$ there are 
$\binom{n-1}{2}$ possible elements $R_{ij}$, and each of them contains
an unknown sign in front of the square root. 
Let us take the product over those of the $2^{\binom {n-1}2}$ 
possible sign combinations which lead to the different $\mathsf{RHS}$'s.
Expanding this product
removes all irrationalities and unknown signs, and 
produces a polynomial (denoted by $\mathfrak{N}_n$)
in the cluster functions $T_I$,
where $I$ runs over subsets of $\left\{ 1,\ldots,n  \right\}$
with $\ge2$ elements.
We call $\mathfrak{N}_n$ the 
\emph{Nanson-like determinantal test of order $n$}.

For example, for $n=5$ identity
\eqref{eq:T_general_Nansen_as_sum_of_products_2}
has the form
(recall that the quantities $T_{1ij}$ and $R_{ij}$ 
are symmetric in $i,j$):
\begin{align*}
		&8T_{12}T_{13}T_{14}T_{15}T_{2345}
		-T_{24}T_{43}T_{35}T_{52}
		-T_{23}T_{34}T_{45}T_{52}
		-T_{23}T_{35}T_{54}T_{42}
		\\
		&\hspace{60pt}=R_{24} R_{25} R_{34} R_{35} - R_{23} R_{25} R_{34} R_{45} + R_{23} R_{24} R_{35} R_{45} \\
		&\hspace{80pt}+ R_{34} R_{45} T_{123} T_{125} + R_{24} R_{35} T_{125} T_{134} + R_{23} R_{45} T_{125} T_{134}\\
		&\hspace{80pt}+ R_{24} R_{45} T_{123} T_{135} + R_{25} R_{34} T_{124} T_{135} + R_{23} R_{35} T_{124} T_{145}  \\
		&\hspace{80pt} + R_{23} R_{34} T_{125} T_{145} - R_{23} R_{25} T_{134} T_{145} - R_{23} R_{24} T_{135} T_{145}\\
		&\hspace{80pt}- R_{35} R_{45} T_{123} T_{124} - R_{34} R_{35} T_{124} T_{125} - R_{25} R_{45} T_{123} T_{134}\\
		&\hspace{80pt}- R_{25} R_{35} T_{124} T_{134} - R_{23} R_{45} T_{124} T_{135}- R_{24} R_{34} T_{125} T_{135}		\\
		&\hspace{80pt} - R_{24} R_{25} T_{134} T_{135} - R_{25} R_{34} T_{123} T_{145}- R_{24} R_{35} T_{123} T_{145}.
\end{align*}
In the right-hand side, there are 
$2^{\binom 42}=64$ possible signs in the $R_{ij}$'s, 
but they lead to ``only''
32 distinct identities.
Multiplying all these $32$ expressions
similarly to \eqref{eq:nanson4_derivation_3_final}
and recalling the definition of the $R$'s
leads to a polynomial in the $T_I$'s with no irrationality.
This produces the determinantal test $\mathfrak{N}_5$.

\subsection{Application to Grothendieck measures and proof of \texorpdfstring{\Cref{thm:non_determinantal_intro}}{Theorem}}
\label{sub:applying_Nanson_test}

In this subsection we employ the Nanson
determinantal test $\mathfrak{N}_4$ to 
prove \Cref{thm:non_determinantal_intro} from Introduction.
That is, we will show that
the Grothendieck measure on partitions
$\mathscr{M}_{\vec \beta; \mathrm{Gr}}(\lambda)$
\eqref{eq:Grothendieck_measure_definition_complete}
is not determinantal 
as a point process on $\mathbb{Z}_{\ge0}$
with points $\ell_j=\lambda_N+N-j$, $j=1,\ldots,N$.

We focus on the case $N=2$ and look
at correlations $\rho_I^{\mathrm{Gr}}$ of the random point configuration
$\{\ell_1,\ell_2\}=\{\lambda_1+1,\lambda_2\}$ 
for all subsets
$I\subseteq \left\{ 0,1,2,3 \right\}$. 
Moreover, we will set $\beta_1=\beta$, $x_1=x_2=x$, and
$y_1=y_2=y$.
Clearly, $\rho^{\mathrm{Gr}}_I=0$ if $|I|=3$ or $4$. 
Moreover, we have $\rho^{\mathrm{Gr}}_{\varnothing}=1$, and 
for two-point subsets we have (where $i>j$):
\begin{equation}
	\label{eq:rho_Gr_2_point}
	\begin{split}
		\rho^{\mathrm{Gr}}_{ \{i,j \} }
		&=
		\mathscr{M}_{\vec \beta; \mathrm{Gr}}\bigl( (i-1,j) \bigr)
		\\&=
		\frac
		{(1-x y)^4}
		{(1-x\beta )^2}\ssp
		\ssp
		 x^{i+j-1} y^{i+j-2} (\beta  x (i-j-1)-i+j) ((j-i) (y-\beta )-\beta),
	\end{split}
\end{equation}
where we used 
\eqref{eq:Grothendieck_polynomials}--\eqref{eq:Grothendieck_measure_definition_complete},
and took the limits as $x_2\to x_1=x$ and $y_2\to y_1=y$.

To compute one-point correlations, 
we employ
\Cref{prop:correlations_tilted}
and 
the correlation kernel $K^{2d}_{\vec\beta;\mathrm{Gr}}$
of the ambient Schur process
(\Cref{prop:Schur_correlation_kernel_via_integrals}).
We have by \Cref{prop:correlations_tilted} (specifically,
by its particular case \eqref{eq:one_point_correlation})
\begin{equation}
	\label{eq:rho_Gr_1_point_via_Fredholm}
	\begin{split}
		\rho^{\mathrm{Gr}}_{ \left\{ i \right\}}&=
		[z^{1}w]\det
		\Bigl( 
			\mathbf{1}
			-
			(1-z)\chi_{[0,i]}K^{2d}_{\vec\beta;\mathrm{Gr}}(\cdot,1;\cdot,1) \chi_{[0,i)}
			+
			(1-w)
			\chi_{[0,i]}K^{2d}_{\vec\beta;\mathrm{Gr}}(\cdot,1;\cdot,1) \chi_{ \left\{ i \right\}}
		\Bigr)
		\\&\hspace{20pt}+
		[z^{0}w]\det
		\Bigl( 
			\mathbf{1}
			-
			(1-z)\chi_{[0,i]}K^{2d}_{\vec\beta;\mathrm{Gr}}(\cdot,2;\cdot,2) \chi_{[0,i)}
			+
			(1-w)
			\chi_{[0,i]}K^{2d}_{\vec\beta;\mathrm{Gr}}(\cdot,2;\cdot,2) \chi_{ \left\{ i \right\}}
		\Bigr)
		.
	\end{split}
\end{equation}
This yields
formulas for
$\rho^{\mathrm{Gr}}_{ \left\{ i \right\}}$, $i=0,1,2,3$, 
namely,
\begin{align}
	\nonumber
	\rho^{\mathrm{Gr}}_{ \left\{ 0 \right\}}&= 
	1-x^2 y^2
	;
	\\
	\nonumber
	\rho^{\mathrm{Gr}}_{ \left\{ 1 \right\}}
	&=
	x^2 y^2(1-x^2y^2)+\frac{(1-x y)^4}{(1-\beta  x)^2}
	;
	\\
	\label{eq:rho_Gr_1_point}
	\rho^{\mathrm{Gr}}_{ \left\{ 2 \right\}}
	&=
	x^4y^4(1-x^2y^2)+
	\frac{x (1-x y)^4 \left(\beta  (\beta  x-2)+x y^2+y (4-2 \beta  x)\right)}{(1-\beta  x)^2}
	;
	\\
	\nonumber
	\rho^{\mathrm{Gr}}_{ \left\{ 3 \right\}}
	&=
	x^6 y^6(1-x^2y^2)
	\\
	\nonumber
	&\hspace{15pt}+
	\frac{x^2 y (1-x y)^4 \left(x^2 y^3+y \left(\beta ^2 x^2-8 \beta  x+9\right)+2 \beta  (2 \beta  x-3)-2 x y^2 (\beta  x-2)\right)}{(1-\beta  x)^2}
	.
\end{align}
\begin{remark}
	\label{rmk:no_Schur_process_needed}
	These one-point 
	correlation functions
	$\rho^{\mathrm{Gr}}_{ \left\{ i \right\}}$
	can be also computed without using the finite-dimensional Fredholm-like 
	determinants \eqref{eq:rho_Gr_1_point_via_Fredholm}.
	Namely, 
	\begin{equation*}
		\rho^{\mathrm{Gr}}_{ \left\{ i \right\}}
		=
		\sum_{j=0}^{i-1}
		\mathscr{M}_{\vec \beta; \mathrm{Gr}}\bigl( (i-1,j) \bigr)
		+
		\sum_{j=i}^{\infty}
		\mathscr{M}_{\vec \beta; \mathrm{Gr}}\bigl( (j,i) \bigr),
	\end{equation*}
	and the infinite sum 
	is explicit because the summands have the form
	\eqref{eq:rho_Gr_2_point}.
	However, this simplification 
	of correlation functions
	works only for small $N$ and small order of correlation functions.
	We will use the full two-dimensional determinantal kernel
	to obtain asymptotics of Grothendieck random partitions
	in \Cref{sec:Grothendieck_measures_definition_discussion}
	below.
\end{remark}

Plugging the correlation functions 
\eqref{eq:rho_Gr_2_point},
\eqref{eq:rho_Gr_1_point}
into the Nanson test 
\eqref{eq:Nanson_test}
(with the help of the representation 
of cluster functions via minors
\eqref{eq:T_via_principal_minors}), we find
\begin{equation}
	\label{eq:N4_concrete_for_Grothendieck}
	\mathfrak{N}_4=
	\beta ^4 x^{34} y^{30} (1-x y)^{42} (1+x y)^2 P_{14}(xy)
	+O\left(\beta ^5\right),\qquad \beta\to0,
\end{equation}
where $P_{14}(xy)$ is a certain degree $14$
polynomial in the single variable $xy$.
We see that $\mathfrak{N}_4$ vanishes at $\beta=0$, as it should be
because then the Grothendieck measure reduces to the
Schur measure which is determinantal. 
On the other hand, for $\beta\ne0$ the test does not vanish in general. For example,
at $x=y=1/2$, we have
\begin{equation*}
	\begin{split}
		&\mathfrak{N}_4=
		\frac{(\beta -4) (\beta -1) \beta ^4}{(\beta -2)^{32}}
		\ssp
		Q_3(\beta)
		Q_5(\beta)
		Q_7(\beta)
		Q_9(\beta),
		\\
		&\frac
		{(\beta-2 )^{32}}
		{(\beta-4 ) (\beta-1 ) \beta ^4}
		\ssp
		\mathfrak{N}_4\Big\vert_{x=y=1/2,\ \beta=-1}\approx
		0.00005021> 0,
	\end{split}
\end{equation*}
where $Q_3,Q_5,Q_7,Q_9$ are certain polynomials in $\beta$ of degrees $3,5,7,9$, respectively.
Since the Nanson determinantal test does not vanish for these values of $x,y,\beta$, this implies 
\Cref{thm:non_determinantal_intro}
from Introduction.

\section{Limit shape of Grothendieck random partitions}
\label{sec:limit_shape}

In this section we employ the standard steepest descent
asymptotic analysis of the correlation kernel
of the Schur process (for example, explained in 
\cite[Section 3]{Okounkov2002})
to derive the limit shape result for Grothendieck
random partitions.

\subsection{Limit shape of the Schur process}
\label{sub:Schur_limit_shape}

In this subsection we assume that all the parameters
$x_i,y_j,\beta_r$ are homogeneous and satisfy
\eqref{eq:tableau_formula_for_G_and_easy_positivity}, that is,
for all $i,j,r$ we have
\begin{equation}
	\label{eq:xyb_homogeneous_parameter}
	x_i=x,\quad y_j=y,\quad \beta_r=\beta;
	\qquad x>0,\quad
	y>0,\quad xy<1,\quad \beta<0.
\end{equation}
We require the parameters be nonzero, otherwise the measure 
degenerates and may not produce asymptotic limit shapes.

Under conditions \eqref{eq:xyb_homogeneous_parameter}, the Schur process 
$\mathscr{M}^{2d}_{\vec \beta;\mathrm{Gr}}$
\eqref{eq:Schur_process_from_Grothendieck}
is 
a
well-defined probability measure on integer
arrays $X^{2d}=\{x_i^m\colon 1\le m,i\le N\}$.
With each such array, we associate a \emph{height function}
on $\mathbb{Z}_{\ge0}\times \left\{ 1,\ldots,N  \right\}$ as follows:
\begin{equation}
	\label{eq:prelimit_height_function}
	H_N(a,t)\coloneqq \#\{j\colon x^t_j\ge a \},
	\qquad a\in \mathbb{Z}_{\ge0},\quad t=1,\ldots,N .
\end{equation}
In words, $H_N(a,t)$ is the number of particles of the configuration $X^{2d}$
at level $t$
which are to the right of $a$.
In particular, we have
\begin{equation}
	\label{eq:H_N_condition_for_x_i_i}
	H_N(x^t_t,t)=t,\qquad t=1,\ldots,N .
\end{equation}

The following limit shape result for the Schur process can
be obtained in a standard manner via the steepest descent
analysis of the correlation kernel
$K_{\vec\beta;\mathrm{Gr}}^{2d}$
\eqref{eq:Kernel_Schur}--\eqref{eq:F_t_function_for_kernel}.
We refer to 
\cite{okounkov2003correlation},
\cite{Okounkov2005},
\cite{BorFerr2008DF},
\cite{Duits2011GFF},
\cite{SaenzKnizelPetrov2018}
for similar steepest descent arguments.

\begin{theorem}[Limit shape for Schur processes]
	\label{thm:limit_shape_Schur}
	There exists a function 
	$\mathfrak{H}(\xi,\tau)$ in $\xi\in \mathbb{R}_{\ge0}$,
	$\tau\in[0,1]$,
	depending on the parameters $x,y,\beta$ \eqref{eq:xyb_homogeneous_parameter} such that in probability, 
	\begin{equation}
		\label{eq:limit_shape_Schur}
		\lim_{N\to+\infty}\frac{H_N(\lfloor \xi N \rfloor ,\lfloor \tau N \rfloor )}{N}
		=
		\mathfrak{H}(\xi,\tau).
	\end{equation}
	The function 
	$\mathfrak{H}(\xi,\tau)$ is continuous,
	piecewise differentiable,
	weakly decreases in both $\xi$ and $\tau$, 
	and its gradient $\nabla \mathfrak{H}=
	(\partial_\xi\mathfrak{H},\partial_\tau \mathfrak{H})$
	belongs to the triangle
	\begin{equation}
		\label{eq:gradient_triangle_condition}
		-1\le \partial_\xi\mathfrak{H}\le 0,
		\qquad 
		-1\le \partial_\tau\mathfrak{H}\le 0,
		\qquad 
		\partial_\tau \mathfrak{H}\ge  \partial_\xi \mathfrak{H}.
	\end{equation}
	See \Cref{fig:limit_shape_Schur} for an illustration.
\end{theorem}

\begin{figure}[htpb]
	\centering
	\includegraphics[width=.9\textwidth]{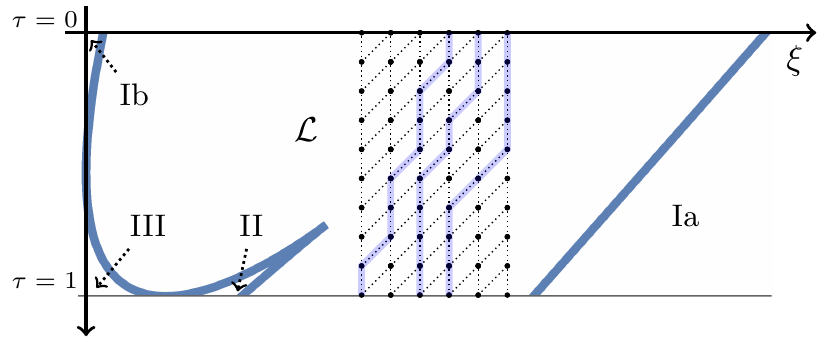}
	\caption{An example of the frozen boundary 
	curve in the $(\xi,\tau)$ coordinates, and an example of several up-diagonal
	paths as in \Cref{fig:nonintersecting_path_ensemble} serving as
	the level lines for the pre-limit height function 
	$H_N$ \eqref{eq:prelimit_height_function}.
	There are no up-diagonal paths in the frozen zones Ia-b, so $\nabla\mathfrak{H}=(0,0)$.
	In zone II, the paths go diagonally, so $\nabla\mathfrak{H}=(-1,-1)$.
	Finally, in zone III, the paths go vertically, so 
	$\nabla\mathfrak{H}=(-1,0)$. These frozen zone 
	gradients correspond to the vertices of the 
	triangle \eqref{eq:gradient_triangle_condition}.
	In this example, we have $x=1/3$, $y=1/5$, $\beta=-6$.
	For other values of parameters, zones Ib and II may be present.
	Zone III is always present, see \Cref{lemma:uniquely} below.}
	\label{fig:limit_shape_Schur}
\end{figure}

Throughout the rest of this subsection we will give an idea of 
proof of \Cref{thm:limit_shape_Schur}
together with the necessary formulas for the gradient $\nabla \mathfrak{H}$.
The integrand in the kernel $K_{\vec\beta;\mathrm{Gr}}^{2d}$ \eqref{eq:Kernel_Schur}
can be rewritten as
\begin{equation*}
	\frac{e^{N \bigl(S\bigl(z;\tfrac{a}{N},\tfrac{t}{N}\bigr) - S\bigl(w;\tfrac{b}{N},\tfrac{s}{N}\bigr)\bigr)}}{z-w},
\end{equation*}
where
\begin{equation}
	\label{eq:S_function_for_Schur}
	S(z;\xi,\tau)\coloneqq
	-(\xi-1)\log z+\log(1-z^{-1}y)-\log(1-zx)-(1-\tau)\log(1-\beta z^{-1}).
\end{equation}
The critical point equation $\frac{\partial}{\partial z}\ssp S(z;\xi,\tau)=0$
is equivalent to a cubic polynomial equation on~$z$:
\begin{equation}
	\label{eq:cubic_equation_z}
	\begin{split}
	&\xi  x z^3
	-(\xi +\beta  x (\xi +\tau -1)+(\xi +1) x y-1)z^2 \\
	&\hspace{100pt}
	+ (\beta  (\xi +\tau +\xi  x y+\tau  x y-2)+\xi  y)z-
	\beta  y (\xi +\tau -1)
	=0.
	\end{split}
\end{equation}
The region in the $(\xi,\tau)$ plane
where \eqref{eq:cubic_equation_z} has two complex conjugate nonreal roots
is called the \emph{liquid region} $\mathcal{L}$.
Inside it,
the gradient $\nabla\mathfrak{H}(\xi,\tau)$ belongs to the
interior of the 
triangle \eqref{eq:gradient_triangle_condition}.
In \Cref{fig:limit_shape_Schur}, $\mathcal{L}$ is the region inside 
the frozen boundary curve which we denote by $\partial\mathcal{L}$. 
For $(\xi,\tau)\in \mathcal{L}$,
denote by $z_c=z_c(\xi,\tau)$ the unique root of \eqref{eq:cubic_equation_z}
in the upper half complex plane. This is a critical point of the function $S$
\eqref{eq:S_function_for_Schur}.

\medskip

We forgot a standard steepest descent analysis of the kernel
$K_{\vec\beta;\mathrm{Gr}}^{2d}$ which would constitute a
detailed proof of \Cref{thm:limit_shape_Schur}. Instead, we
briefly explain how to derive explicit formulas for the
gradient $\nabla \mathfrak{H}$. We need only the \emph{a
priori} assumption (which would follow from the steepest
descent) that this gradient depends on the critical point
$z_c$ in a harmonic way when $z_c$ belongs to the upper
half-plane. When the point $(\xi,\tau)$ approaches the
boundary of the liquid region, the critical points $z_c$ and
$\bar z_c$ merge and become a real double critical point
which is a double root of the cubic equation
\eqref{eq:cubic_equation_z}.  Therefore, the frozen boundary
curve $\partial\mathcal{L}$ can be obtained in parametric
form by solving the equations
\begin{equation}
	\label{eq:double_crit_equations}
	\frac{\partial}{\partial z}\ssp S(z;\xi,\tau)=
	\frac{\partial^2}{\partial z^2}\ssp S(z;\xi,\tau)=0,
\end{equation}
in $(\xi,\tau)$, and taking $z_c=\bar z_c\in \mathbb{R}$ as a parameter. 
Equivalently, $\partial\mathcal{L}$ is the discriminant curve of the cubic equation \eqref{eq:cubic_equation_z}.
See \eqref{eq:xi_tau_equations}
below for this parametrization of the frozen boundary curve 
(we do not an explicit parametrization just yet).

We are only interested in the ``physical'' part
of the frozen boundary which lives in the
half-infinite
rectangle
$(\xi,\tau)\in[0,\infty)\times [0,1]$,
and so not all values of $z_c=\bar z_c\in \mathbb{R}$
correspond to points of the frozen boundary $\partial\mathcal{L}$.
Modulo this remark,
we get the following trichotomy of the frozen zones:

\begin{proposition}[Frozen zone trichotomy]
	\label{proposition:trichotomy}
	Depending on the 
	location of the double critical point,
	we have:
	\begin{enumerate}[$\bullet$]
		\item 
			If $\partial\mathcal{L}$ is adjacent to zones Ia or
			Ib, then $z_c=\bar z_c>0$.
		\item 
			If $\partial\mathcal{L}$ is adjacent to zone II, then $\beta<z_c=\bar z_c<0$.
		\item 
			If $\partial\mathcal{L}$ is adjacent to zone III, then $z=\bar z_c<\beta$.
	\end{enumerate}
\end{proposition}

Parts of 
$\partial\mathcal{L}$ 
bounding zones Ia-b
are asymptotically formed by
up-diagonal paths. In particular, 
the slope of these parts of $\partial\mathcal{L}$ 
in the $(\xi,\tau)$ coordinates cannot
exceed $1$. One can check that 
in \Cref{fig:limit_shape_Schur},
the
rightmost part of the frozen boundary $\partial\mathcal{L}$ is not linear and
has slope slightly less than 1.
On the other hand,
the boundaries of zones II and III are 
\emph{not} formed by our up-diagonal paths.
Instead, one should use suitably chosen ``dual paths''
defined through the complement
of the particle configuration $X^{2d}=\{x^m_j\}$.

\medskip

Using \Cref{proposition:trichotomy}, one can show that
inside the liquid region, we have the following expressions 
for the gradient of the limit shape in terms of the critical point $z_c(\xi,\tau)$:
\begin{equation}
	\label{eq:gradient_via_angles}
	\partial_\xi\ssp \mathfrak{H}(\xi,\tau)=
	-\frac{\mathop{\mathrm{Arg}}z_c(\xi,\tau)}{\pi}
	,\qquad 
	\partial_\tau\ssp \mathfrak{H}(\xi,\tau)=
	\frac{\mathop{\mathrm{Arg}}\bigl(z_c(\xi,\tau)-\beta\bigr)-\mathop{\mathrm{Arg}}z_c(\xi,\tau)}{\pi}
	.
\end{equation}
That is, $1+\partial_\xi \ssp \mathfrak{H}$ and $-\partial_\tau\ssp\mathfrak{H}$
are the normalized angles of the triangle in the complex plane with 
vertices $0,\beta$, and $z_c$, adjacent to $0$ and $z_c$, respectively (recall that both partial derivatives are negative). 
See \Cref{fig:triangle} for an illustration.

\begin{figure}[htpb]
	\centering
	\includegraphics[width=.4\textwidth]{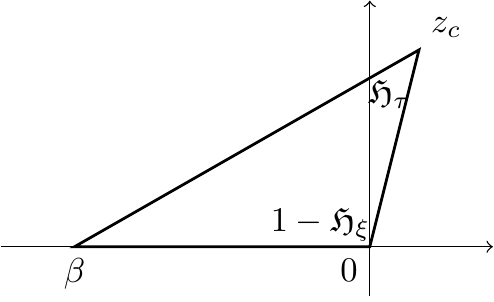}
	\caption{Triangle in the complex plane with vertices $0,\beta$, and $z_c$. 
		When $z_c$ approaches the real line at $(0,+\infty)$,
		$(\beta,0)$, and $(-\infty,\beta)$,
		we have, respectively,
		$\nabla\mathfrak{H}=(0,0)$, 
		$\nabla\mathfrak{H}=(-1,-1)$, 
		and
		$\nabla\mathfrak{H}=(-1,0)$.}
	\label{fig:triangle}
\end{figure}

\begin{remark}
	\label{rmk:complex_Burgers}
	From the cubic equation 
	$\frac{\partial}{\partial z}\ssp S(z;\xi,\tau)=0$,
	one can readily check that the complex critical point $z_c(\xi,\tau)$ satisfies
	the following version of the \emph{complex Burgers equation}:
	\begin{equation}
		\label{eq:complex_Burgers}
		\frac{\partial z_c(\xi,\tau)}{\partial \xi}=
		\left( 1-\frac{z_c(\xi,\tau)}{\beta} \right)
		\frac{\partial z_c(\xi,\tau)}{\partial \tau}.
	\end{equation}
	We refer to \cite{OkounkovKenyon2007Limit} for general details on how 
	the complex Burgers equation arises for limit shapes
	of planar dimer models.
\end{remark}

\subsection{From Schur to Grothendieck limit shapes}
\label{sub:schur_to_groth_shapes}

From the limit shape result for the Schur process
(\Cref{thm:limit_shape_Schur} above), we readily get the
limit shape of Grothendieck random partitions.  Indeed,
recall from \Cref{prop:Grothendieck_measures_embedding} that
the shifted random variables $\ell_i=\lambda_i+N-i$, $i=1,\ldots,N $, under the
Grothendieck measure $\mathscr{M}_{\vec \beta;\mathrm{Gr}}$
\eqref{eq:Grothendieck_measure_definition_complete}
are equal in distribution to the particle coordinates
$x^i_i=\mu^i_i+N-i$ corresponding to the random partitions
under the Schur process
$\mathscr{M}^{2d}_{\vec\beta;\mathrm{Gr}}$
\eqref{eq:Schur_process_from_Grothendieck}. 
The Schur process possesses a limit shape, so
when $i$ grows proportionally to
$N$, the random variables $x^i_i$ also scale proportionally to
$N$.
More precisely, \Cref{thm:limit_shape_Schur} and
the observation
\eqref{eq:H_N_condition_for_x_i_i}
imply that for all $\tau\in[0,1]$ we have
\begin{equation}
	\label{eq:x_ii_convergence}
	\frac{x^{i}_i}{N}\to
	\mathfrak{L}(\tau),\qquad i=\lfloor \tau N \rfloor , \qquad N\to+\infty,
\end{equation}
where the convergence is in probability.
Here $\mathfrak{L}(\tau)$ is a weakly decreasing function satisfying
the equation 
\begin{equation}
	\label{eq:functional_equation}
	\mathfrak{H}
	\left( 
	\mathfrak{L}(\tau),\tau
	\right)
	=\tau \quad \textnormal{for all }\tau\in[0,1],
\end{equation}
where $\mathfrak{H}(\xi,\tau)$ is the limit shape of the Schur process.
In other words, the shape $\mathfrak{L}(\tau)$ is the cross-section of the
Schur process limit shape surface $\eta=\mathfrak{H}(\xi,\tau)$
in the $(\xi,\tau,\eta)$ coordinates by the plane $\eta=\tau$.

\begin{lemma}
	\label{lemma:uniquely}
	The function $\mathfrak{L}(\tau)$, $\tau\in[0,1]$, is 
	continuous and is
	uniquely determined by equation \eqref{eq:functional_equation} and by continuity at the endpoints $\tau=0,1$. In particular, $\mathfrak{L}(1)=0$.
\end{lemma}
\begin{proof}
	Observe that
	$\mathfrak{H}(\xi,\tau)$
	strictly decreases in
	$\xi$ as long as $\mathfrak{H}\ne 0,1$. Indeed, 
	$\mathfrak{H}$ is strictly monotone
	in the liquid region thanks to the first of the identities
	in
	\eqref{eq:gradient_via_angles}.
	In the frozen zones, we have
	$\partial_\xi\ssp \mathfrak{H}=0$ only in zones Ia-b, 
	and there we have $\mathfrak{H}=0$ in Ia or $\mathfrak{H}=1$ in Ib.
	Indeed, 
	$\partial_{\xi} \mathfrak{H}=0$ implies $z_c(\xi, \tau) \in \mathbb{R}$,
	see \eqref{eq:gradient_via_angles}. 
	Since we only consider the ``physical'' part $(\xi,\tau)\in[0,\infty)\times [0,1]$, it follows that
	$(\xi,\tau)$ must lie on the frozen boundary.
	This implies that
	the function $\mathfrak{L}(\tau)$
	is determined by \eqref{eq:functional_equation} uniquely for $\tau\ne 0,1$, and is 
	continuous.

	For $\tau=0$, is it natural to set
	$\mathfrak{L}(0)=\max\{\xi\colon \mathfrak{H}(\xi,0)=0 \}$ by continuity.
	For $\tau=1$, it suffices to show that
	$(\mathfrak{L}(\tau),\tau)$ for $\tau$ close to $1$
	must belong to the frozen zone III and not Ib.
	For $\tau=1$, the critical point 
	equation \eqref{eq:cubic_equation_z} has a root $z=\beta$
	independently of $\xi$. 
	The discriminant of the remaining quadratic equation is negative for
	\begin{equation*}
		\frac{2}{1+\sqrt{xy}}
		<1+\xi
		<\frac{2}{1-\sqrt{xy}}.
	\end{equation*}
	In particular, the discriminant is positive for $\xi$ close to zero,
	so the point $(\xi,\tau)=(0,1)$
	lies in a frozen zone. 
	By looking at the 
	value of $z_c=\bar z_c$
	at a point of the adjacent frozen boundary, one can verify that 
	the neighborhood of $(\xi,\tau)=(0,1)$ is always in zone III.
	Thus, 
	$\mathfrak{H}(\xi,1)$
	is strictly monotone in $\xi$ in this neighborhood.
	Setting $\mathfrak{L}(1)=0$, we get the
	continuity of $\mathfrak{L}(\tau)$ at $\tau=0$, as desired.
\end{proof}

We arrive at the following
limit shape result for Grothendieck random partitions:

\begin{theorem}
	\label{thm:Groth_measures_convergence}
	Let $\lambda=(\lambda_1,\ldots,\lambda_N )$ be the
	Grothendieck random partition 
	distributed as
	$\mathscr{M}_{\vec \beta;\mathrm{Gr}}$
	\eqref{eq:Grothendieck_measure_definition_complete}
	with 
	parameters $x,y,\beta$ as in \eqref{eq:xyb_homogeneous_parameter}
	(in particular, $\beta<0$).
	For any fixed $\tau\in[0,1]$
	we have the convergence in probability:
	\begin{equation}
		\label{eq:lambda_i_limit_shape}
		\frac{\lambda_{\lfloor \tau N \rfloor }}{N}\to \mathfrak{L}(\tau)+\tau-1,\qquad N\to+\infty,
	\end{equation}
	where the function $\mathfrak{L}(\tau)$ is defined before \Cref{lemma:uniquely}.
\end{theorem}

In particular, the shift by $\tau-1$ in \eqref{eq:lambda_i_limit_shape}
comes from $\lambda_i=\ell_i+i-N$, $i=1,\ldots,N $.
When we need to indicate the dependence of $\mathfrak{L}(\tau)$ on the parameters, 
we will write
$\mathfrak{L}(\tau\mid x,y,\beta)$.

\medskip

To help visualize the Grothendieck limit shape
determined by the function $\mathfrak{L}(\tau)$,
we employ the coordinate system rotated by 
$45^\circ$ (for illustration, see \Cref{fig:YD_45_and_limsh}, left, in the Introduction).
In this way, Young diagrams and their limit shapes become functions
$\mathfrak{W}(u)$, $u\in \mathbb{R}$, satisfying
\begin{equation}
	\label{eq:continual_YD}
	|\mathfrak{W}(u)-\mathfrak{W}(v)|\le |u-v|,\qquad \mathfrak{W}(u)=|u| \quad\textnormal{for all large enough $|u|$}.
\end{equation}
Define the \emph{norm} of a continuous Young diagram by
\begin{equation}
	\label{eq:cont_YD_norm}
	\|\mathfrak{W}\|\coloneqq \frac{1}{2}\int_{-\infty}^{+\infty}\left( \mathfrak{W}(u)-|u| \right)du.
\end{equation}
Before the limit, the functions 
$\mathfrak{W}_N(u)$ corresponding to Young diagrams $\lambda$ with at most $N$ rows
are piecewise linear with derivatives $\pm1$
and integer maxima and minima. Note that $\|\mathfrak{W}_N\|=|\lambda|$ is the number of boxes in the
Young diagram.

The space of all functions satisfying 
\eqref{eq:continual_YD}
is called the \emph{space of} \emph{continuous Young diagrams} by 
Kerov, see \cite[Chapter~4]{Kerov-book}.\footnote{Our continuous Young diagrams 
are centered at zero, while Kerov considered a slightly more general framework.
This difference is not essential for us here.}
Young diagrams in the coordinate system rotated by $45^\circ$
were first considered in connection with
the Vershik--Kerov--Logan--Shepp (VKLS)
limit shape of Plancherel random partitions, see
\cite{logan_shepp1977variational},
\cite{VershikKerov_LimShape1077}.

Let the pre-limit continuous Young diagrams
$\mathfrak{W}_N(u)$ correspond to the
Grothendieck random partitions with parameters $(x,y,\beta)$.
The convergence
\eqref{eq:lambda_i_limit_shape} in
\Cref{thm:Groth_measures_convergence} implies
the pointwise convergence in probability as $N\to+\infty$ of the
rescaled functions $\frac{1}{N}\ssp\mathfrak{W}_N(uN)$ to a limit shape.
This limit shape is a
continuous Young diagram
$u\mapsto \mathfrak{W}(u)$ which has
parametric form
\begin{equation}
	\label{eq:W_curve_parametric_form}
	u=\mathfrak{L}(\tau)-1
	,\qquad 
	\mathfrak{W}=\mathfrak{L}(\tau)-1+2\tau,
	\qquad 
	\tau\in[0,1].
\end{equation}
This parametric form 
follows from
the change of coordinates from
$\left( \tau,\mathfrak{L}(\tau)+\tau-1 \right)$
to $(u,\mathfrak{W}(u))$ under the $45^\circ$ rotation.
When we need to indicate the dependence of $\mathfrak{W}(u)$ on the parameters of the Grothendieck
measure, 
we will write
$\mathfrak{W}(u\mid x,y,\beta)$.
Thus, we have established \Cref{thm:G_limsh} from the Introduction.

In \Cref{sub:simulations_and_particular_cases} below we present graphs of the limit
shapes \eqref{eq:W_curve_parametric_form}
for several choices of parameters $(x,y,\beta)$ of the Grothendieck measure.

\subsection{Properties of Grothendieck limit shapes}
\label{sub:grothendieck_limit_properties}

Here let us make several general observations 
in connection with the limit shape result for Grothendieck random partitions
(\Cref{thm:Groth_measures_convergence}).

\subsubsection{Differential equations}

Differentiating \eqref{eq:functional_equation}
in $\tau$, we see that 
$\mathfrak{L}(\tau)$ satisfies the differential equation
$\mathfrak{L}'(\tau)=
\frac{1-\partial_\tau\ssp \mathfrak{H}(\mathfrak{L}(\tau),\tau)}{\partial_\xi\ssp \mathfrak{H}(\mathfrak{L}(\tau),\tau)}$.
In terms of the critical point,
with the help of \eqref{eq:gradient_via_angles},
this equation has the form
\begin{equation}
	\label{eq:diff_eq_for_L}
	\mathfrak{L}'(\tau)
	=
	-\frac{\pi-\mathop{\mathrm{Arg}}\bigl(z_c(\mathfrak{L}(\tau),\tau)-\beta\bigr)
	+
	\mathop{\mathrm{Arg}}z_c(\mathfrak{L}(\tau),\tau)}
	{\mathop{\mathrm{Arg}}z_c(\mathfrak{L}(\tau),\tau)}.
\end{equation}
Here $z_c=z_c(\xi,\tau)$ is the root of the cubic equation 
\eqref{eq:cubic_equation_z} in the upper half plane if $(\xi,\tau)$ belongs to the liquid region.
When $(\xi,\tau)$ is in a frozen zone, 
$z_c$ should be taken real such that the arguments 
in \eqref{eq:gradient_via_angles}
give the gradient $\nabla\mathfrak{H}$ in this frozen zone. 
We refer to the 
trichotomy in \Cref{proposition:trichotomy}, see also
\Cref{fig:limit_shape_Schur} for an illustration.

The limit shape continuous Young diagram 
$\mathfrak{W}(u)$ 
\eqref{eq:W_curve_parametric_form}
in the $45^\circ$ rotated coordinate system
satisfies
a more symmetric
differential equation
\begin{equation}
	\label{eq:diff_eq_for_W}
	\mathfrak{W}'(u)=
	\frac{\pi-\mathop{\mathrm{Arg}}\bigl(z_c(u+1,\frac{\mathfrak{W}(u)-u}{2})-\beta\bigr)
	-
	\mathop{\mathrm{Arg}}z_c(u+1,\frac{\mathfrak{W}(u)-u}{2})}
	{\pi-\mathop{\mathrm{Arg}}\bigl(z_c(u+1,\frac{\mathfrak{W}(u)-u}{2})-\beta\bigr)
	+
	\mathop{\mathrm{Arg}}z_c(u+1,\frac{\mathfrak{W}(u)-u}{2})}.
\end{equation}

We remark that the root $z_c(\xi,\tau)$ of the cubic equation
\eqref{eq:cubic_equation_z} depends on $(\xi,\tau)$ in a
nonlinear and somewhat implicit manner.
Therefore, it may be challenging to extract
useful information about the
Grothendieck limit shape 
from the 
differential
equations \eqref{eq:diff_eq_for_L}--\eqref{eq:diff_eq_for_W}.
Even for producing the plots in \Cref{sub:simulations_and_particular_cases}
below we relied not on these differential equations, but rather on the 
original implicit equation 
\eqref{eq:functional_equation}.

\subsubsection{Staircase frozen facet}

Observe that when $(\xi,\tau)$ is in the frozen zone II, 
we have $\nabla\mathfrak{H}=(-1,-1)$,
which corresponds to taking
$z_c$ from $(\beta,0)$.
Thus, in this frozen zone we have from 
\eqref{eq:diff_eq_for_L} and \eqref{eq:diff_eq_for_W}:
\begin{equation}
	\label{eq:staircase_frozen_derivatives}
	\mathfrak{L}'(\tau)=-\frac{1}{2},\qquad 
	\mathfrak{W}'(u)=0.
\end{equation}

Notice that for \eqref{eq:staircase_frozen_derivatives}
to hold, the point $(\mathfrak{L}(\tau),\tau)$
of the Grothendieck limit shape
must belong to the frozen zone II. 
In fact, this is possible for certain choices of the 
parameters $(x,y,\beta)$, namely, when $\beta$ is sufficiently
large in the absolute value:
\begin{lemma}
	\label{lemma:large_negative_beta_always_zone_II}
	Let the parameters $(x,y,\beta)$ satisfy
	\eqref{eq:xyb_homogeneous_parameter}.
	For any fixed $x,y$,
	there exists $\beta_0<0$ such that for all
	all $\beta<\beta_0$, the frozen zone II
	extends through the whole horizontal strip $\xi>0$, $0<\tau<1$ in the $(\xi,\tau)$ coordinates.
	See \Cref{fig:large_beta} for an illustration.
\end{lemma}

\begin{figure}[htpb]
	\centering
	\includegraphics[width=.55\textwidth]{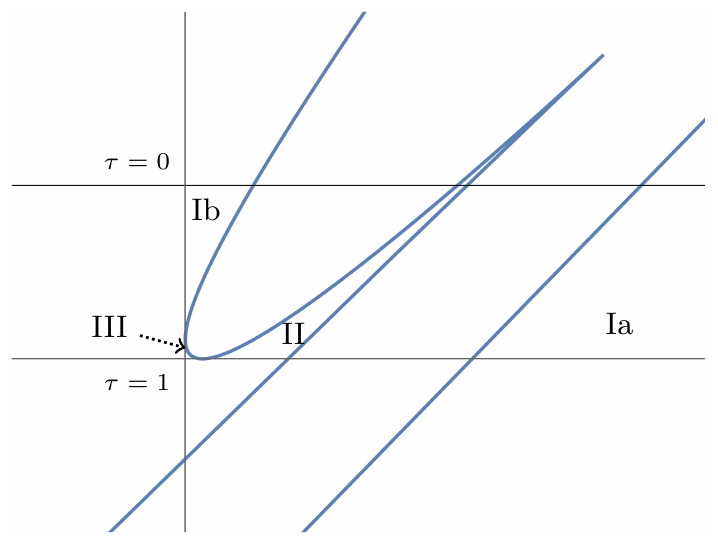}
	\caption{Frozen boundary curve in the full space $\mathbb{R}^2$
	with $x=1/3$, $y=1/5$, $\beta=-25$.
	As $\beta$ decays to $-\infty$, the cusp point goes to
	infinity along the main diagonal in the first quadrant in
	the coordinates $(\xi,1-\tau)$.}
	\label{fig:large_beta}
\end{figure}

\begin{proof}
	An explicit parametrization of the frozen
	boundary curve by $z\in \mathbb{R}$ is obtained by 
	solving the double critical point equations
	\eqref{eq:double_crit_equations}. This parametrization
	has the form
	\begin{equation}
		\label{eq:xi_tau_equations}
		\xi=
		\frac{(1-x y) \left(y \left(\beta +x z^2-2
   z\right)+z^2 (1-\beta  x)\right)}{(1-x z)^2
	 (y-z)^2},\quad 
		\tau=1+
		(z-\beta)^2\ssp
		\frac{(1-x y) \left(y-x z^2\right)}{(-\beta)  
		(1-x z)^2 (y-z)^2}.
	\end{equation}
	One can check the following facts about the frozen boundary curve $\partial\mathcal{L}$
	in the whole space $(\xi,\tau)\in \mathbb{R}^2$:
	\begin{enumerate}[$\bullet$]
		\item $\partial\mathcal{L}$ is tangent to the horizontal coordinate line 
			at a unique point $(\xi,1)$ with $\xi>0$. This point
			corresponds to $z=\beta$ which is a double zero of $\tau-1$.
			Substituting $z=\beta$ into $\xi$ produces a positive quantity.
		\item $\partial\mathcal{L}$ is 
			tangent to the vertical coordinate line at a unique point $(0,\tau)$ with $\tau<1$.
			This point corresponds to $z\to \infty$. Taking this limit in $\tau$ shows that the tangent point
			has $\tau<1$.
		\item The slope of the curve $\partial\mathcal{L}$ 
			in the coordinates $(\xi,1-\tau)$ (as in \Cref{fig:limit_shape_Schur})
			is
			\begin{equation*}
				-\frac{\partial\tau/\partial z}{\partial\xi/\partial z}=1-\frac{z}{\beta},
			\end{equation*}
			which changes sign only at $z=\beta$ and $z=\infty$.
		\item For $z=y$ and $z=1/x$, the curve $\partial \mathcal{L}$
			goes to infinity in two different asymptotic directions.
			For $z\to y$ we have $\xi,1-\tau\to-\infty$, 
			and for $z\to 1/x$ we have $\xi,1-\tau\to+\infty$.
			Each of these asymptotic directions has degree 2, that is,
			there are exactly two components of $\partial\mathcal{L}$
			escaping to infinity in each of the first and the third quadrants in the coordinates
			$(\xi,1-\tau)$.
	\end{enumerate}

	Now let us look at the ``cusp'' point of $\partial\mathcal{L}$,
	that is, where the third derivative of $S(z;\xi,\tau)$ vanishes
	along with the first two.
	In \Cref{fig:limit_shape_Schur}, the cusp is at the tip of the 
	frozen zone~II.
	Let us show that the cusp point is always unique and exists in the full space
	$(\xi,\tau)\in \mathbb{R}^{2}$.
	Take
	$\frac{\partial^3}{\partial z^3}\ssp S(z;\xi,\tau)$,
	and
	substitute into it $\xi,\tau$ as in \eqref{eq:xi_tau_equations}.
	We obtain a rational function in $z$ and the parameters 
	$(x,y,\beta)$ whose numerator 
	is a cubic polynomial
	\begin{equation*}
		\mathsf{P}(z)=
		z^3 x (1 + x y - x\beta )
		-
		3z^2xy
		+3zx y \beta
		+
		y (y -\beta - x y \beta).
	\end{equation*}
	The discriminant of $\mathsf{P}(z)$ 
	is $-27\ssp x^2 y^2 (1-x\beta)^2 (1-x y)^2 (y-\beta )^2$,
	which is manifestly negative. Therefore, 
	there is a unique real root $z$ of $\mathsf{P}(z)$, and it corresponds to the cusp point.
	
	Let us look at the behavior of the cusp point for large negative $\beta$.
	We have 
	\begin{equation*}
		\mathsf{P}(z) = \mathsf{P}_0(z)\beta+O(1),\qquad
		\beta\to-\infty,
		\quad \textnormal{where}\quad
		\mathsf{P}_0(z)\coloneqq -x^2 z^3+3 x y z-y (1+x y).
	\end{equation*}
	The polynomial $\mathsf{P}_0(z)$ has a unique real root,
	denote it by $z_0$. Clearly, the root of $\mathsf{P}(z)$
	becomes close to $z_0$ as $\beta\to-\infty$.
	Next, for $z=z_0$, we have
	\begin{equation*}
		\xi=1-\tau+O(1)=
		\frac{(1-x y) \left(y-x z_0^2\right)}{(1-x z_0)^2
		(y-z_0)^2}\ssp\beta+O(1),\qquad \beta\to-\infty.
	\end{equation*}
	Since $\mathsf{P}_0(-\sqrt{y/x})>0$ and the coefficient by $z^3$ in $\mathsf{P}_0$ is negative, we see that $y-x z_0^2>0$.  This inequality implies that the cusp point of $\partial\mathcal{L}$ goes to infinity along the main diagonal in the first quadrant in the coordinates $(\xi,1-\tau)$.

	We conclude that for large $\beta$, four components of the frozen boundary escape as $\xi,1-\tau\to+\infty$, and two components escape as $\xi,1-\tau\to-\infty$. Together with the tangence properties observed at the beginning of the proof, this implies that each horizontal line at height $\tau$, $\tau\in[0,1]$, intersects the frozen boundary precisely four times. The frozen zone II is between the middle two intersections. This completes the proof.
\end{proof}
	
By \Cref{lemma:large_negative_beta_always_zone_II}, if $|\beta|$ is sufficiently large and $\beta<0$, the limit shape of the Grothendieck random partition always has a part where the derivative satisfies \eqref{eq:staircase_frozen_derivatives}. In particular, the density of the particles $\ell_i$ is $\frac{1}{2}$. We call the part of the Grothendieck limit shape where \eqref{eq:staircase_frozen_derivatives} holds the \emph{staircase frozen facet}.

Let us discuss how the partition $\lambda$ looks in the staircase facet. For the two-dimensional Schur process, in zone II the up-diagonal paths are densely packed and move diagonally. In terms of the particle configuration $X^{2d}=\{x_i^m\colon 1\le m,i\le N\}$, this means that
\begin{equation*}
	x^{m+1}_i=x^m_i-1,\qquad x^m_{i+1}=x^m_i-1.
\end{equation*}
Thus, for the coordinates $\lambda_i$ and $\ell_i=\lambda_i+N-i$ of the Grothendieck random partition, where $\ell_i=x^i_i$ in distribution (\Cref{prop:Grothendieck_measures_embedding}), we have in the staircase facet:
\begin{equation*}
	\ell_{i+1}
	=\ell_i-2,\qquad 
	\lambda_{i+1}
	=\lambda_i-1.
\end{equation*}
Thus, in this facet, the Young diagram $\lambda$ locally looks like a staircase with no fluctuations. This justifies the name ``staircase frozen facet''. In the $45^\circ$ rotated coordinates, the limit shape 
$\mathfrak{W}(u\mid x,y,\beta)$
of $\lambda$ is horizontal in this facet. We refer to \Cref{sub:simulations_and_particular_cases} for illustrations of staircase facets.

\subsubsection{Reduction to Schur measures and Plancherel limit shapes}
\label{subsub:Plancherel_reduction}

Observe that for all $\beta\le 0$, under the Schur process $\mathscr{M}^{2d}_{\vec\beta;\mathrm{Gr}}$ \eqref{eq:Schur_process_from_Grothendieck}, the marginal distribution of the partition $\mu^N$ is simply the Schur measure with probability weights
\begin{equation}
	\label{eq:Schur_m_for_beta_zero}
	(1-xy)^{N^2}\ssp
	s_{\mu^N}(\underbrace{x,\ldots,x}_N )\ssp
	s_{\mu^{N}}(\underbrace{y,\ldots,y}_N ).
\end{equation}
When $\beta=0$, the weights $(-\beta)$ of the diagonal edges
in the directed graph in
\Cref{fig:nonintersecting_path_ensemble} vanish. Thus, for
$\beta=0$, all up-diagonal paths must go vertically, and
almost surely $\mu^m=\mu^N$ for all $m=1,\ldots, N $. This
implies that for fixed $x,y$ and as $\beta\nearrow 0$, the
limit shape $\mathfrak{W}(u\mid x,y,\beta)$ of the
Grothendieck random partition $\lambda_i=\mu^i_i$ converges 
to that of the Schur measure
\eqref{eq:Schur_m_for_beta_zero}. Denote the latter limit shape by
$\mathfrak{S}_{\tau=1}(u)$. It depends on the parameters 
$x,y$ only through their product, and is independent of $\beta$.

On the other hand, the marginal distribution of the partition $\mu^{1}$ is the following Schur measure:
\begin{equation}
	\label{eq:Schur_m_for_beta_final}
	(1-xy)^{N^2}(1-x\beta)^{-N(N-1)}
	s_{\mu^1}(\underbrace{x,\ldots,x}_N)\ssp
	s_{\mu^1}(\underbrace{y,\ldots,y}_N;\underbrace{-\hat \beta,\ldots,-\hat \beta}_{N-1} ),
\end{equation}
where $(-\hat \beta,\ldots,-\hat \beta )$ is the dual specialization (e.g.,
see \cite[Section 2]{BorodinGorinSPB12} for the definition).
Denote the limit shape of $\mu^1$ by
$\mathfrak{S}_{\tau=0}(u)$. It depends on all our parameters $x,y,\beta$.

Using, for example, the Robinson--Schensted--Knuth
correspondence \cite{Knuth1970}, \cite{fulton1997young}, one
can show that the expected numbers of boxes in the partitions 
$\mu^N$ and $\mu^1$ are, respectively,
\begin{equation*}
	\mathbb{E}|\mu^N|=\frac{N^2xy}{1-xy},\qquad 
	\mathbb{E}|\mu^1|=\frac{N^2xy}{1-xy}-N(N-1)\beta y.
\end{equation*}
Dividing by $N^2$ (which comes from rescaling both coordinate directions
of the continuous Young diagram by 
$N^{-1}$), we see that the
norms \eqref{eq:cont_YD_norm} of the limit 
shapes are
\begin{equation*}
	\|\mathfrak{S}_{\tau=1}\|=\frac{xy}{1-xy},\qquad 
	\|\mathfrak{S}_{\tau=0}\|=\frac{xy}{1-xy}-\beta y.
\end{equation*}

When $x,y$ are small,
$\|\mathfrak{S}_{\tau=1}\|$ is of order $xy$, and 
the rescaled
limit shape
\begin{equation*}
	\frac{1}{\sqrt{xy}}\ssp\mathfrak{S}_{\tau=1}(u\sqrt{xy})
\end{equation*}
converges to the celebrated Vershik--Kerov--Logan--Shepp
(VKLS) curve
\begin{equation}
	\label{eq:VKLS}
	\Omega(u)\coloneqq
	\begin{cases}
		\frac{2}{\pi} \left(u \arcsin(\frac{u}{2}) + \sqrt{4 - u^2}\right)
		,&
		|u|\le 2;\\
		|u|,&|u|>2.
	\end{cases}
\end{equation}
Note that $\|\Omega\|=1$.

By \Cref{prop:Grothendieck_measures_embedding} and
\Cref{thm:Groth_measures_convergence}, the Grothendieck
random partition $\lambda$ should have a limit shape which
is between those of $\mu^1$ and $\mu^N$. When $-\beta y\ll
xy\ll 1$, the limit shapes of $\mu^1$ and $\mu^N$ should be close
to each other and to the VKLS shape \eqref{eq:VKLS}. 
Together with numerical experimentation 
in \Cref{sub:simulations_and_particular_cases} below,
this prompts the following conjecture:

\begin{conjecture}
	\label{conj:Groth_Plancherel}
	Let $x,y\searrow 0$ and $\beta=\beta(x)\nearrow 0$ such that $-\beta(x)\ll x$. 
	Then the rescaled Grothendieck limit shape 
	$\frac{1}{\sqrt{xy}}\ssp \mathfrak{W}(u\sqrt{xy}\mid x,y,\beta(x))$ 
	converges to the VKLS shape $\Omega(u)$.
\end{conjecture}

Consider another regime when $x,y\searrow 0$ but $\beta=\beta(y) \to-\infty$ such that $-\beta y$ is fixed. 
This change in $\beta$ does not affect $\mu^N$, and the norm of $\mathfrak{S}_{\tau=1}$
goes to zero. After rescaling, $\mathfrak{S}_{\tau=1}$ is close to
VKLS shape.
The limit shape of $\mu^1$ at $\tau=0$, on the other hand,
grows macroscopically, but one can show that 
$\mathfrak{S}_{\tau=0}$ still contains a Plancherel-like part 
at scale $xy$ in the neighborhood of $u=1$.
In the two-dimensional
Schur process picture, for $\beta\to-\infty$
the up-diagonal paths strongly prefer to go diagonally. 
Therefore, we expect that in this regime, the Grothendieck limit shape 
contains a 
(shifted)
Plancherel-like part. However, 
numerical experimentation 
\Cref{sub:simulations_and_particular_cases} below
suggests that this shape is not exactly the 
VKLS shape. Let us formulate a conjecture:

\begin{conjecture}
	\label{conj:Groth_Plancherel_II}
	Let $x,y\searrow 0$ and $\beta=\beta(y)=-K/y$, where $K>0$ is fixed. 
	There exists $K_0>0$ such that for all $K>K_0$, in the
	$O(\sqrt{xy})$-neighborhood of $u=1$, the Grothendieck limit shape 
	$\mathfrak{W}(u\mid x,y,\beta(y))$
	is close to
	\begin{equation}
		\label{eq:Omega_for_shPl}
		\frac{u+1}{2}+\frac{\sqrt{xy}}{2}\ssp \Omega_{(K)}\left( \frac{u-1}{\sqrt{xy}} \right),
	\end{equation}
	where
	$\Omega_{(K)}$ is a suitable $K$-dependent deformation of the VKLS shape $\Omega$. 
	We expect that as $K\to -\infty$, the shapes $\Omega_{(K)}$ approach $\Omega$.
\end{conjecture}

We have formulated Conjectures \ref{conj:Groth_Plancherel} and \ref{conj:Groth_Plancherel_II} only for limit shapes, but similar Plancherel-like behavior should arise for Grothendieck random partitions themselves.

\subsection{Grothendieck limit shape plots}
\label{sub:simulations_and_particular_cases}

The limit shape surface $\mathfrak{H}(\xi,\tau)$ of the two-dimensional Schur process has the normal vector $\nabla\mathfrak{H}(\xi,\tau)$. This gradient is expressed through the solution $z_c(\xi,\tau)$ to the cubic equation \eqref{eq:cubic_equation_z}, see \eqref{eq:gradient_via_angles}.
However, $\mathfrak{H}(\xi,\tau)$ itself is not explicit, making it necessary to employ numerical integration to graph the surface. This is achieved by integrating the gradient along the $\xi$ direction, starting from $+\infty$ and moving towards the point $(\xi,\tau)$.

Recall that the Grothendieck limit shape $\mathfrak{L}(\tau)$ is the cross-section of the Schur process limit shape surface $\eta=\mathfrak{H}(\xi,\tau)$ in the $(\xi,\tau,\eta)$ coordinates, at the plane $\eta=\tau$. This cross-section is not explicit either. Therefore, we need to solve numerically the implicit equation \eqref{eq:functional_equation} to get the desired function $\mathfrak{L}(\tau)$. After obtaining $\mathfrak{L}(\tau)$, we use it to graph the shape $\mathfrak{W}(u)$ in the coordinate system rotated by $45^\circ$ using the parametric representation \eqref{eq:W_curve_parametric_form}. 

We remark that the differential equations \eqref{eq:diff_eq_for_L} and \eqref{eq:diff_eq_for_W} for $\mathfrak{L}(\tau)$ or $\mathfrak{W}(u)$, respectively, are not very useful for graphing directly, as they cannot be solved explicitly. While a numerical solution of these differential equations is possible, it would require specific convergence estimates, which we avoid with our more direct approach.

\medskip

We implement a cubic equation solver in Python to find the roots $z_c(\xi,\tau)$ along a regular grid of $(\xi,\tau)$, utilizing the code from \cite{shril_halder_CubicEquationSolver}. Then (also with Python) we perform direct numerical integration of the $\xi$-gradient of the height function \eqref{eq:gradient_via_angles}. After that we solve the implicit equation \eqref{eq:functional_equation}. This procedure yields the values of $\mathfrak{W}(u)$ along a non-regular grid in $u$, which is sufficient for graphing the limit shape of Grothendieck random partitions. Our Python code is available at \cite{gavrilova_petrov_2023_code}.

\begin{remark}
	Here we do not perform probabilistic simulations of Grothendieck random partitions, but rather focus on numerically graphing the Grothendieck limit shapes which exist due to \Cref{thm:Groth_measures_convergence}.
\end{remark}

Next we descibe the resulting plots of Grothendieck limit shapes. The images are located at the end of the paper.

\subsubsection{Basic example ($x=1/3,\,y=1/5,\,\beta=-6$)}

First, we take the same parameters as in \Cref{fig:limit_shape_Schur}. 
In \Cref{fig:sim1}, the
left pane displays the limit shape surface $\mathfrak{H}$ for the Schur
process (red), the plane $\eta=\tau$ (blue), and the curve
$(\mathfrak{L}(\tau),\tau,\tau)$ in the cross-section. The
top right pane shows the projection of the cross-section
onto the bottom coordinate plane $(\xi,1-\tau)$, and also includes the frozen boundary curve. 
The frozen boundary is the same as in
\Cref{fig:limit_shape_Schur}. The bottom right pane presents
the limit shape of Grothendieck random partitions in the
coordinates $(u,\mathfrak{W}(u))$. The limit shape
$\mathfrak{W}(u)$ always lies below the line $u+2$, as the
number of nonzero parts in the Grothendieck random partition
is limited to at most $N$.

\subsubsection{Large negative beta ($x=1/3,\,y=1/5,\,\beta=-25$)}

Second, we consider the case of large negative $\beta$.
In the top left pane in \Cref{fig:sim2}, we have zoomed in around the flat
section of the surface $\mathfrak{H}$ (red).
The blue
plane corresponds to $\eta=\tau$, and the black meshed
plane extends the zone II frozen facet of $\mathfrak{H}$ which has
$\nabla\mathfrak{H}=(-1,-1)$. 
We see that the intersection of the blue plane with the red surface is a straight line in this neighborhood.
The bottom left pane displays the
projection of the cross-section, similar to
\Cref{fig:sim1}. By
\Cref{lemma:large_negative_beta_always_zone_II}, the
black curve must traverse through zone II. On the right
pane, we added a horizontal line to highlight the staircase frozen facet where the
limit shape $\mathfrak{W}(u)$ is horizontal.

\subsubsection{Plancherel-like behavior for small negative beta}

In \Cref{fig:sim3,fig:sim4}
we numerically support \Cref{conj:Groth_Plancherel}
that the rescaled Grothendieck limit shape converges to the 
VKLS shape $\Omega(u)$ \eqref{eq:VKLS} as $x,y\to0$ such that $\beta\ll x$.
In \Cref{fig:sim3} the parameters $x=y=1/40$ are fixed. 
As $\beta$ gets close to zero (we chose three orders, $(xy)^{\frac{1}{2}}, (xy)^{\frac{3}{4}}$, and $xy$),
we see that the plots of $\mathfrak{W}(u)$ get closer to the VKLS shape.
Moreover, \Cref{fig:sim4} demonstrates that for smaller $x=y=1/100$, taking $\beta=1/1000$ (order $(xy)^{\frac{3}{4}}$),
makes the shape $\mathfrak{W}(u)$ closer than for $x=y=1/40$.
Indeed, we have for the uniform norms:
\begin{equation*}
\begin{split}
	(xy)^{-\frac12}\cdot\bigl\|\mathfrak{W}(\cdot \mid x,y,\beta) - \sqrt{xy}\,\Omega(\cdot /\sqrt{xy})\bigr\|_C
	\bigg\vert_{x=y=1/40,\,\beta=1/250}
	&\approx 0.10;
	\\
	(xy)^{-\frac12}\cdot\bigl\|\mathfrak{W}(\cdot \mid x,y,\beta) - \sqrt{xy}\,\Omega(\cdot /\sqrt{xy})\bigr\|_C 
	\bigg\vert_{x=y=1/100,\,\beta=1/1000}
	&\approx 0.06;
	\\
	(xy)^{-\frac12}\cdot\bigl\|\mathfrak{W}(\cdot \mid x,y,\beta) - \sqrt{xy}\,\Omega(\cdot /\sqrt{xy})\bigr\|_C 
	\bigg\vert_{x=y=1/900,\,\beta=1/27000}
	&\approx 
	0.044
	,
\end{split}
\end{equation*}
which suggests that these expressions should
decay to zero.

\subsubsection{Plancherel-like behavior for large negative beta}
\label{subsub:positive_beta}

In \Cref{fig:sim5}, we consider the regime of \Cref{conj:Groth_Plancherel_II},
and take $x=y=1/40$, $\beta=-120$, so $-\beta y=3$.
The Grothendieck limit shape $\mathfrak{W}(u)$ has a staircase
frozen facet, and to the right of it we 
observe a curved part of size $O(\sqrt{xy})$. Zooming in,
we see that this part of $\mathfrak{W}(u)$ does not seem to be close to the shifted and scaled VKLS shape,
see \Cref{fig:sim5}, right.

\subsubsection{Positive beta ($x=1/3,\,y=1/5,\,\beta=1/12$)}
By \Cref{prop:Grothendieck_extended_positivity}, 
the Grothendieck measure $\mathscr{M}_{\vec\beta;\mathrm{Gr}}(\lambda)$
on partitions 
is also nonnegative for 
$0\le \beta<\min (x^{-1},y)$. 
Setting $\beta>0$ 
makes the corresponding two-dimensional Schur process 
\eqref{eq:Schur_process_from_Grothendieck}
a signed probability measure.
However, in this case we can still define the 
surface $\mathfrak{H}(\xi,\tau)$ 
using the root $z_c(\xi,\tau)$ of the cubic equation
\eqref{eq:cubic_equation_z}.
Then we can define the curve $\mathfrak{L}(\tau)$
as the solution of the implicit equation
\eqref{eq:functional_equation}, and 
finally obtain a shape $\mathfrak{W}(u)$
via \eqref{eq:W_curve_parametric_form}.
This leads to the following conjecture:

\begin{conjecture}
	\label{conj:positive_beta}
	The curve $\mathfrak{W}(u)=\mathfrak{W}(u\mid x,y,\beta)$
	is well-defined by the procedure described above for all
	$\beta<\min(x^{-1},y)$. Moreover, this curve $\mathfrak{W}(u)$
	is the limit shape of the random partitions 
	distributed according to the Grothendieck measure with homogeneous parameters
	\eqref{eq:xyb_homogeneous_parameter} in the $45^\circ$ rotated coordinate system.
\end{conjecture}

In \Cref{fig:sim6}, we numerically support \Cref{conj:positive_beta}
by considering parameters 
$x=1/3$, $y=1/5$, $\beta=1/12$.
We plot the surface $\mathfrak{H}(\xi,\tau)$ in 
\Cref{fig:sim6}, top left. In the bottom left pane
we plot
the curve $\mathfrak{L}(\tau)$ together with the
``frozen boundary''. An interesting feature is that
here $\mathfrak{L}(\tau)$ is tangent to this
``frozen boundary''. 
Finally, in \Cref{fig:sim6}, right, 
we plot the conjectural limit shape $\mathfrak{W}(u\mid\frac{1}{3},\frac{1}{5},\frac{1}{12})$.
From additional numerical examples we also noticed that as $\beta\nearrow \min(x^{-1},y)$,
we have 
$\|\mathfrak{W}(\cdot\mid x,y,\beta)\|\to0$.

\subsection{Exact sampling Grothendieck random partitions by Schur dynamics}
\label{sub:sampling}

It is known that Schur processes can be exactly sampled
using push-block type dynamics or Robinson--Schensted--Knuth
(RSK) correspondences. We refer to \cite{BorFerr2008DF},
\cite{Borodin2010Schur}, \cite{BorodinGorinSPB12},
\cite[Section~4]{MatveevPetrov2014}, or
\cite{Betea_etal2014} for various expositions of general
sampling mechanisms for Schur processes. An application to
our Schur process $\mathscr{M}^{2d}_{\vec\beta;\mathrm{Gr}}$
\eqref{eq:Schur_process_from_Grothendieck} is implemented in
Python \cite[file
\texttt{RSK\_code.py}]{gavrilova_petrov_2023_code} and
follows the RSK dynamics on interlacing arrays as in
\cite{MatveevPetrov2014}. We only work with homogeneous parameters, but a straightforward modification would cover 
the fully inhomogeneous case.
The results of the simulation are given in \Cref{fig:samples}.

Let us briefly describe our sampling mechanism in terms of semistandard Young tableaux (which are in a well-known bijection 
with interlacing arrays). Start from an empty Young tableau $T(0)=\varnothing$.

In the first stage, insert into this tableau
an $N\times N$ matrix $A$ of independent geometric random variables with distribution $\mathrm{Prob}(\xi=k)=
(1-xy)\cdot(xy)^{k}$, $k=0,1,2,\ldots $ using the classical
RSK correspondence \cite{Knuth1970}.
This procedure is performed in $N$ steps, and in each $t$-th
step we form the word $1^{A_{t1}}\ldots N^{A_{tN}}$, where $A_{ij}\in \mathbb{Z}_{\ge0}$ are the elements
of the matrix $A$, and powers of letters mean repetition. This word is then inserted into the Young tableau
$T(t-1)$
using the usual RSK insertion.
After these $N$ steps, the shape of our Young
tableau $T(N)$ is distributed according to the Schur measure
\eqref{eq:Schur_m_for_beta_zero} with specializations
$(x,x,\ldots,x)$ and $(y,y,\ldots,y )$. 

In the second stage, take an $(N-1)\times N$ matrix of independent Bernoulli random variables with distribution 
$\mathrm{Prob}(\eta=1)=-\beta x/(1-\beta x)$, and insert it into the tableau $T(N)$ using 
the dual RSK correspondence. This procedure is performed in $N-1$ steps, and in each $s$-th step we form the word 
$1^{B_{s1}}\ldots N^{B_{sN}}$ (where $B_{ij}\in\left\{ 0,1 \right\}$ are the elements of $B$), and insert it into the Young tableau 
$T(N+s-1)$ using the dual RSK insertion. An implementation
of the dual RSK insertion for semistandard Young tableaux
(equivalently, interlacing arrays of integers) that we used is the algorithm $\mathscr{Q}_{\mathrm{row }}^{q=0}[-\hat{\beta}]$ from
\cite[Section~4.3]{MatveevPetrov2014}. 
After these $N-1$ steps, the shape of our Young
tableau $T(2N-1)$ is distributed according to the Schur measure
\eqref{eq:Schur_m_for_beta_final} with specializations
$(x,x,\ldots,x)$ and $(y,y,\ldots,y;-\hat \beta,\ldots,-\hat \beta)$.

To obtain the Grothendieck random partition, one has to track different parts of the shape of the 
evolving Young tableau $T(N+s-1)$. Namely, set 
\begin{equation}
	\label{eq:lambda_Groth_from_T}
	\lambda_{N-s+1}=T(N+s-1)_{N-s+1},\qquad s=1,2,\ldots,N ,
\end{equation}
where $T(N+s-1)_j$ means the $j$-th part of the shape of the Young tableau. 

\begin{proposition}
	\label{prop:lambda_as_Groth}
	The distribution of the random Young diagram $\lambda=(\lambda_1,\ldots,\lambda_N )$
	defined by \eqref{eq:lambda_Groth_from_T} 
	coincides with the Grothendieck measure \eqref{eq:Grothendieck_measure_definition_complete_intro} 
	with homogeneous parameters
	$x_i=x$, $y_j=y$, $\beta_r=\beta$.
\end{proposition}
\begin{proof}
	The joint distribution of the shapes of the semistandard 
	Young tableaux $T(N+s-1)$, $s=1,\ldots,N$,
	is given by the Schur process
	$\mathscr{M}^{2d}_{\vec\beta;\mathrm{Gr}}$ \eqref{eq:Schur_process_from_Grothendieck} (with homogeneous parameters),
	where the shape of $T(N+s-1)$ is $\mu^{N-s+1}$. Indeed, this statement follows from the general Schur dynamics result
	\cite[Theorem~10]{Borodin2010Schur} (where instead of the push-block dynamics one can use the Robinson--Schensted--Knuth one, 
	cf.~\cite{BorodinPetrov2013NN}, \cite{MatveevPetrov2014}). 
	With this identification, the desired claim follows from 
	\Cref{prop:Grothendieck_measures_embedding}.
\end{proof}

\bibliographystyle{alpha}
\bibliography{bib}

\bigskip

\textsc{S. Gavrilova, HSE University (Moscow, Russia) and Massachusetts Institute of Technology (Cambridge, MA, USA)}

E-mail: \texttt{sveta117@mit.edu}

\medskip

\textsc{L. Petrov, University of Virginia (Charlottesville, VA, USA)}

E-mail: \texttt{lenia.petrov@gmail.com}

\newpage

	\begin{figure}[t]
		\centering
		\includegraphics[width=.6\textwidth]{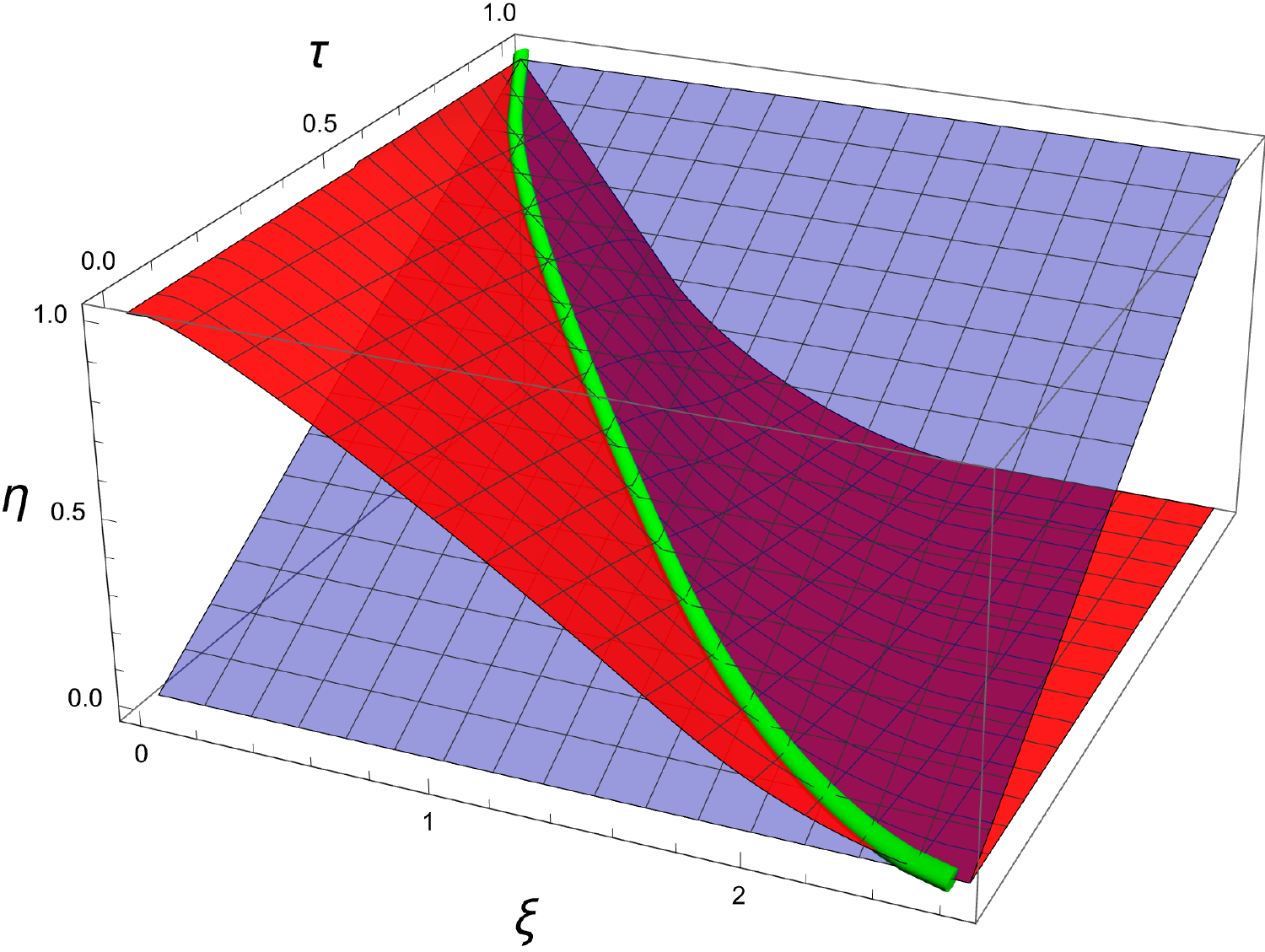}
		\hspace{10pt}
		\raisebox{110pt}{\begin{minipage}{.3\textwidth}
			\includegraphics[width=\textwidth]{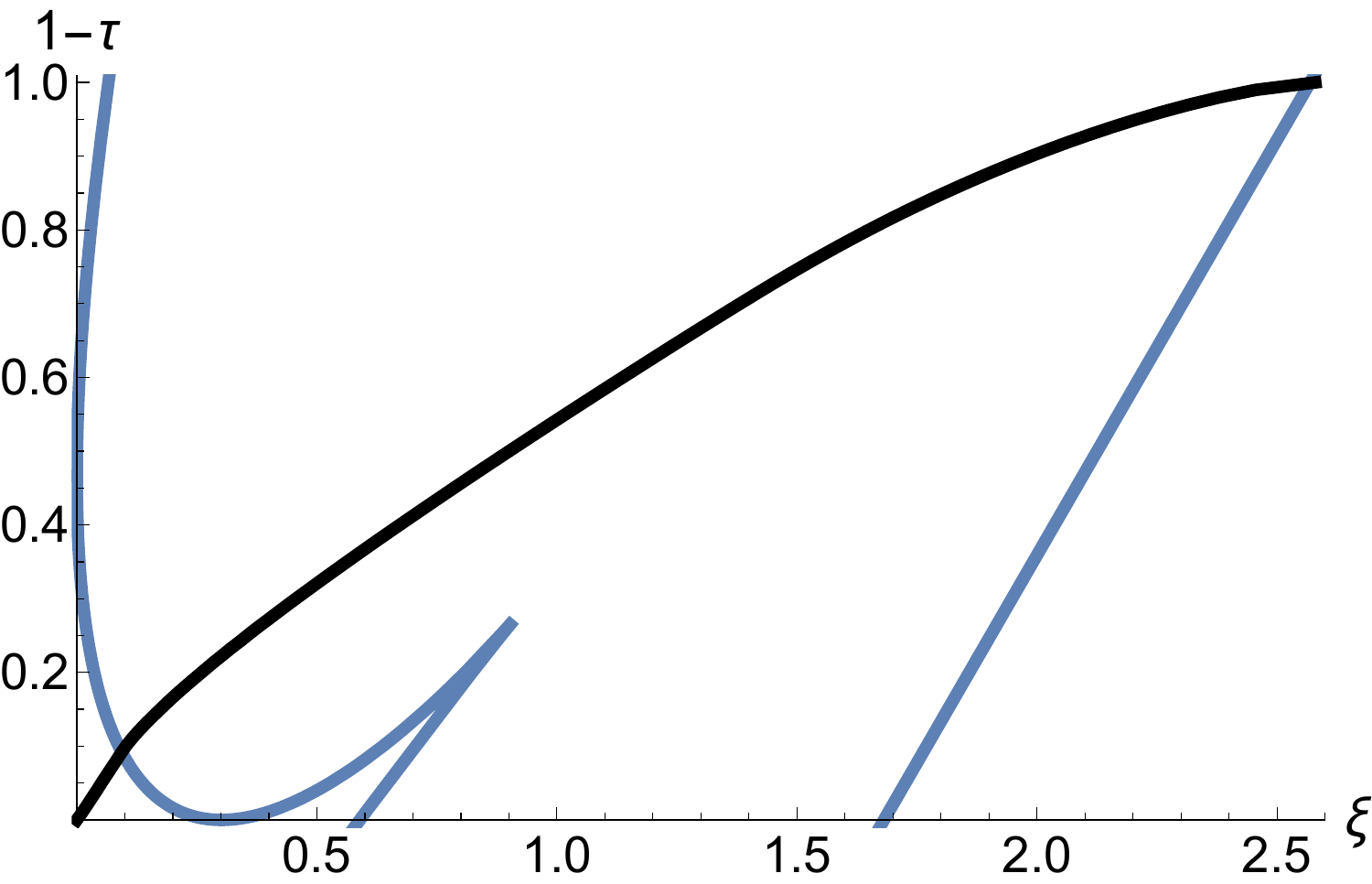}
				\\\vspace{10pt}\\
		\includegraphics[width=\textwidth]{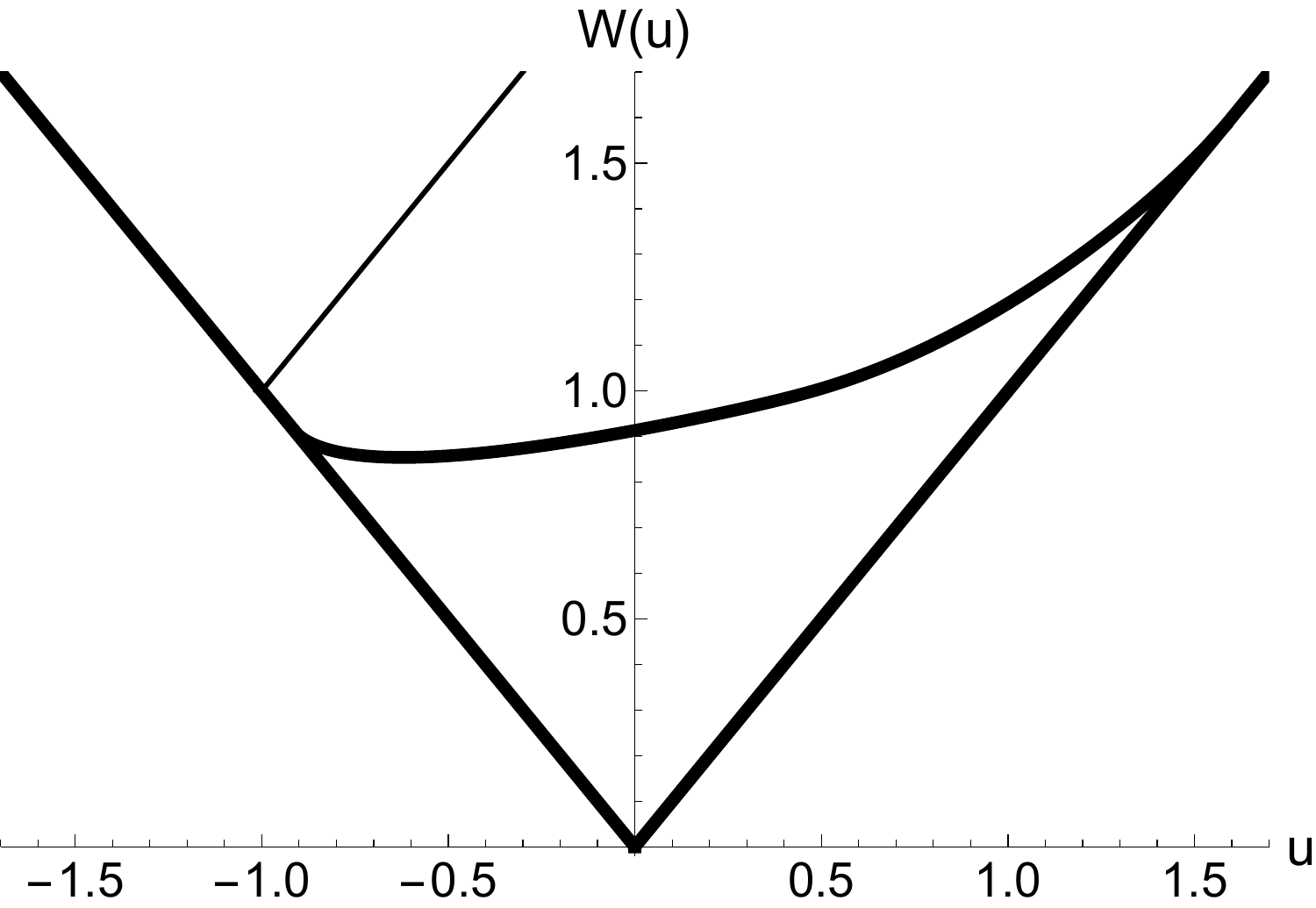}
	\end{minipage}}	
		\caption{Graphs with $x=1/3$, $y=1/5$, $\beta=-6$. See \Cref{sub:simulations_and_particular_cases} for description.}
		\label{fig:sim1}
	\end{figure}

	\begin{figure}[htbp]
		\centering
		\raisebox{75pt}{\begin{minipage}{.41\textwidth}
				\includegraphics[width=\textwidth]{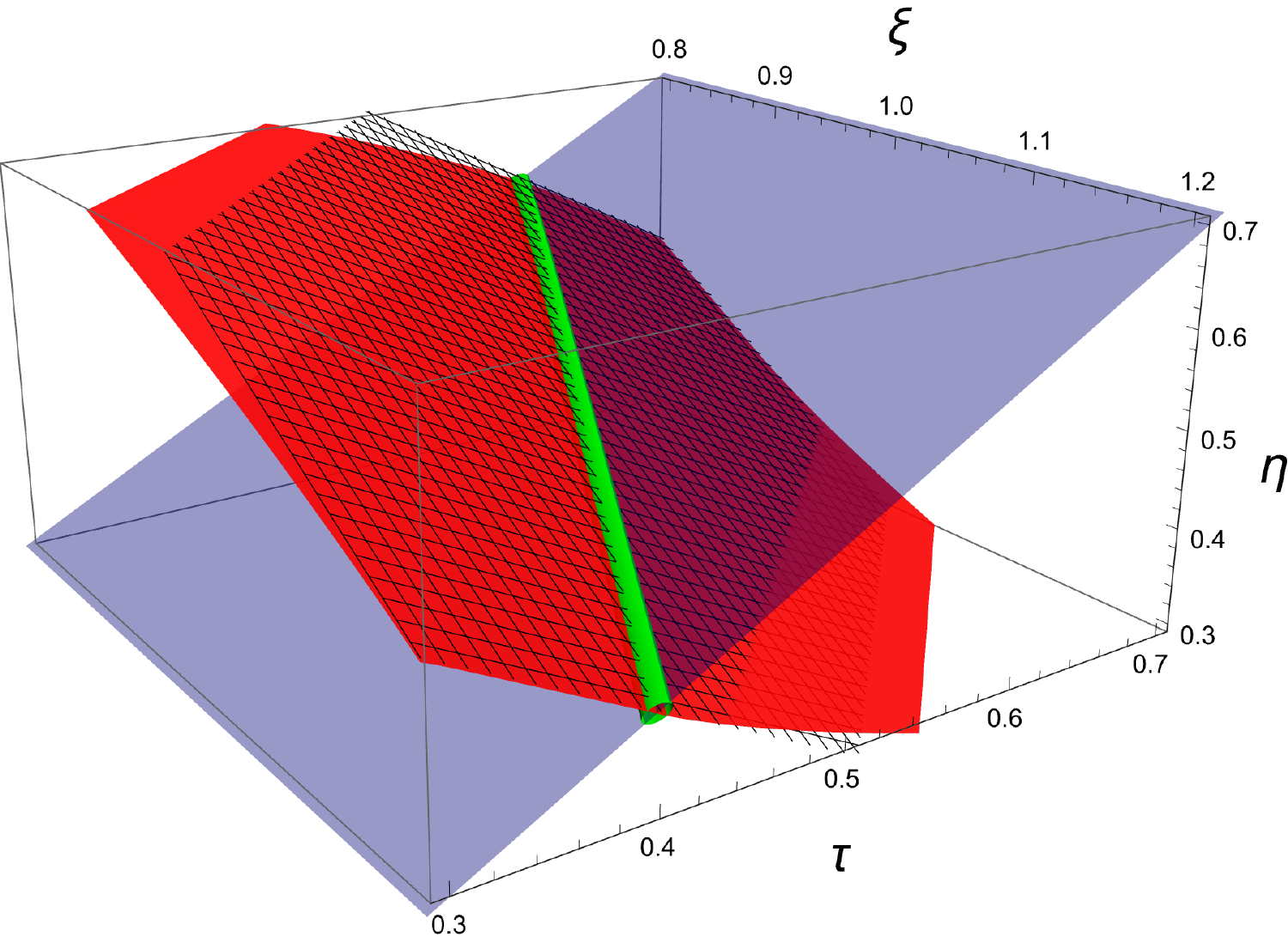}
				\\\vspace{10pt}\\
				\includegraphics[width=\textwidth]{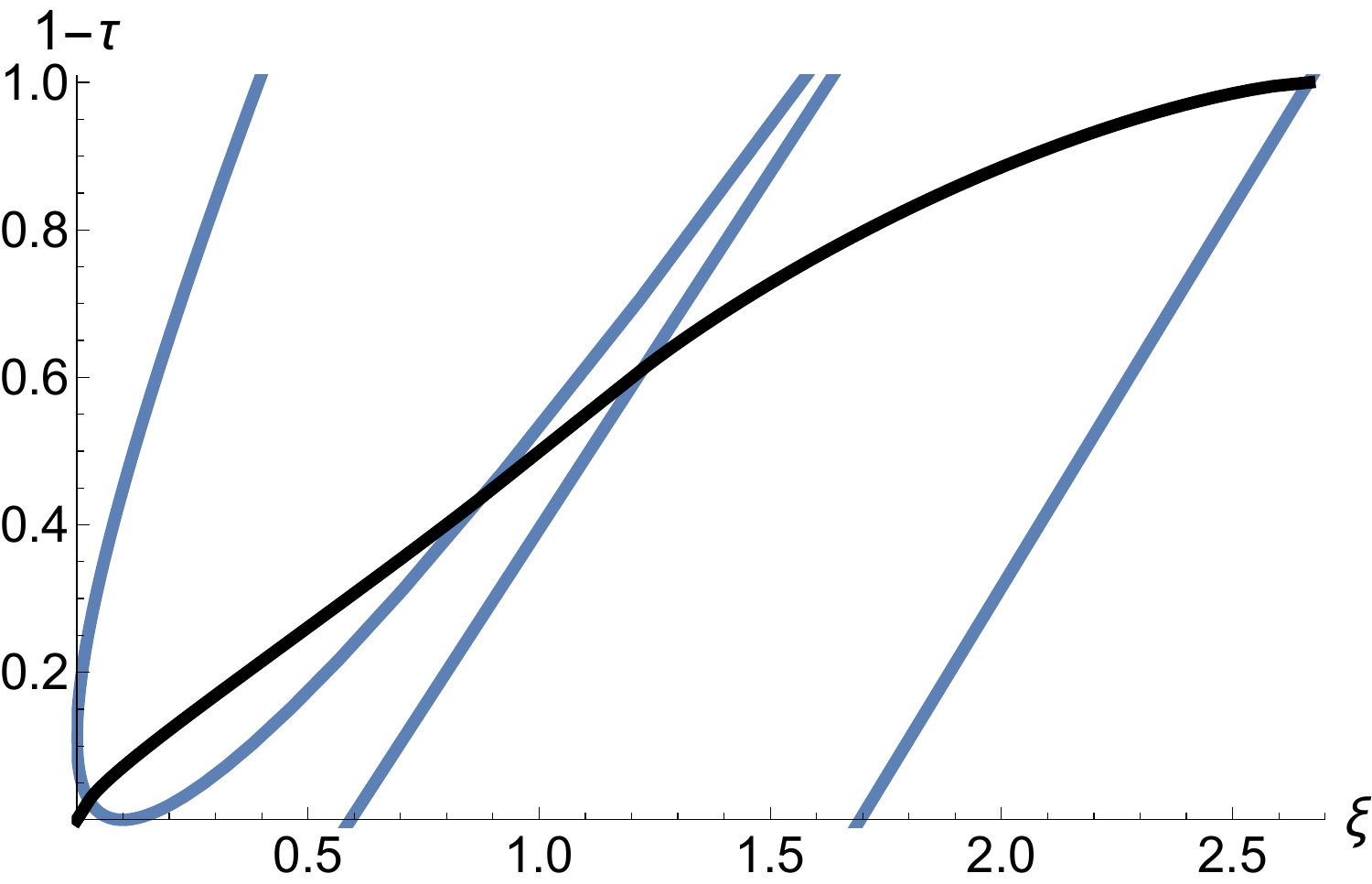}
		\end{minipage}}	
		\hspace{10pt}
		\includegraphics[width=.55\textwidth]{W_graph_large_beta.pdf}
		\caption{Graphs with $x=1/3$, $y=1/5$, $\beta=-25$. See \Cref{sub:simulations_and_particular_cases} for description.}
		\label{fig:sim2}
	\end{figure}

	\begin{figure}[htpb]
		\centering
		\includegraphics[trim={0.7cm 0.6cm .7cm 0cm}, clip, width=.46\textwidth]{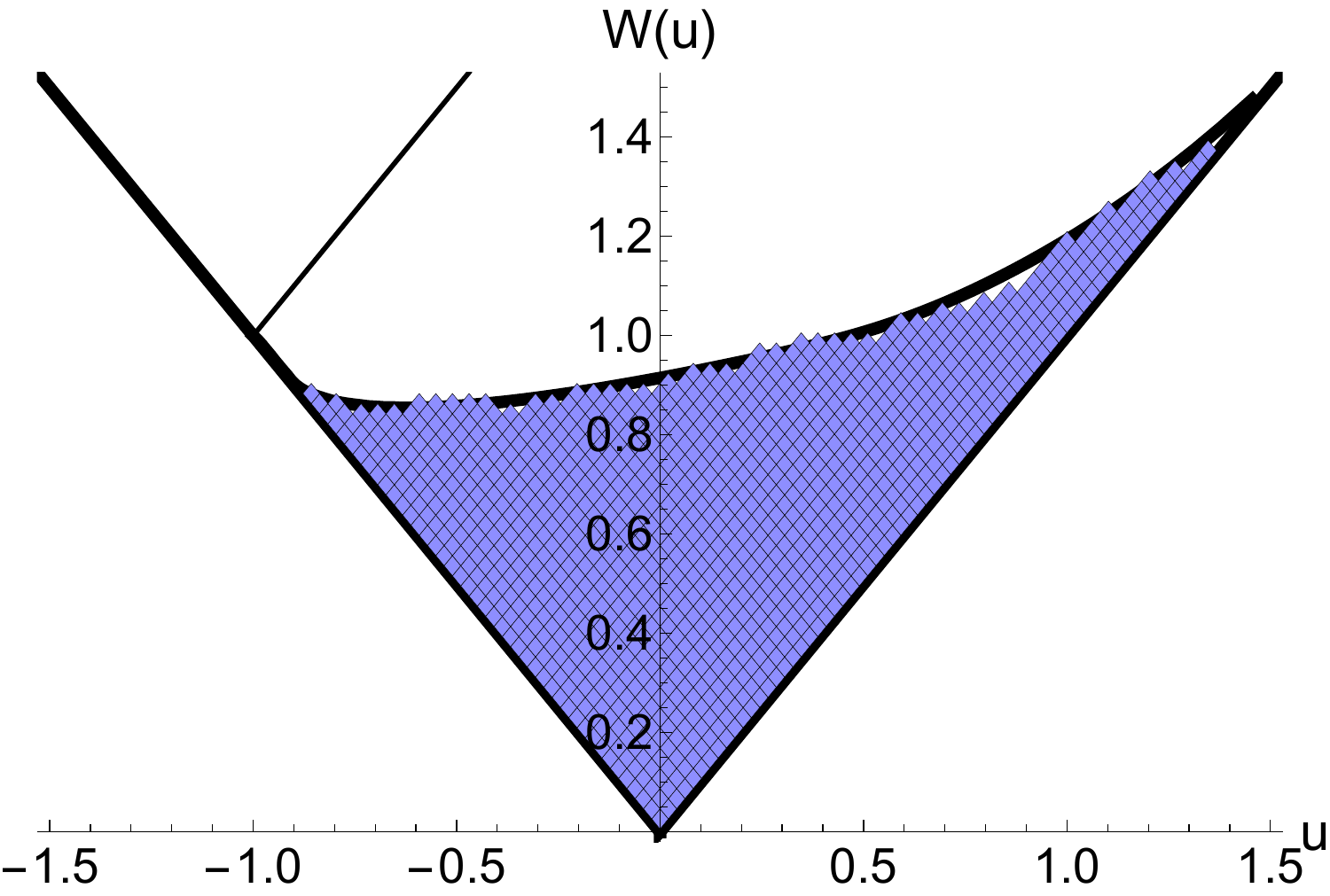}
		\qquad 
		\includegraphics[trim={0.7cm 0.6cm .7cm 0cm}, clip, width=.46\textwidth]{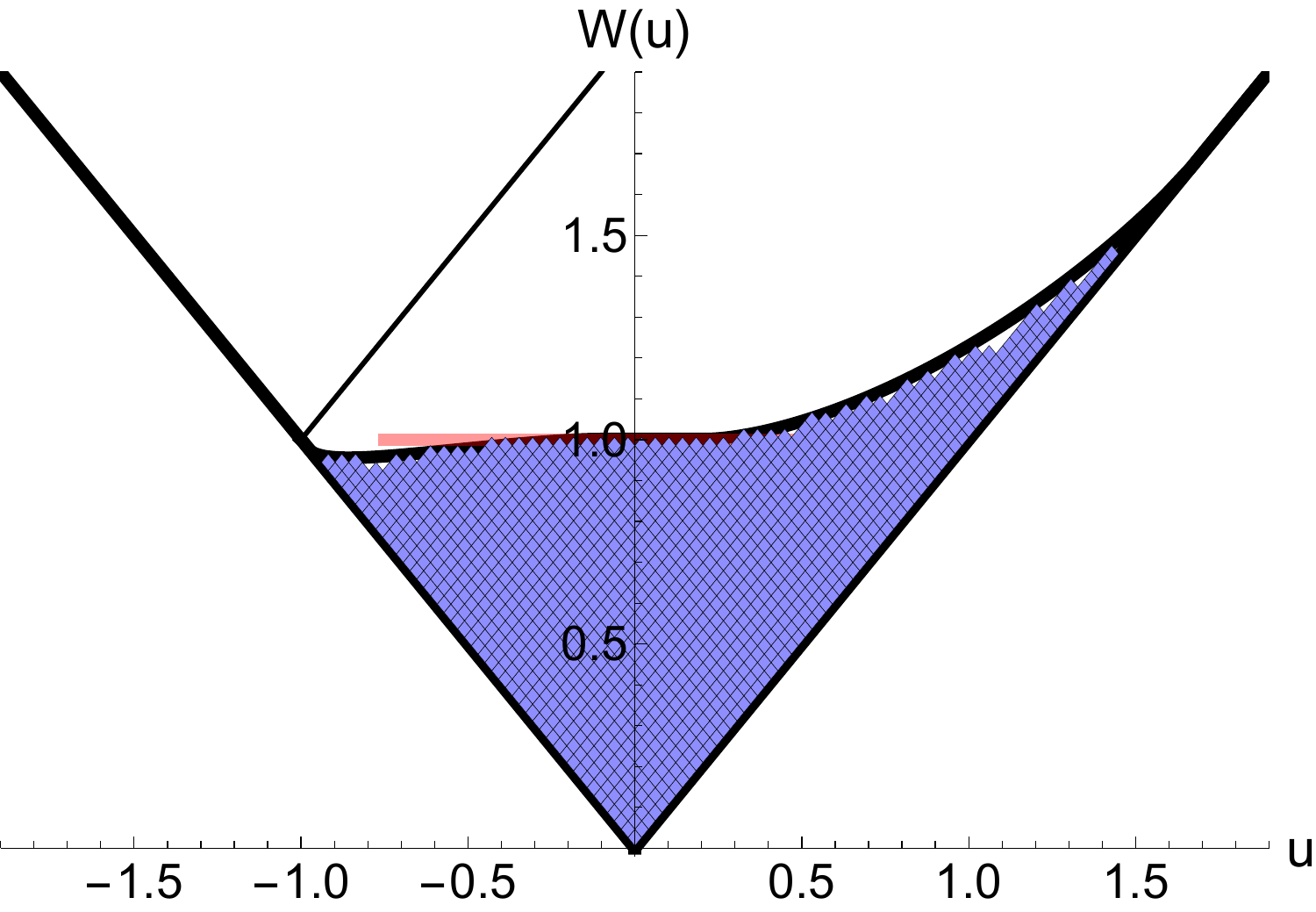}
		\caption{Exact samples of the Grothendieck random partitions with $N=50$ and parameters
			$x=1/3$, $y=1/5$, and $\beta=-6$ (left plot) or $\beta=-25$ (right plot). See \Cref{sub:sampling} 
			for a discussion of the sampling mechanism. We observe that the samples follow the limit shapes from 
		\Cref{fig:sim1,fig:sim2}, as it should be. In particular, notice the staircase frozen facet on the right plot.}
		\label{fig:samples}
	\end{figure}

	\begin{figure}[htpb]
		\centering
		\includegraphics[width=.7\textwidth]{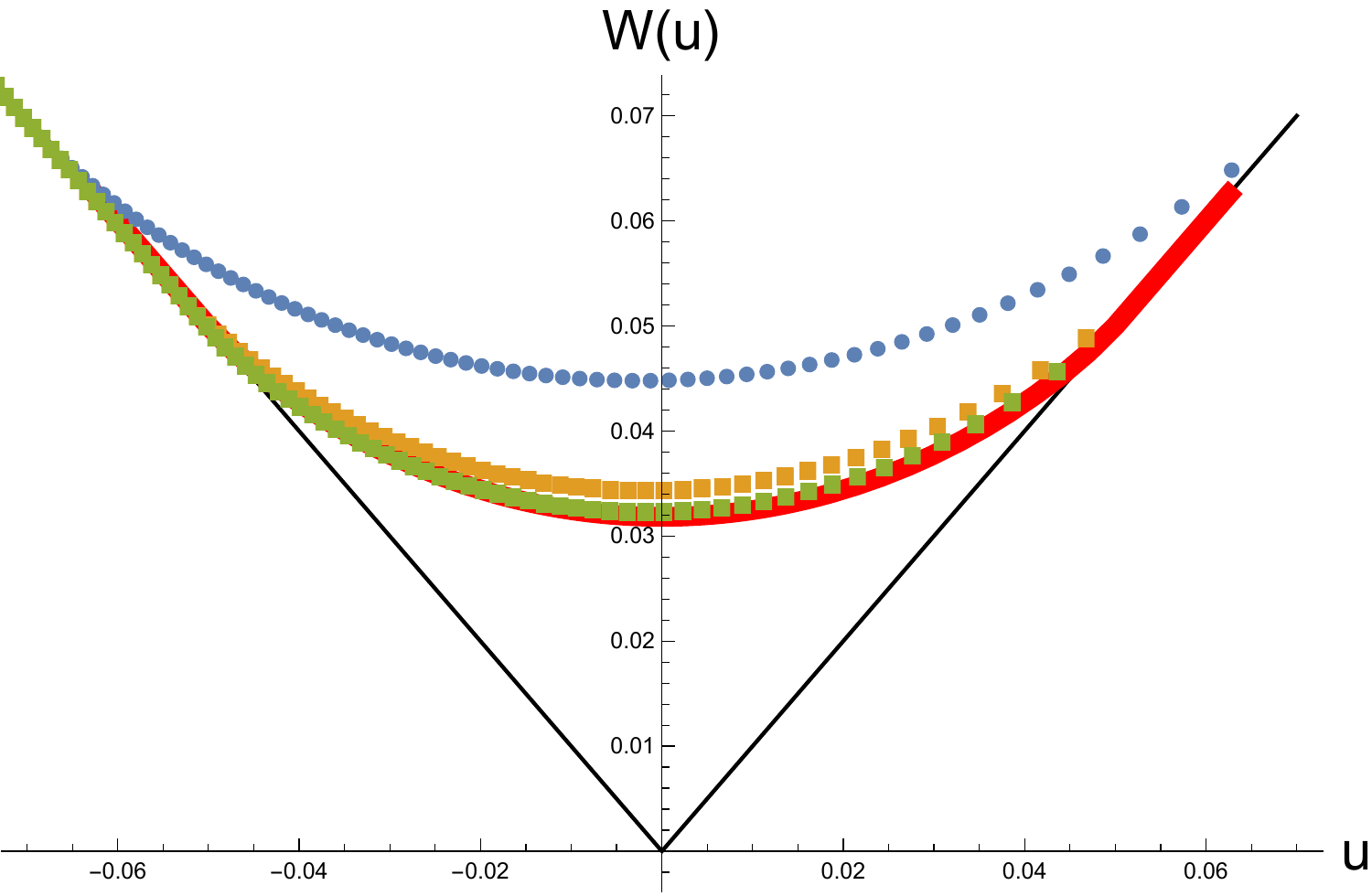}
		\caption{Graphs of $\mathfrak{W}(u)$ with $x=y=1/40$ and $\beta=-1/40$ (round),
			$\beta=-1/250$ (yellow squares), and $\beta=-1/1600$ (green squares).
		We also added the scaled VKLS curve $\sqrt{xy}\,\Omega(u/\sqrt{xy})$. 
		See \Cref{sub:simulations_and_particular_cases} 
		for more detail.}
		\label{fig:sim3}
	\end{figure}

	\begin{figure}[htpb]
		\centering
		\includegraphics[width=.7\textwidth]{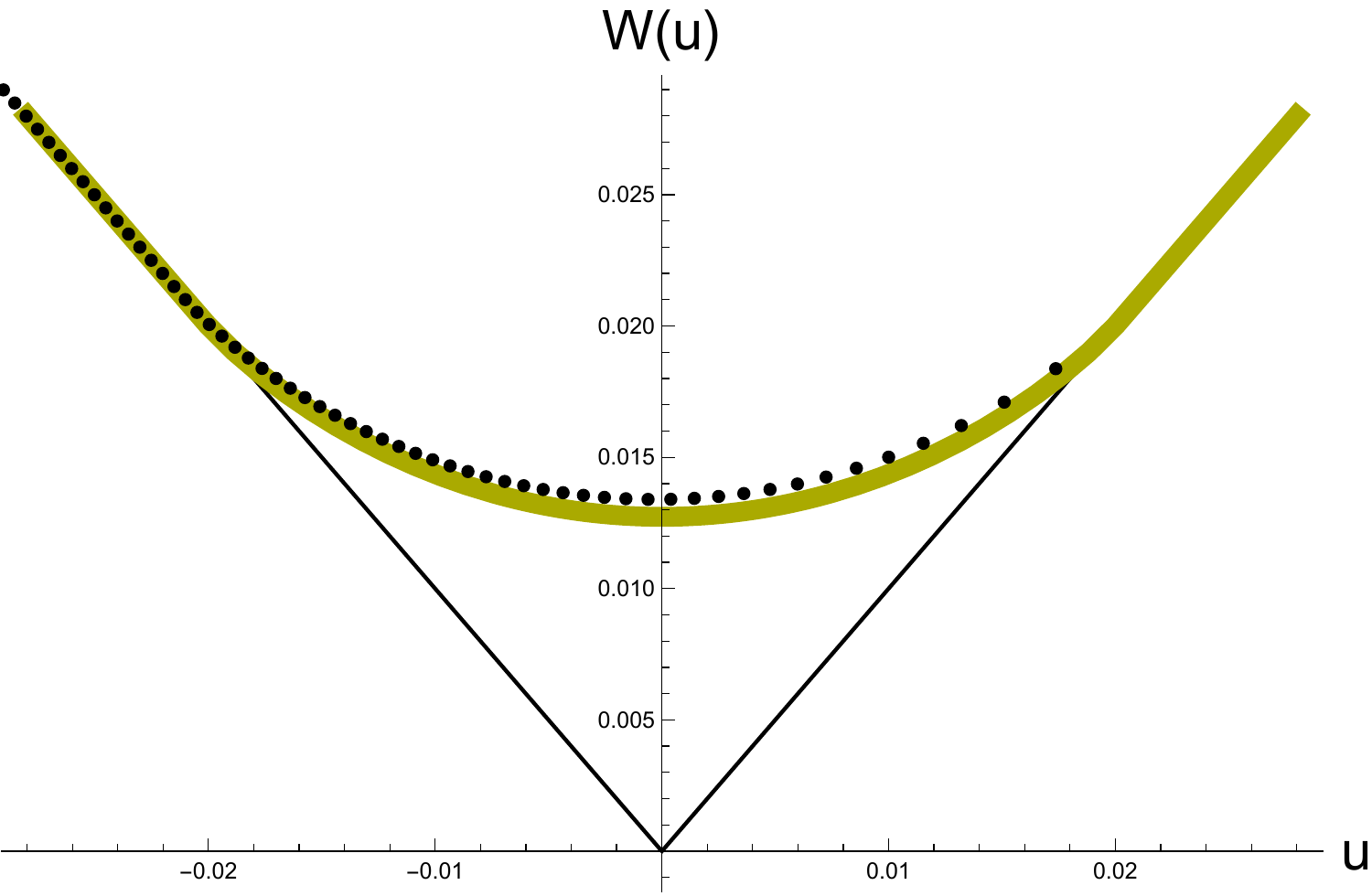}
		\caption{Graph of $\mathfrak{W}(u)$ with $x=y=1/100$ and $\beta=-1/1000$ and the 
		scaled VKLS curve $\sqrt{xy}\,\Omega(u/\sqrt{xy})$.
		See \Cref{sub:simulations_and_particular_cases} 
		for more detail.}
		\label{fig:sim4}
	\end{figure}

	\begin{figure}[htpb]
		\centering
		\raisebox{80pt}{\begin{minipage}{.3\textwidth}
			\includegraphics[width=\textwidth]{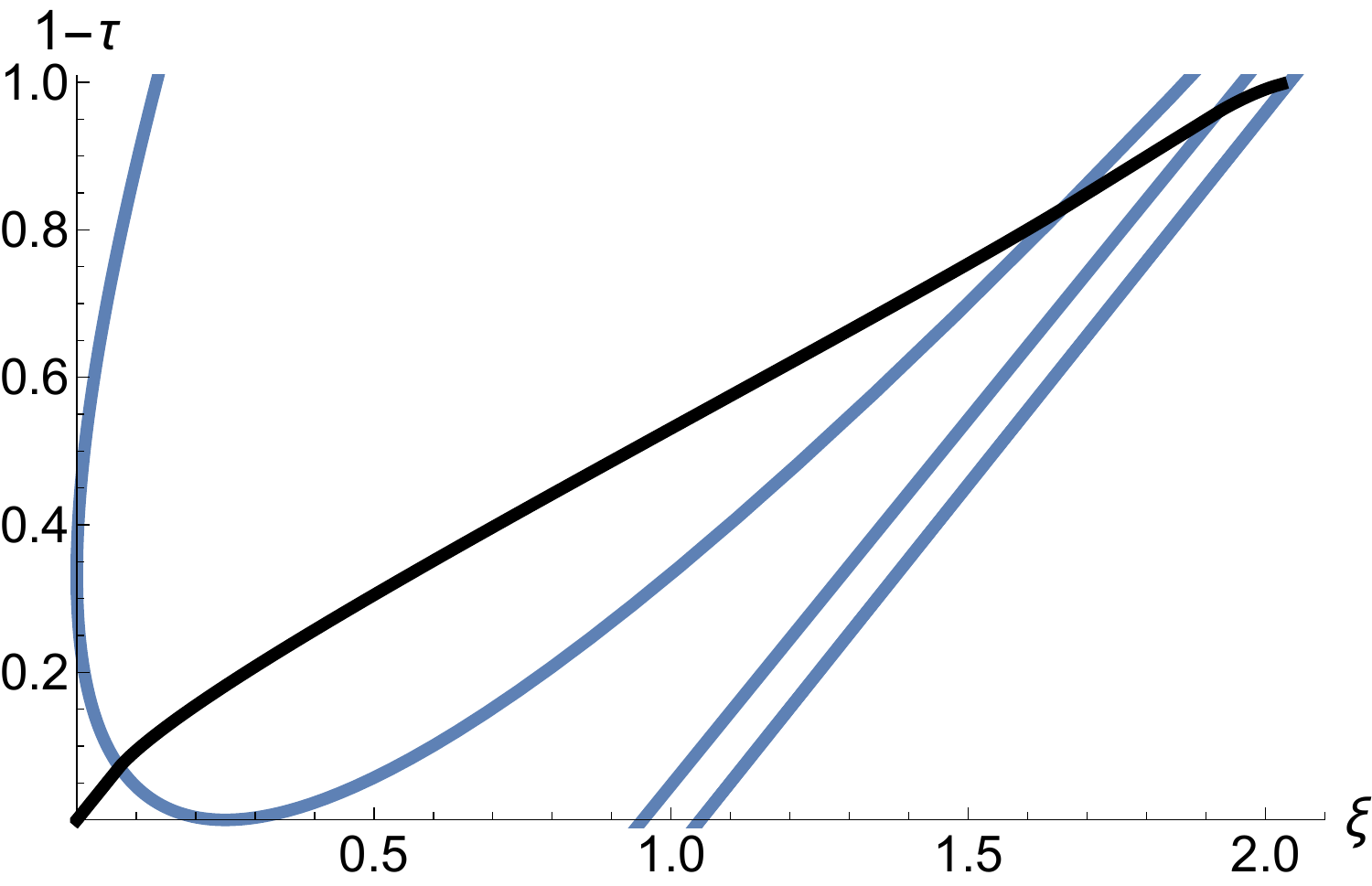}
				\\\vspace{10pt}\\
		\includegraphics[width=\textwidth]{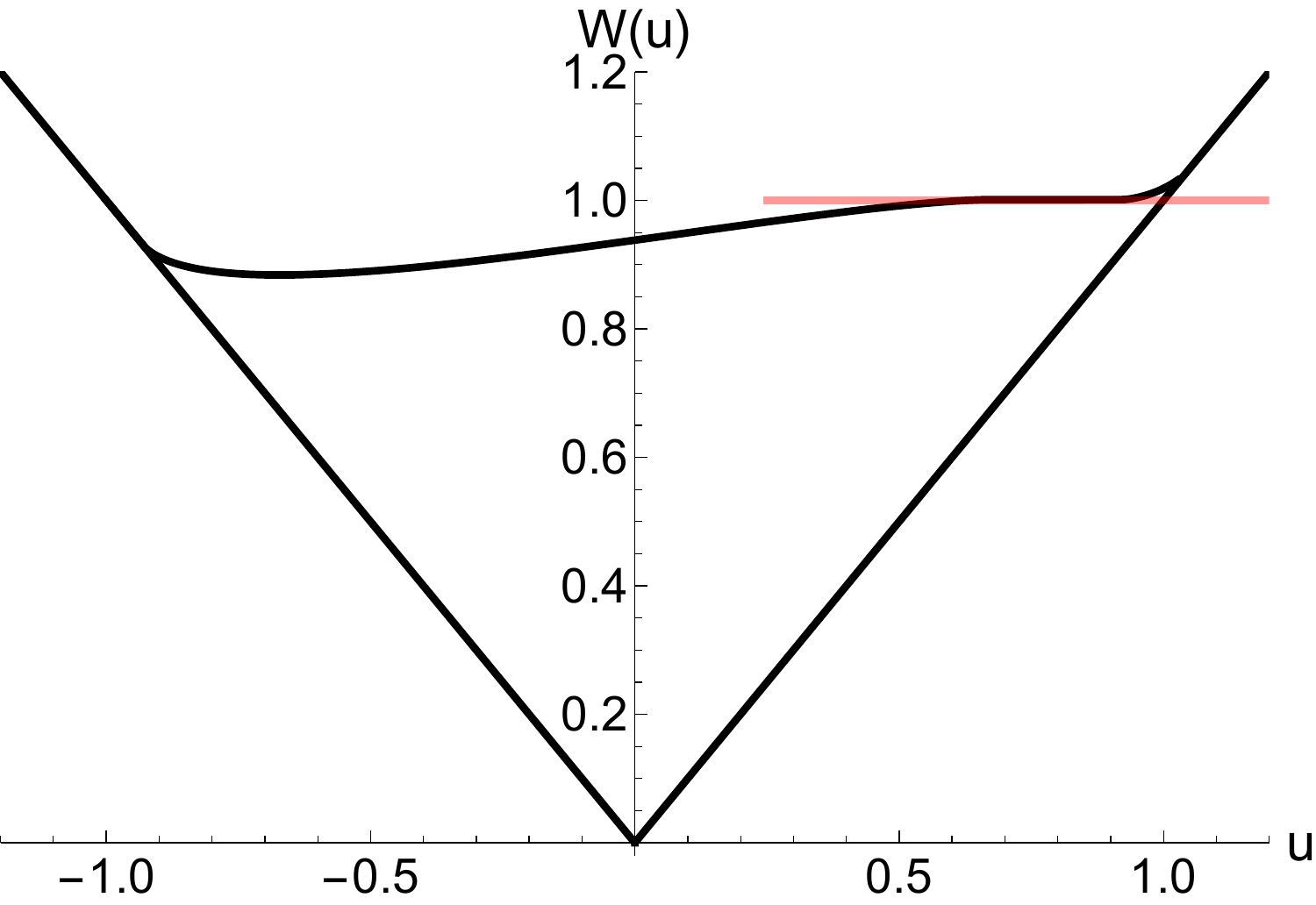}
	\end{minipage}}	
		\hspace{10pt}
		\includegraphics[width=.6\textwidth]{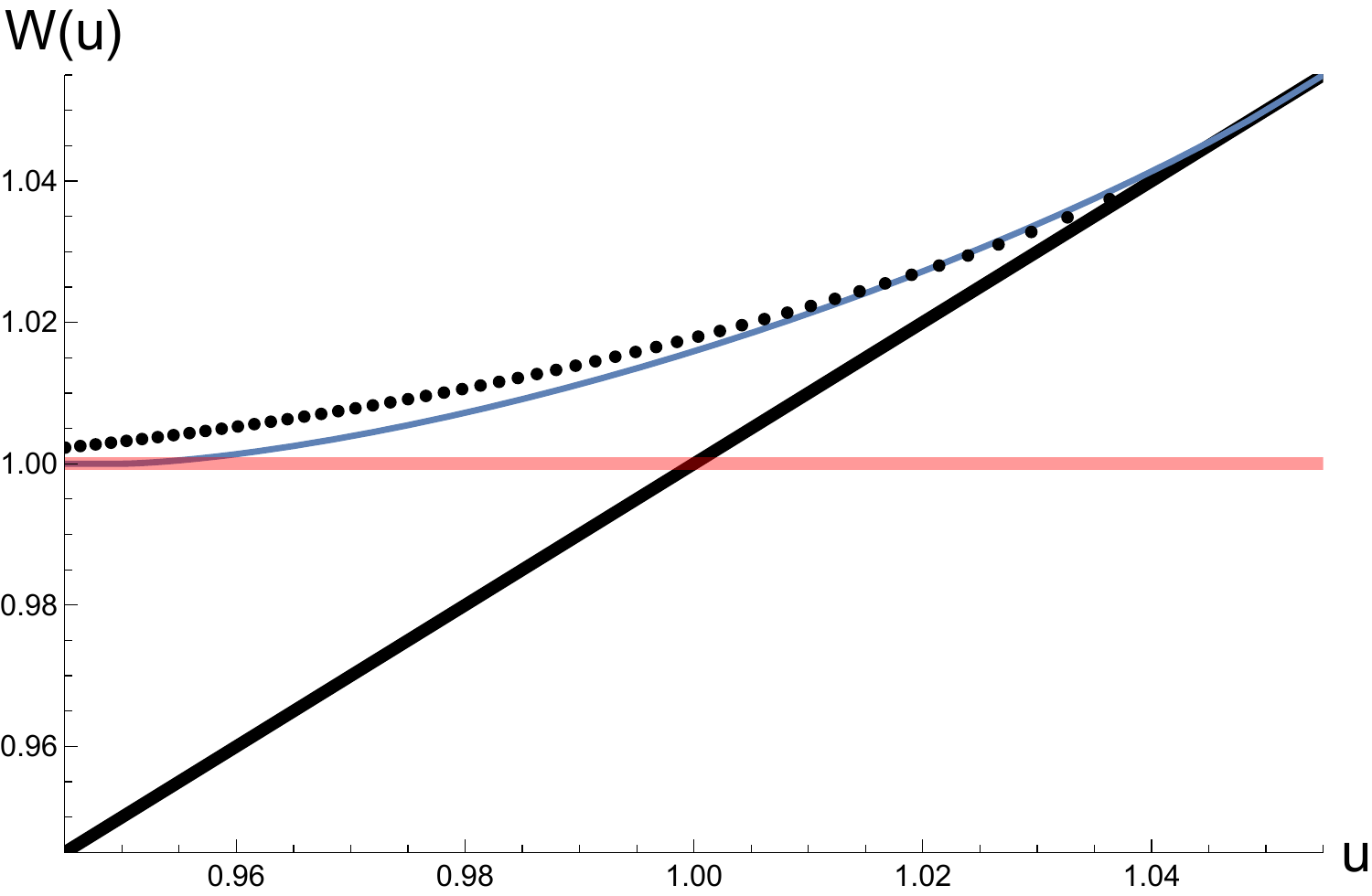}
		\caption{Graphs with $x=y=1/40$ and $\beta=-120$. On the right, we added the shifted and scaled VKLS shape 
		(given by \eqref{eq:Omega_for_shPl} with $\Omega_{(K)}$ replaced
		by $\Omega$ \eqref{eq:VKLS}). The Grothendieck limit 
		shape apparently is not close to the shifted and scaled VKLS shape.
		See \Cref{sub:simulations_and_particular_cases} 
		for more detail.}
		\label{fig:sim5}
	\end{figure}
	
	\begin{figure}[htpb]
		\centering
		\raisebox{80pt}{\begin{minipage}{.33\textwidth}
			\includegraphics[width=\textwidth]{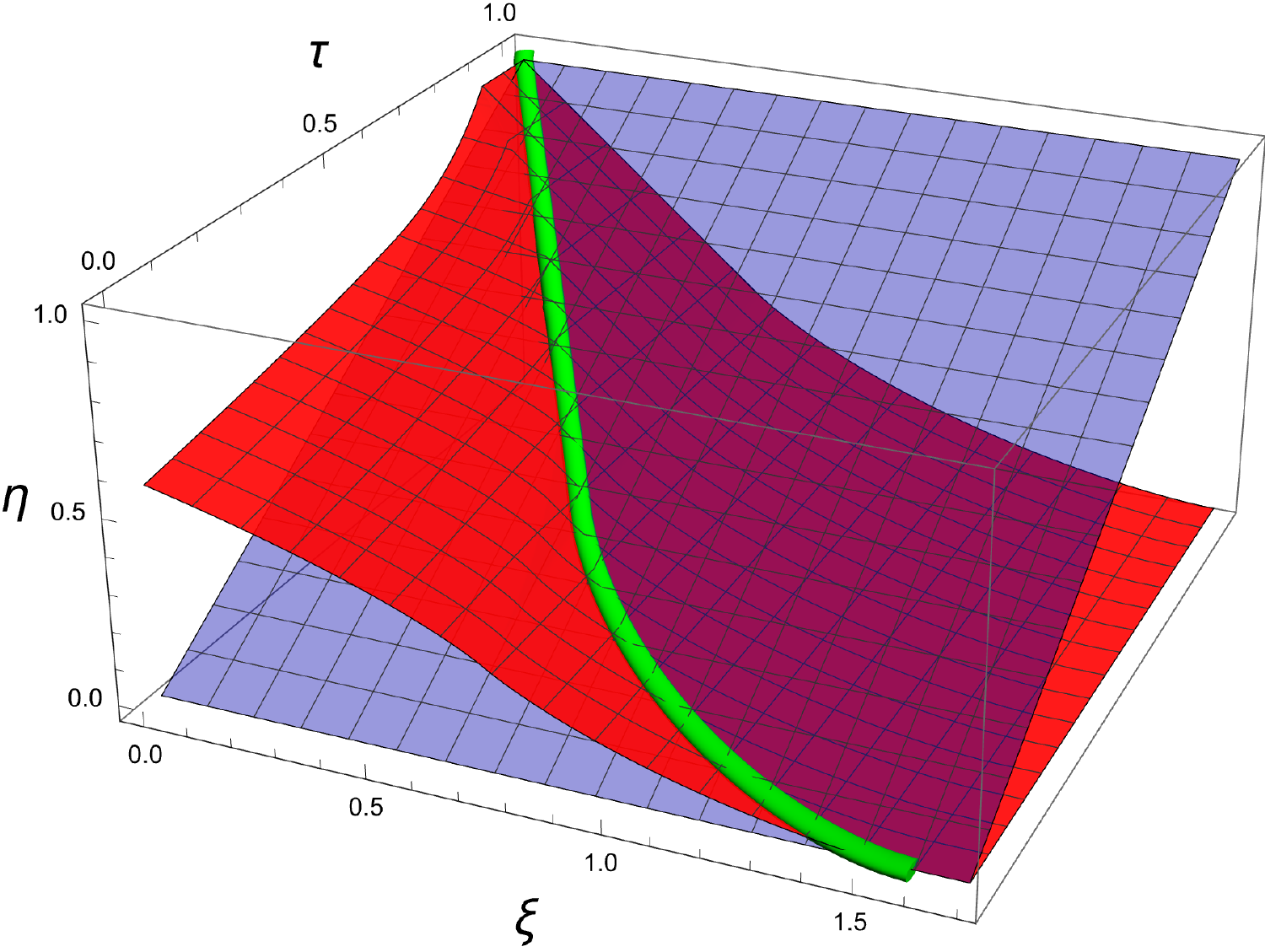}
				\\\vspace{10pt}\\
		\includegraphics[width=\textwidth]{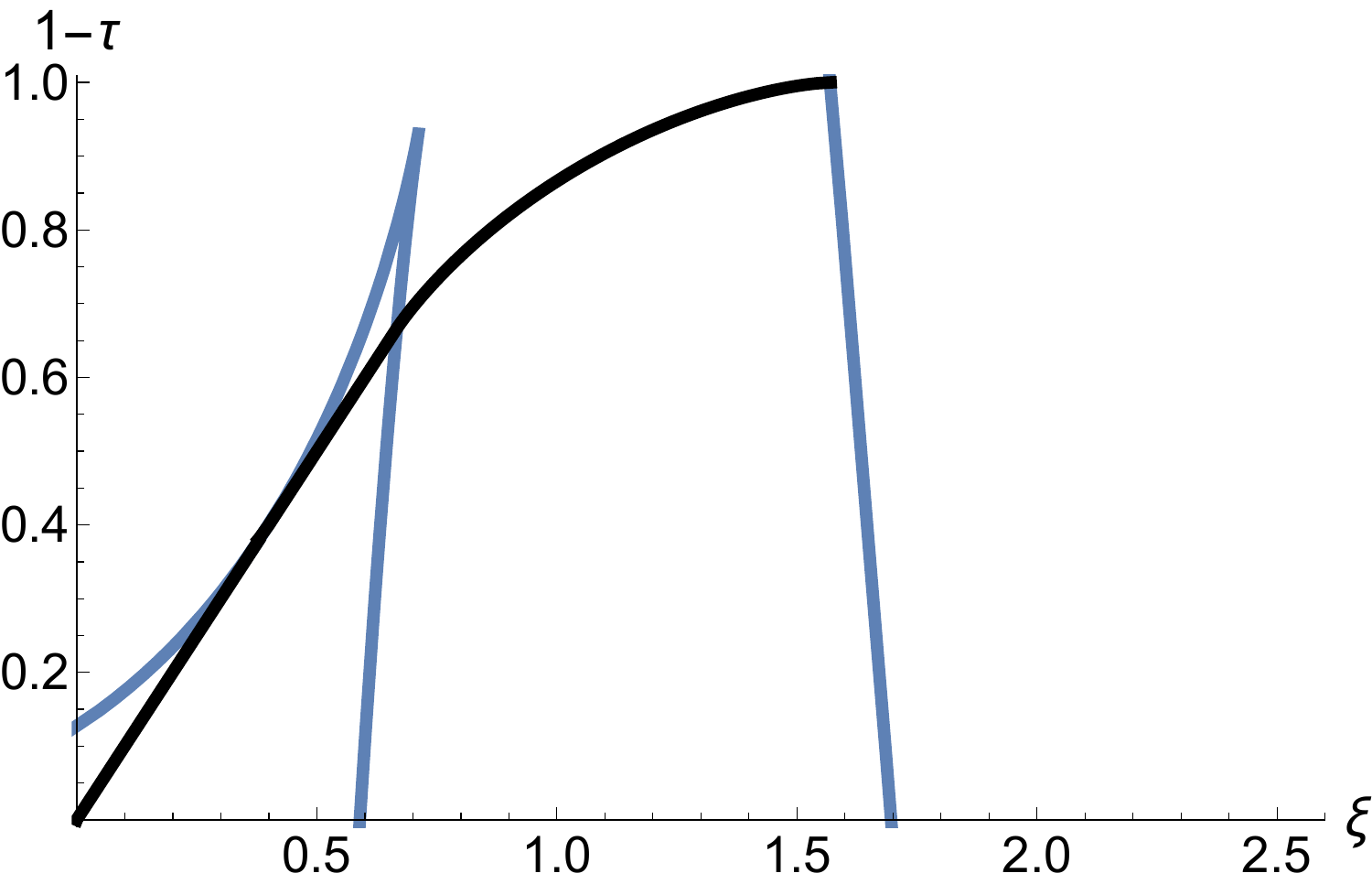}
	\end{minipage}}	
		\hspace{10pt}
		\includegraphics[width=.6\textwidth]{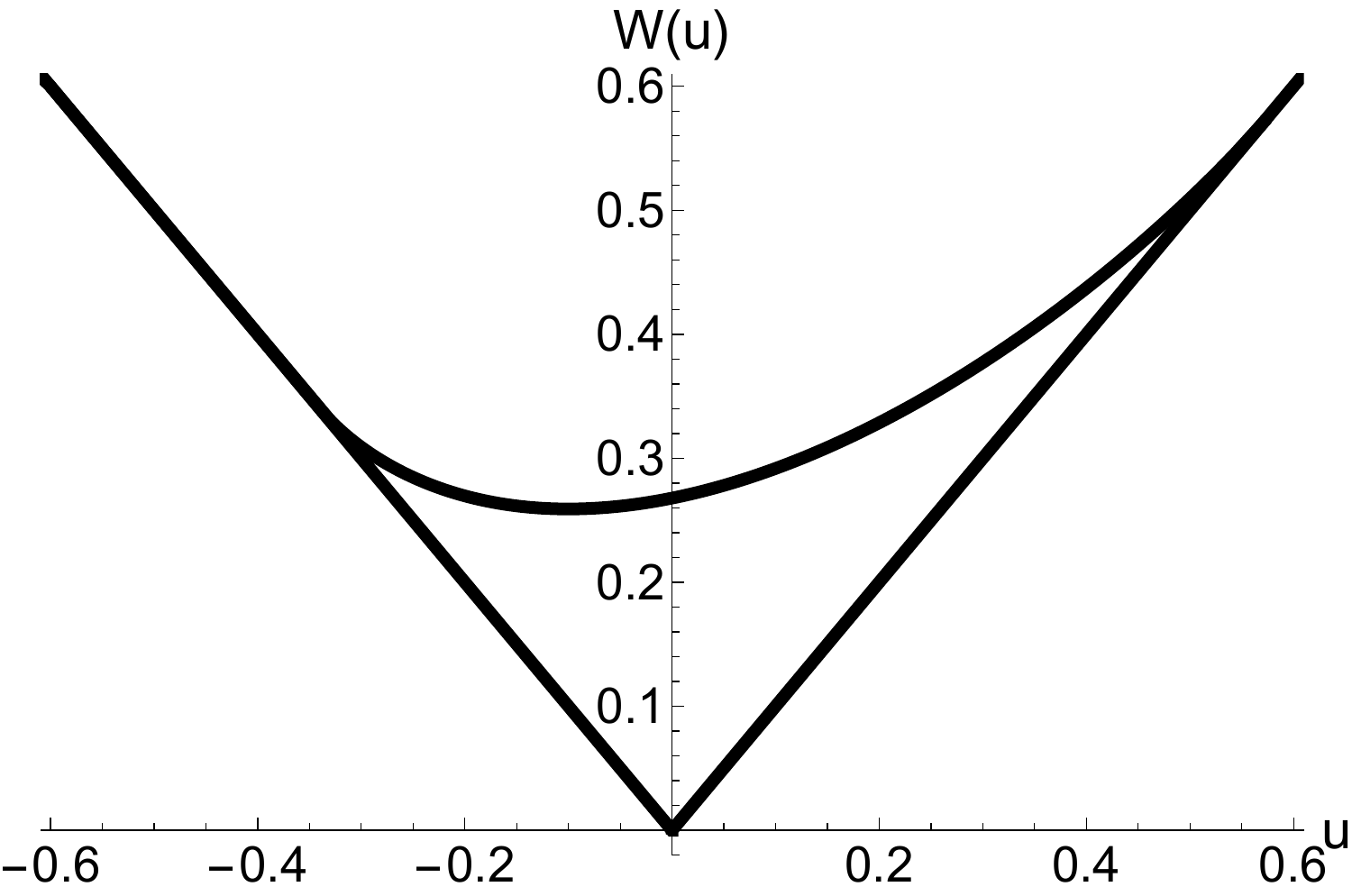}
		\caption{Graphs with $x=1/3$, $y=1/5$, and $\beta=1/12$. 
			The two graphs on the right do not correspond to a
			nonnegative probability measure on two-dimensional point configurations.
			However, we conjecture that the Grothendieck measures
			still converge to the limit shape $\mathfrak{W}(u)$ on the right.
			See \Cref{sub:simulations_and_particular_cases} and 
			\Cref{conj:positive_beta}
			in particular
			for details.}
		\label{fig:sim6}
	\end{figure}

\end{document}